\documentclass[11pt]{amsart}
 
\theoremstyle{plain}

\usepackage{amsmath,amssymb,amsfonts}
\usepackage{enumerate,comment}
\usepackage{amscd, amssymb, latexsym, amsmath, amscd,xcolor}
\usepackage[all]{xy}
\usepackage{pb-diagram}
\usepackage{picture}
\usepackage{graphicx,tikz}
\usepackage{multicol,lipsum}

\newtheorem{thm}[equation]{Theorem}
\newtheorem{prop}[equation]{Proposition}
\newtheorem{cor}[equation]{Corollary}

\newtheorem{lemma}[equation]{Lemma}

\newtheorem{definition}[equation]{Definition}

\numberwithin{equation}{section}

\newcommand{\Q}{\mathbb Q}
\newcommand{\Z}{\mathbb Z}
\newcommand{\R}{\mathbb R}
\newcommand{\C}{\mathbb C}

\newcommand{\F}{\mathbb F}

\newcommand{\A}{\mathbb A}

\def\Aut{{\rm Aut}}

\def\SL{{\rm SL}}

\def\GSp{{\rm GSp}}

\def\PGSp{{\rm PGSp}}

\def\PGSO{{\rm PGSO}}

\def\Spin{{\rm Spin}}
\def\Sp{{\rm Sp}}
\def\Spin{{\rm Spin}}

\def\GL{{\rm GL}}
\def\PGL{{\rm PGL}}
\def\GSO{{\rm GSO}}

\def\SO{{\rm SO}}

\def\Sp{{\rm Sp}}

\usepackage[all]{xy}

\def\A{{\mathbb A}}

\def\R{{\mathbb R}}

\def\Z{{\mathbb Z}}

\def\C{{\bf C}}

\def\C{{\mathbb C}}

\begin{document}

\title[Adjoint Lifting for $\GL_3$]{Triality and
Adjoint lifting for $\GL_3$} 
 \author{Wee Teck Gan}
 
 \address{Department of Mathematics, National University of Singapore, 10 Lower Kent Ridge Road
Singapore 119076} 
\email{matgwt@nus.edu.sg}

 
 \dedicatory{
In memory of \\
Benedict H. Gross \\
 (June 22, 1950- Dec 19, 2025)
}

\maketitle

\section{\bf Introduction} 
This paper and its prequel \cite{ChG} aim  to highlight some implications of the principle of triality in the Langlands program. 
The principle of triality refers to the fact that the (split) simply-connected $\Spin_8$ and the adjoint $\PGSO_8$  possess an  order 3 outer automorphism $\theta$.  In \cite{ChG}, by composing known cases of Langlands functorial lifting with the triality automorphism,
 we showed  that one can obtain new and dramatically different instances of functorial lifting.    In the present  paper, we shall consider the implication of  triality  in the theory of twisted endoscopy. 
\vskip 5pt

After recalling some necessary background on the triality automorphism in \S \ref{S:triality},
 we determine the twisted endoscopic groups associated to the twisted space $(\PGSO_8, \theta)$ in  \S \ref{S:endoscopic}.
It turns out that the twisted endoscopic groups of $(\PGSO_8, \theta)$ are  one of the following: the exceptional group ${\rm G}_2$, $\SL_3$ and $\SO_4$.
 For the case of $H = \SL_3$, for example, the associated map of dual groups is the unique lift of the adjoint representation
\[  \tilde{{\rm Ad}}:  H^{\vee} =  \PGL_3(\C) \longrightarrow \Spin_8(\C).  \]
 Indeed, $\tilde{\rm Ad}$ fits into the folllowing composite of morphisms of dual groups:
  \begin{equation}  \label{E:seq}
   \begin{CD}
 \GL_3(\C) @>>> \PGL_3(\C) @>\tilde{{\rm Ad}}>> \Spin_8(\C) @>>> \SO_8(\C)  @>{\rm std}>> \GL_8(\C), \end{CD} \end{equation}
 which is nothing but the adjoint representation of $\GL_3$ on its semisimple Lie algebra $\mathfrak{sl}_3$. 
 \vskip 5pt
 
 Using the stable twisted trace formula established by  Moeglin and Waldspurger, one would thus expect to show the 
 (weak) twisted endoscopic lifting from  $\SL_3$ to $\PGSO_8$.  On the other hand,  the functorial lifting associated to the first and third arrows are simply the pullbacks of automorphic forms from $\GL_3$ to $\SL_3$ and from $\PGSO_8$ to $\SO_8$ respectively, while  the functorial lifting associated to the standard representation of $\SO_8$ is known  by \cite{A, CKPSS}. Hence, by combining these known functorial liftings with the twisted endoscopic lifting induced by $\tilde{{\rm Ad}}$, one would be able to show the (weak) adjoint lifting from $\GL_3$ to $\GL_8$.   Without analyzing the stable trace formula too deeply, we carry out this strategy in this paper and  show  the anticipated adjoint lifting under a mild local  condition:
\vskip 5pt

\begin{thm}  \label{T:intro2}
Assume that $\pi$ is a cuspidal representation of $\GL_3$ over a number field, such that for some finite place $v_0$, the local component $\pi_{v_0}$ is a discrete series representation.  Then the weak adjoint lifting ${\rm Ad}(\pi)$ exists as an automorphic representation of $\GL_8$. 
\end{thm}
\vskip 10pt

 We recall that the adjoint lifting for $\GL_2$ is a classical result of Gelbart and Jacquet \cite{GJ}.
In fact, the local condition in the theorem can be removed by a more extensive investigation of the stable twisted trace formula, one which involves the consideration of the other twisted endoscopic groups of $(\PGSO_8, \theta)$  and hence will not be fully pursued  here.
 We shall however sketch  in \S \ref{SS:general} the argument for  the general case, punting some of the technical details  to 
 an  ongoing joint work  with S.W. Shin and Y. Takanashi,  the goal of which is to obtain an endoscopic classification for the group $G_2$ via the twisted endoscopy for $(\PGSO_8, \theta)$.   For the rest of the introduction, we shall proceed assuming that  Theorem \ref{T:intro2} is available without the local condition. 
\vskip 5pt

As a second result of this paper, we analyze the possible shapes of the isobaric decomposition of the automorphic representation ${\rm Ad}(\pi)$ supplied by Theorem \ref{T:intro2}; this is often referred to as a {\it cuspidality criterion} for the adjoint lifting. The results, contained in Theorems \ref{T:AI1}, \ref{T:AI2}, \ref{T:selfdual} and \ref{T:cuspidality}, can be summarized as follows:
\vskip 5pt

\begin{thm} \label{T:intro3}
Let $\pi$ be a unitary cuspidal automorphic representation of $\GL_3$ and assume that ${\rm Ad}(\pi)$ has been constructed as  the composite of the Langlands functorial lifting given in (\ref{E:seq}), as discussed above. 
Then ${\rm Ad}(\pi)$ is an isobaric sum of cuspidal representations with unitary central characters. Moreover, ${\rm Ad}(\pi)$  is  self-dual of orthogonal type with trivial central character.   More precisely, one has:
\vskip 5pt
\noindent (i) ${\rm Ad}(\pi)$ contains a cuspidal $\GL_1$-summand if and only if $\pi$ is obtained by automorphic induction from a cyclic cubic field extension. In this case, then isobaric deocmposition of ${\rm Ad}(\pi)$ is of type $1+1 + 3 +3$  or $1^8$, depending on whether $\pi$ can be obtained as an automorphic induction from one or four cyclic cubic extensions.
\vskip 5pt

\noindent (ii) ${\rm Ad}(\pi)$ contains a cuspidal $\GL_2$-summand if and only if $\pi$ is obtained by automoprhic induction from a non-cyclic cubic field extension. In this case, the isobaric decomposition of ${\rm Ad}(\pi)$ is of type $2+6$, $2+3+3$ or $2^4$, depending on whether $\pi$ can be obtained as an automorphic induction from one (for the first two cases) or four (for the last case) non-cyclic cubic extensions.
\vskip 5pt

\noindent (iii) Assume that $\pi$ is not an automorphic induction from any cubic field extension, so that ${\rm Ad}(\pi)$ does not contain any cuspidal $\GL_1$- or $\GL_2$-summands.
Then ${\rm Ad}(\pi)$ contains a cuspidal $\GL_3$-summand if and only if $\pi$ is a twist (by an automorphic character) of a self-dual cuspidal representation. In this case, the isobaric   decomposition of ${\rm Ad}(\pi)$ is of type $3+5$.
\vskip 5pt

\noindent (iv) Assume that we are not in the setting of (i), (ii) or (iii). 
Then one of the following holds:
\vskip 5pt
\begin{itemize}
\item[(a)] ${\rm Ad}(\pi)$ is the sum of two self-dual cuspidal $\GL_4$-summands of orthogonal type and nontrivial central character; this occurs precisely when there is a quadratic extension $K/F$ such that the base change $BC_{K/F}(\pi)$ can be obtained as an automorphic induction, in which case it is an automorphic induction from 4 different noncyclic cubic extensions of $K$.
\vskip 5pt

\item[(b)] ${\rm Ad}(\pi)$ is cuspidal; this occurs when the base change of $\pi$ to any quadratic extension $K/F$ is not an automorphic induction.
\end{itemize}
\vskip 5pt

Except in (i), ${\rm Ad}(\pi)$ is a discrete orthogonal L-parameter, i.e. it is the functorial lift of  a globally generic cuspidal automorphic representation of the split group $\SO_8$.  Moreover, 
\[ {\rm ord}_{s=1} L^S( s, {\rm Ad}(\pi), \wedge^3) = {\rm ord}_{s=1} L^S(s, {\rm Ad}(\pi), {\rm Sym}^2) < 0,  \]
so that the order of poles at $s=1$ of $L^S(s, {\rm Ad}(\pi), \wedge^3)$  (which is an L-function of Langlands-Shahidi type) is equal to the number of summands in the isobaric decomposition of ${\rm Ad}(\pi)$.

\end{thm}
\vskip 5pt

\noindent The proof of this theorem relies crucially on the functorial lifting results  of Gelbart-Jacquet \cite{GJ}, Arthur-Clozel \cite{AC}, 
Ramakrishnan \cite{Ra1}, Kim-Shahidi \cite{KS1} and  Kim \cite{Ki}, as well as the cuspidality criteria for some of these liftings due to Ramakrishnan-Wang \cite{RW} and Asgari-Raghuram \cite{AR}.
 
\vskip 5pt

It is of course of great interest to determine precisely the fibers and image of the adjoint lifting. In particular, the question of the fibers of the adjoint lifting is closely related to the question of multiplicities and rigidity for the cuspidal spectrum of $\SL_3$. 
It has been shown by Blasius \cite{B}  that the multiplicity-one property and the rigidity of L-packets both fail for $\SL_n$ for $n \geq 3$. 
In other words, there are cuspidal representations of $\SL_n$ ($n \geq 3$) with multiplicity $>1$, and there are inequivalent cuspidal L-packets which are nearly equivalent.  
Lapid  \cite{L1} extended this line of study with a more quantitative analysis, especially for the endoscopic L-packets of $\SL_n$ obtained by endoscopic transfer ($=$ automorphic induction) from elliptic tori. In particular, he showed that in the context of $\SL_3$, the so-called {\it global multiplicity} of a cuspidal endoscopic L-packet is either $1$ or $2$ (see \cite[\S2, Pg 162]{L1} for the definition and \cite[Theorem 4]{L1} for the result). 
 For example, in the context of Theorem \ref{T:intro3}(i), the degenerate case where ${\rm Ad}(\pi)$ is the sum of 8 cubic characters coincides precisely with the case where the global multiplicity is $2$. 
\vskip 5pt

Let us conclude this introduction by mentioning two applications of the adjoint lifiting for $\GL_3$. The first is an application to the strong Artin conjecture for certain 3-dimensional Galois representations. These 3-dimensional Galois representations were considered by Lapid \cite[\S 5]{L2} and are analogs of the tetrahedral 2-dimensional Galois representations. Recall that as a consequence of his proof of base change for $\GL_2$ (and the adjoint lifting of $\GL_2$), Langlands showed that  tetrahedral 2-dimensional Galois representations are automorphic, i.e. the strong Artin conjecture holds for them.   In \cite[Prop. 2]{L2}, Lapid showed that  for any prime power $q$, the $q$-dimensional tetrahedral Galois representations (as defined in \cite[\S 5]{L2}) are automorphic, {\em provided that the adjoint lifting  for $\GL_q$ is known}.
Hence, one consequence of Theorem \ref{T:intro2} is:
\vskip 5pt

\begin{thm} \label{T:artin1}
The strong Artin conjecture holds for   3-dimensional tetrahedral Galois representations (as deifned in \cite[\S 5]{L2}).
\end{thm} 
\noindent In \S \ref{S:artin}, we shall recall the precise definition and construction of these 3-dimensional tetrahedral Galois representations. 


\vskip 5pt

As a second application of the adjoint lifting for $\GL_3$, we show the following bound towards the Ramanujan-Petersson conjecture for $\GL_3$:
\vskip 5pt

\begin{thm} \label{T:Rama1}
Let $\pi$ be a cuspidal representation of $\GL_3$ with unitary central character and a discrete series local component over a number field $k$. At each place $v$ where $\pi_v$ is unramified, let $\lambda$ be an eigenvalue of its Satake parameter in $\GL_3(\C)$. Writing $|\lambda| = q_v^a$ for some $a \in \R$ (and with $q_v$ equal to the cardinality of the residue field of $k_v$), one has:
\[ |a| \leq  \frac{1}{4} - \frac{1}{130}. \]
\end{thm}

In fact, in an earlier version of this paper, the author obtained the weaker bound
\[  |a| \leq \frac{21}{65} < \frac{1}{3}. \] 
An observation communicated to us by V. Blomer allows us to strengthen the result as above. In any case, the result  relies on the general bounds of Luo-Rudnick-Sarnak \cite{LRS} and improves upon a result of Blomer-Brumley \cite{BB}, which gives $|a| \leq 5/14$  for $\GL_3$.  The Ramanujan-Petersson conjecture, of course,  asserts that $a =0$. A recent paper of LY Yang \cite{Y}, extending earlier work of Ramakrishnan \cite{Ra3} over $\Q$, shows that this holds for infinitely many places $v$ of $k$.    
\vskip 10pt

  \noindent{\bf Acknowledgments:}   We thank Gaetan Chenevier (who should have been a coauthor of this paper),  Erez Lapid, Dipendra Prasad, Sug Woo Shin  and Yugo Takanashi for helpful discussions at various stages of this work and comments on earlier drafts of this paper. In addition, we thank Henry Kim for sending us a copy of the paper \cite{KS3}, Freydoon Shahidi for discussion on the fibers of the Rankin-Selberg  lifting from $\GL_2 \times \GL_3$ to $\GL_6$, David Hansen for his catalytic question which leads to the content of Section \ref{S:Ramanujan} and the first version  of Theorem \ref{T:Rama1}, and Valentin Blomer and Farrell Brumley for their suggestions leading to the final version of Theorem \ref{T:Rama1}. 
  The author is  supported by a Tan Chin Tuan Centennial Professorship at NUS.

\vskip 10pt

\section{\bf Triality}  \label{S:triality}

 In this section, we recall some necessary background on the triality automorphism of $\Spin_8$ and $\PGSO_8$.

 \vskip 5pt

 \subsection{\bf Symmetric composition algebras and $\Spin_8$}
 Let us first recall from \cite[\S 2]{ChG} a construction of $\Spin_8$ with its triality automorphism  via the theory of symmetric composition algebras. For the definition and properties of symmetric composition algebras, the reader may consult \cite{KMRT}. 
 By \cite[Theorem 34.37]{KMRT},  there are essentially two  isomorphism classes of   split symmetric composition algebra over a field $F$.  These may be described as follows:
 \vskip 5pt
 
 \begin{itemize}
 \item  Start with the split octonion algebra $(\mathbb{O}, \cdot, N)$ over $F$, where $N$ is its norm form permitting composition: $N(x \cdot y) = N(x) \cdot N(y)$ for all $x,y \in \mathbb{O}$.
 Then one  defines a new multiplication on $\mathbb{O}$  by $x \ast y = \bar{x} \cdot \bar{y}$. The algebra $(\mathbb{O}, \ast, N)$ is  the para-octonion algebra and one still has $N( x\ast y) = N(x) \cdot N(y)$ for all $x,y \in \mathbb{O}$.  Moreover,  $\Aut(\mathbb{O}, \cdot) = \Aut(\mathbb{O}, \ast) \simeq G_2^{\mathbb{O}}$. 
 \vskip 5pt

\item Assume now that  $F$  contains a primitive 3rd root of unity $\omega$ (so ${\rm char}(F) \ne 3$).
The other split symmetric composition algebra of dimension $8$ is realised on the space $\mathbb{M}$ of trace zero $3 \times 3$ matrices with multiplication given (miraculously) by \cite[Pg. 472, (34.18)]{KMRT}
\[   x \ast y = \omega xy + (1-\omega)yx  - \frac{1}{3} T(yx), \quad \text{  for $x,y \in \mathbb{M}$,}  \]
where $T$ is the trace of a $3 \times 3$ matrix. If the characteristic polynomial of  $x \in \mathbb{M}$ is
\[ P_x(t) =  t^3  + S(x) t  - \det(x),   \]
then the quadratic form on $\mathbb{M}$ which permits composition is  $\frac{1}{3} \cdot S(x)$ \cite[Prop. 34.19]{KMRT}. The automorphism group  of $(\mathbb{M}, \ast), S$ is the group $\PGL_3 \subset \SO(\mathbb{M}, S)$ acting on $\mathbb{M}$ by conjugation. In other words, the action of $\PGL_3$ on $\mathbb{M}$ is isomorphic to the adjoint representation. 
 \end{itemize}
 Thus, over a general field $F$ with  ${\rm char}(F) \ne 3$, the $F$-forms of the above  symmetric composition algebras are classified by $H^1(F, G_2)$ and 
 $H^1(F, \PGL_3)$ respectively.

\vskip 5pt
Given   a symmetric composition algebra $(V, \ast, Q)$  of dimension $8$ over $F$, let us define
\[  \Spin(V,\ast, Q) = \{ (g_1, g_2, g_3) \in \SO(V,Q)^3:  g_1(x \ast y) = g_2(x) \ast g_3(y) \text{ for $x,y \in V$} \}. \]
Then this group is known to be isomorphic to $\Spin_8$ over $\overline{F}$. 
Moreover, the cyclic permutation of the three coordinates in $\SO(V,Q)^3$ preserves the subgroup  $\Spin(V, \ast, Q)$ and defines an order $3$ automorphism 
\[  \theta_V:  \Spin(V, \ast, Q) \longrightarrow \Spin(V, \ast, Q). \]
The fixed group of $\theta_V$ is thus the subgroup $\Aut(V,\ast, Q) \subset \SO(V,Q)$.  
Because the automorphism $\theta$ necessarily preserves the center $Z_{\Spin(V,\ast, Q)}$, it descends to an automorphism of the adjoint group $\PGSO(V,Q)$, still denoted by $\theta_V$.
 \vskip 5pt
 
 \vskip 10pt

 \subsection{\bf Splitting of $\Aut$-sequence.}
Recall that one has a short exact sequence
\begin{equation} \label{E:Aut}
  \begin{CD}
1 @>>> \PGSO_8 @>>>  \Aut(\Spin_8)  @>>>   {\rm Out}(\Spin_8) = S_3 @>>> 1. \end{CD} \end{equation}
The data of a symmetric composition algebra $(V,\ast, Q)$ thus gives a splitting of this short exact sequence over the order 3 subgroup $A_3 \subset S_3$. Since we could apply the above discussion to the two symmetric composition algebra $(\mathbb{O}, \ast, N)$ and $(\mathbb{M}, \ast, S)$ (when $F = \overline{F}$ say), we obtain two splittings of this short exact sequence over $A_3$ whose fixed subgroups are ${\rm G}_2^\mathbb{O}$ and $\PGL_3$ respectively. According to \cite[Pg. 487, Cor. 35.10 and Prop. 36.14]{KMRT}, these are the only two possible splittings (up to inner conjugation) over $\overline{F}$.  
\vskip 5pt

In particular, the embedding
\[  \PGL_3 \simeq \Aut(\mathbb{M}, \ast, S) \hookrightarrow \Spin(\mathbb{M}, \ast, S)  \]
is invariant under the action of $\theta_{\mathbb{M}}$  and is the map denoted by $\tilde{\rm Ad}$ in the introduction.
\vskip 5pt


\subsection{\bf Representations}
 By construction, the group $\Spin(V, \ast, Q)$ is equipped with  three projections  
\[  \rho_j:  \Spin(V,\ast, Q) \longrightarrow \SO(V,Q), \]
with ${\rm Ker}(\rho_j) \simeq \mu_2 \subset Z_{\Spin(V, \ast, Q)}$.  Thus one has 3 inequivalent $8$-dimensional irreducible orthogonal representations of $\Spin(V,\ast, Q)$ which are permuted by $\theta$.
Likewise,  the projection from the three $\SO(V,Q)$'s  gives 3 maps
\[  f_j:  \SO(V,Q) \longrightarrow \PGSO(V,Q), \]
which are permuted by $\theta$.
\vskip 5pt

The following lemma will only be used  much  later in the paper, but this seems a convenient place to mention it:
\vskip 5pt

\begin{lemma}. \label{L:rep}
Let
\[  \rho_1, \, \rho_2, \, \rho_3: \Spin_8(\C) \longrightarrow \SO_8(\C) \subset  \GL_8(\C) \]
be the three inequivalent 8-dimensional irreducible representations of $\Spin_8(\C)$.  
Then
\[  \rho_1 \otimes \rho_2 \simeq \rho_3 \oplus \wedge^3 \rho_3. \]
\end{lemma}
\vskip 5pt

Now consider, for each $j$, the composite
\[ \begin{CD}
  \GL_3(\C) @>>> \PGL_3(\C) @>\tilde{\rm Ad}>> \Spin_8(\C) @>\rho_j>>  \SO_8(\C) \subset \GL_8(\C). 
 \end{CD} \]
 Since $\tilde{\rm Ad}$ is invariant under the action of $\theta_{\mathbb{M}}$, this composite is independent of $j$ as a representation of $\GL_3(\C)$. As mentioned in the introduction, it is simply the adjoint representation ${\rm Ad}$ of $\GL_3(\C)$. 
 The above lemma thus implies the following corollary:
 \vskip 5pt
 
 \begin{cor}  \label{C:rep}
 One has an isomorphism 
 \[  {\rm Ad} \otimes {\rm Ad} \simeq {\rm Ad} \oplus \wedge^3 {\rm Ad}.\]
 of representations of $\GL_3(\C)$.
\end{cor}

\vskip 5pt

 \subsection{\bf Chevalley-Steinberg \'epinglage.} \label{SS:CS}
For the purpose of this paper, the above geometric algebra approach to $(\PGSO_8,\theta)$ is less useful. Instead, we 
 will give a root theoretic description of triality,  by fixing a Chevalley-Steinberg system of \'epinglage for $\PGSO_8$. 
 More precisely,  we start with a Chevalley model of the split group $\PGSO_8$ over $\mathbb{Z}$ and 
  fix a pair $T \subset B$ consisting of a maximal split torus contained in a Borel subgroup over $\Z$. This gives a system of simple roots 
  \[   \Delta(T,B) = \{\alpha_0, \alpha_1, \alpha_2, \alpha_3\},\]
   with $\alpha_0$ corresponding to the branch vertex in the Dynkin diagram:
  
  
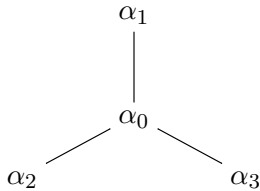
\[
\xymatrix@R=10pt{
&  \alpha_1 \ar@{-}[dd]  &  \\
& & \\
 &   \alpha_0   &  \\
\alpha_2 \ar@{-}[ru] &  &  \alpha_3 \ar@{-}[lu]  \\
} \]
  
    We may fix a pinning $(T, B,  \{x_{\alpha_i} \} )$, with $x_{\alpha_i}$ a basis vector in the root subspace $\mathfrak{u}_{\alpha_i}$.  This then defines a splitting of the exact sequence
   \[  \begin{CD}
   1 @>>> \PGSO_8 @>>> \Aut(\PGSO_8) @>>> {\rm Out}(\PGSO_8) @>>> 1, \end{CD} \]
   with image 
   \[ \Aut(\PGSO_8, T,B, \{ x_{\alpha_i} \}) \subset \Aut(\PGSO_8). \]
   Since ${\rm Out}(\PGSO_8) \simeq S_3$, this gives us a  triality automorphism $\theta$ which   permutes $\{x_{\alpha_1}, x_{\alpha_2}, x_{\alpha_3}\}$ cyclically and fixes $x_{\alpha_0}$, so that
  \[  \PGSO_8^{\theta} \simeq {\rm G}_2. \]
 
 \vskip 5pt
 
   The dual group $\Spin_8(\C)$ thus inherits $T^{\vee} \subset B^{\vee}$ and a pinning $\{ x_{\alpha_i^{\vee}} \}$ permuted by $\theta^{\vee}$ as well. 
   Moreover, 
   \[  \Spin_8(\C)^{\theta^{\vee}} \simeq {\rm G}_2(\C) \]
with maximal torus $(T^{\vee})^{\theta^{\vee}}$.  The action of $(T^{\vee})^{\theta^{\vee}}$ on $\mathfrak{spin}_8 = {\rm Lie}(\Spin_8(\C))$ gives rise to a ${\rm G}_2$ root system $\Phi_{G_2}$, which is simply the restriction of the roots of $\Spin_8(\C)$ ($=$ coroots of $\PGSO_8$) to $(T^{\vee})^{\theta}$.  Writing
\[  \alpha = \text{restriction of $\alpha_i^{\vee}$  ($i=1,2,3$)} \quad \text{and} \quad \beta = \text{restriction of $\alpha_0^{\vee}$}, \]
one has a system of simple roots $\{\alpha, \beta\}$ for $\Phi_{G_2}$. The long root spaces have dimension $1$ whereas the short root spaces have dimension $3$. 
Moreover, the triality automorphism $\theta^{\vee}$ preserves each of these root spaces for $\Phi_{G_2}$.

   \vskip 5pt
   
 This Chevalley-Steinberg system of pinning provides us with a coordinate system with respect to which we will perform some explicit computations in the next section.
  
 \vskip 5pt

\vskip 10pt

 \section{\bf Twisted Endoscopic Groups} \label{S:endoscopic}
 We 
 shall be applying the theory of twisted endoscopy developed by Langlands and Kottwitz-Shelstad \cite{KSh}  to the twisted space defined by the triality automorphism $\theta$ 
 of $\PGSO_8$.  The general theory of the twisted trace formula was developed in \cite{LW} and its stabilization 
  has been established in the  2-volume book  \cite{MW1, MW2} of Moeglin-Waldspurger (modulo the twisted weighted fundamental lemma for general quasi-split groups). To exploit this stable twisted trace formula, we need to work out the elliptic twisted endoscopic data for $(\PGSO_8, \theta)$. This is the purpose of this section. 
 \subsection{\bf Lie algebra computations.} \label{SS:Lie alg}
\vskip 5pt

Our aim now is to determine, up to isomorphism, the elliptic endoscopic data $(H,s, \xi)$ 
of the twisted space $(\PGSO_8, \theta)$ over the number field $k$, 
following \cite[\S 2.1]{KSh}.\footnote{In our setting, the cocycle ${\bf a}$ {\it loc. cit.} is chosen to be $1$, 
which forces the $a'$ in (2.1.4a) to be $1$ as well.} 
As $\PGSO_8$ is adjoint and split over $k$, their description simplifies a little and 
we may assume that these triple $(H,s,\xi)$ are as follows: 
\begin{itemize}
\item the element $s$ is a finite order element of $T^{\vee}$,
\vskip 5pt
\item the ({\it endoscopic}) group $H$ is a split connected reductive group over $k$ equipped with an isomorphism 
between $H^{\vee}$ and the centralizer ${\rm C}_s$ in $\Spin_8(\C)$ of the element $s \cdot \theta^{\vee} \in \Spin_8(\C) \cdot \theta^{\vee}$,
\vskip 5pt
\item ({\it elliptic} condition) $Z(\xi(H^{\vee}))^0 \subset Z_{\Spin_8}$, or equivalently, $H$ is semisimple. 
\end{itemize}
Indeed, it follows from a classical result of Steinberg \cite[Chap. 1, \S 1.1, Pg. 14]{KSh} that ${\rm C}_s$ is connected
for all $s \in T^\vee$. It is harmless to view $\xi$ as an inclusion and write $H^\vee={\rm C}_s$.
 \vskip 5pt 
 To determine the possible $H^{\vee}$, we take a general $s \cdot \theta^{\vee}$, consider its adjoint action on $\mathfrak{spin}_8$ and determine the fixed Lie subalgebra $\mathfrak{spin}_8^{s \cdot \theta^{\vee}}$.  Since we are interested in $s \cdot \theta^{\vee}$ up to conjugacy by $\Spin_8(\C)$, we may consider $s$ up to $\theta$-conjugacy by elements of $T^{\vee}(\C)$. As the group $T^\vee(\C)$ is divisible, it is not hard to see that one can assume that 
 \[  s \in T^{\vee}(\C)^{\theta^{\vee}}. \]
 \vskip 5pt
 
 Since $\Spin_8(\C)$ is simply-connected, $X_*(T^{\vee}) = X^*(T)$ is equal to the root lattice of $\PGSO_8$.
Hence we may write any element of $T^{\vee}(\C)^{\theta^{\vee}}$ as:  
 \begin{equation} \label{E:torus}   s(u,t) =\alpha_0(u) \cdot  \alpha_1(t) \cdot \alpha_2(t) \cdot \alpha_3(t) , \quad \text{with $u, t \in \C^{\times}$.} \end{equation}
 By our discussion in the previous subsection, the action of $s(u,t) \cdot \theta^{\vee}$ on $\mathfrak{spin}_8$ stabilizes $\mathfrak{t}^{\vee} = {\rm Lie}(T^{\vee})$ and 
 the root subspaces for the ${\rm G}_2$-root system.
  Hence,
  \[  \mathfrak{spin}_8^{s(u,t) \cdot \theta^{\vee}}
= (\mathfrak{t}^{\vee})^{\theta^{\vee}} \oplus \bigoplus_{\gamma \in \Phi_{G_2}}  \mathfrak{u}_{\gamma}^{s(u,t) \cdot \theta^{\vee}}. \]
 In particular, to determine this Lie subalgebra fixed by $s(u,t) \cdot \theta^{\vee}$, we need to figure out its fixed space in each root subspace $\mathfrak{u}_{\gamma}$ for $\gamma \in \Phi_{G_2}$. This is a linear algebra computation which we illustrate for the case where $\gamma = \alpha$ and $\beta$. 
 \vskip 5pt
 
 \begin{itemize}
 \item $\gamma= \alpha$: the action of $s(u,t) \cdot \theta^{\vee}$ on the basis elements $\{x_{\alpha_1^{\vee}}, x_{\alpha_2^{\vee}}, x_{\alpha_3^{\vee}} \}$ is given by:
 \[  \begin{CD}
 x_{\alpha_i^{\vee}}  @>\theta^{\vee}>> x_{\alpha_{i+1}^{\vee}} @>{\rm Ad}(s(u,t))>>  t^2/u \cdot x_{\alpha_{i+1}^{\vee}}. \end{CD} \]
 Hence, with respect to the basis $\{x_{\alpha_i^{\vee}} \}$, the action of $s(u,t) \cdot \theta^{\vee}$ on $\mathfrak{u}_{\alpha}$ is given by the matrix 
 \[  t^2/u \cdot  E = t^2/u \cdot \left( \begin{array}{ccc}
 0 &0 & 1 \\
 1 & 0 & 0 \\
 0 & 1 & 0 \end{array} \right). \]
 The characteristic polynomial of $E$ is $-\lambda^3 +1$, so that its eigenvalues are the 3rd roots of unity. Hence, the 1-eigenspace of the $s(u,t) \cdot \theta^{\vee}$-action on  $\mathfrak{u}_{\alpha}$ is nonzero if and only if $t^6 =u^3$. 
 
 \vskip 5pt
 
 \item $\gamma = \beta$: the action of $s(u,t) \cdot \theta^{\vee}$ on $x_{\alpha_0^{\vee}}$ is given by
 \[  x_{\alpha_0^{\vee}} \mapsto u^2/t^3 \cdot x_{\alpha_0^{\vee}}. \]
 Hence $x_{\alpha_0^{\vee}}$ is fixed if and only if $u^2 = t^3$. 
  \end{itemize}
 We summarize the results of the computation for the other root spaces in the following table (in which we only indicate the results for positive roots). 
 \vskip 10pt
 
 
  
 
 
 

 \begin{center}
 \begin{tabular}{|c|c|c|}
  \hline
  Long roots &  $\theta$-orbit of & 1-eigenspace nonzero iff  \\
\hline  
  $\beta$ & $\alpha_0^\vee$ & $t^3 = u^2$ \\
\hline
 $3\alpha+ \beta$  & $\alpha_0^\vee+\alpha_1^\vee+\alpha_2^\vee+\alpha_3^\vee$ & $t^3 = u$ \\
 \hline
  $3 \alpha+2 \beta$ & $2\alpha_0^\vee+\alpha_1^\vee+\alpha_2^\vee+\alpha_3^\vee$ & $u=1$ \\ 
  \hline 
  \hline
   Short roots & $\theta$-orbit of & 1-eigenspace nonzero iff    \\
 \hline
   $\alpha$ & $\alpha_1^\vee$ & $t^6 = u^3$     \\
 \hline
  $\alpha+ \beta$  & $\alpha_0^\vee+\alpha_1^\vee$ & $t^3 = u^3$ \\
  \hline
 $2\alpha + \beta$ & $\alpha_0^\vee+\alpha_1^\vee+\alpha_2^\vee$ &  $t^3 =1$  \\
   \hline
 \end{tabular}
 \end{center} 
 \vskip 5pt

 Using the above data, our job now is to find all pairs $(u,t) \in \C^{\times} \times \C^{\times}$ such that $\mathfrak{spin}_8^{s(u,t) \cdot \theta^{\vee}}$ is  semisimple.
 Since we are interested in doing so up to conjugacy, it is helpful to first make a couple of elementary observations:
 \vskip 5pt
 
 \begin{itemize}
 \item[(a)] If $\zeta$ is a 3rd root of unity, then $s(u,t) \cdot \theta^{\vee}$ and $s(u,t \zeta)\cdot \theta^{\vee}$ are conjugate by the element
 \[   \alpha_1(1) \cdot \alpha_2(\zeta)  \cdot \alpha_3(\zeta^2)  \in T^{\vee}(\C). \]
 
 \item[(b)] The elements $s(u,1)\cdot \theta^{\vee}$ and $s(u^{-1}, 1)\cdot \theta^{\vee}$ are conjugate by the Weyl group element $w_{\alpha_0}$ associated to the simple reflection in $\alpha_0$. 
  \end{itemize}
 
 \vskip 5pt
 We can now begin our analysis:
 \vskip 5pt
 
 \begin{itemize}
 \item[(i)]  Referring to the 3 equations for the long root spaces (in the last column of the above table) as long root equations, we observe that if all 3 long root equations hold, then one has $ u=1$ and $t^3 =1$, so that the 3 short root equations also hold.  By observation (a) above, we may take $t =1$ as well. Hence, one obtains $(u,t) = (1,1)$ and $\mathfrak{spin}_8^{\theta^{\vee}} = \mathfrak{g}_2$. 
 \vskip 5pt
 
 \item[(ii)]  If two of the long root equations hold, then so does the last one, and hence we are reduced to the case above. 
 \vskip 5pt
 
 \item[(iii)]  If exactly one of the long root equations holds, let us suppose without loss of generality that the last one holds, i.e. $u=1$, but $t^3 \ne 1$. 
 Then we observe that among the short root equations, only the first one has a chance of holding, and this happens precisely when $t^6 =1$ but $t^3 \ne 1$. 
 If this short root equation does not hold, the fixed space is not semisimple and hence the resulting endoscopic group will not be elliptic. Hence, the only candidate for an elliptic endoscopic group here is for this equation to hold. By observation (a) above, there is no loss of generality in taking $t = -1$, so that we have $(u,t) = (1,-1)$, and $\mathfrak{spin_8}^{s(1,-1) \cdot \theta^{\vee}} =\mathfrak{so}_4$. 
 \vskip 5pt
 
 \item[(iv)]  Suppose that none of the long root equations hold. Then we need at least 2 of the short root equations to hold. But then all 3 of the short root equations will hold. Then we deduce that $u^3 = 1 = t^3$. By (a) again, we may take $t =1$, and by (b), we see that $s(\zeta,1) \cdot \theta^{\vee}$ and $s(\zeta^{-1}, 1) \cdot \theta^{\vee}$ are conjugate. Hence we have $(u,t) = (\zeta, 1)$ and $\mathfrak{spin}_8^{s(\zeta,1) \cdot \theta^{\vee}} = \mathfrak{sl_3}$. 
 \end{itemize}

 \vskip 5pt
 
 \subsection{\bf Elliptic endoscopic data.}
 From the Lie algebra computations of the previous subsection, we can deduce:
 \vskip 5pt

 \begin{prop}  \label{P:ell-endo-data}
 The elliptic endoscopic groups for the twisted space $(\PGSO_8, \theta)$ determined by the triality automorphism are the split groups ${\rm G}_2$, $\SL_3$ and $\SO_4$. 
 Moreover, the outer automorphism group of each of  these elliptic endoscopic data is trivial. 
 \end{prop}
 
 \vskip 5pt
 \begin{proof}
As mentionned above, because $\Spin_8(\C)$ is simply-connected, the group $\Spin_8(\C)^{ s \cdot \theta^{\vee}}$ is in fact connected \cite[Chap. 1, \S 1.1, Pg. 14]{KSh}.
  The Lie algebra computations above thus allow us to identify $\Spin_8(\C)^{ s \cdot \theta^{\vee}}$ up to isogeny, so we need to pin down the isogeny class. 
 In item (i) above, we already know that:
 \[  \Spin_8(\C)^{ \theta^{\vee}}  = {\rm G}_2(\C). \]
 For item (iv), observe that $s(\zeta, 1) \cdot \theta^{\vee}$ has order $3$ and thus provide a splitting of the automorphism short exact sequence (\ref{E:Aut}) over $\Z/3\Z$. As we have noted,  by \cite[Cor. 35.10 and Prop. 36.14]{KMRT}, there are only two such splittings (up to conjugacy) with fixed groups of type ${\rm G}_2$ or $\PGL_3$. Hence, in the context of item (iv), we have
 \[  \Spin_8(\C)^{ s(\zeta, 1) \cdot \theta^{\vee}}    \simeq \PGL_3 (\C).\]
 Finally, for item (iii), we observe that 
 \[  
 \mathfrak{spin_8}^{s(1,-1) \cdot \theta^{\vee}}  \subset \mathfrak{spin}_8^{\theta^{\vee}} = \mathfrak{g}_2 \]
 so that  $\Spin_8(\C)^{ s(1,-1) \cdot \theta^{\vee}}$ is a subgroup of  $\Spin_8(\C)^{ \theta^{\vee}}= G_2(\C)$ of type $A_1 \times A_1$. This implies 
 \[ \Spin_8(\C)^{ s(1,-1) \cdot \theta^{\vee}} \simeq \SO_4(\C). \]
 
The  first assertion of the proposition then follows from the definition of twisted endoscopic data given in \cite[Chap. 2]{KSh}.
\vskip 5pt

For the assertion on the triviality of automorphism groups, it suffices to show that for each of the 
three possibilities above for $H^\vee=\Spin_8(\C)^{ s \cdot \theta^{\vee}}$, the normalizer $N$ of $H^\vee$ in $\Spin_8(\C)$ is $H^\vee \times Z$ with $Z=Z_{\Spin_8}$.

\vskip 5pt

Firstly, an inspection of the $8$-dimensional representation ${\rm std}_{|H^\vee}$ (which is respectively of type $1+7$, irreducible, or $1+3+4$), shows  that the centralizer of $H^\vee$ in $\Spin_8(\C)$ is $Z_{H^\vee} \times Z$ in all cases. This shows the claim for  the ${\rm G}_2$ case, since  ${\rm Out} ({\rm G}_2)\,=1$. In the two other cases we have ${\rm Out}(\widehat{H}) \simeq  \Z/2$ and so some additional arguments are needed:

\begin{itemize}
\item[-]  (${\rm SO}_4$ case) We have $N=H^\vee$ since the $3$-dimensional summand of ${\rm std}_{|H^\vee}$ is not isomorphic to its outer conjugate.
\vskip 5pt

\item[-] (${\rm PGL}_3$ case) Assume there is $g \in N$ such that ${\rm int}_g$ induces $h \mapsto {}^{\rm t}h^{-1}$ on $H^\vee \simeq {\rm PGL}_3$. Then $g$ commutes with 
${\rm SO}_3(\C) = (H^\vee)^{{\rm int}_g}$. Now  ${\rm std}_{|{\rm SO}_3(\C)}$ has type $3+5$, and the element $g$ has to preserve each of these summands since it commutes with $\SO_3(\C)$ and thus acts by a scalar  on each summand. Indeed, since ${\rm std}(g)$ belongs to $({\rm O}_3(\C) \times  {\rm O}_5(\C)) \cap\SO_8(\C)$, 
the action of  $g$ on the two summands must be via the same sign $\pm 1$. 
This implies that $g \in Z$, which is a contradiction.
\end{itemize}
\vskip 5pt

\noindent Hence  we have shown that $N=H^\vee \times Z$ in all cases, so that the outer automorphism group of each of  these elliptic endoscopic data is trivial.
 \end{proof}
 \vskip 5pt

 \subsection{\bf Twisted Levi subspaces} \label{SS:twisted Levi}
We now describe  the twisted Levi subspaces for the split $(\PGSO_8, \theta)$. These correspond to $\theta$-stable Levi subgroups of $\PGSO_8$. 
Up to conjugacy, these in turn correspond to the $\theta$-stable subsets of the set $\Delta(T,B)$ of simple roots.  
Hence we may enumerate the proper twisted Levi subspaces as follows:
\vskip 5pt

\begin{itemize}
\item the empty subset of $\Delta(T,B)$: this corresponds to the $\theta$-stable maximal torus $T$, giving rise to the twisted Levi subsapce $T \cdot \theta$.
\vskip 5pt

\item the subset $\{ \alpha_0 \}$: this corresponds to a $\theta$-stable  Levi subgroup $L_0 \simeq  (\GL_2 \times \mathbb{G}_m^3)/ \mathbb{G}_m^{\Delta}$, where the action of $\theta$ is given by cyclic permutation of the three $\mathbb{G}_m$'s.

\vskip 5pt

\item the subset $\{\alpha_1, \alpha_2, \alpha_3\}$: this corresponds to a $\theta$-stable Levi subgroup $L \simeq (\GL_2 ^3 \times \mathbb{G}_m)/ j(\mathbb{G}_m^3)^{\Delta}$, where $j$ identifies $\mathbb{G}_m^{\Delta}$  with the center of $\GL_2^3$ and maps to $\mathbb{G}_m$ by the product map. The action of $\theta$ is given by cyclic permutation of the three $\GL_2$'s. 
\end{itemize}
\vskip 5pt

Now the theory of twisted endoscopy applies to each of the above twisted Levi subspaces as well. Hence one may determine their elliptic twisted endoscopic groups
following a similar computation as in the previous subsection. 
In fact, from the general theory of twisted endoscopy, the elliptic twisted endoscopic groups of a twisted Levi subspace of $(\PGSO_8, \theta)$ are naturally identified with Levi subgroups of  the elliptic twisted endoscopic groups $H$ of $(\PGSO_8, \theta)$. Since the groups $H$ are of rank $2$, besides their maximal torus $T_H$, their other proper Levi subgroups are the maximal ones and these are all  isomorphic to $\GL_2$. More precisely:

\vskip 5pt

\begin{itemize}
\item For $H = \SL_3$, there is a unique conjugacy class of maximal Levi subgroups, which we will denote by $\GL_{2,l}$;

\item For $H = G_2$, there are two conjugacy classes of maximal Levi subgroups, which we will denote by $\GL_{2,l}$ and $\GL_{2,s}$, corresponding to the long and short root $\GL_2$'s respectively;

\item For $H = \SO_4$, there are two conjugacy classes of maximal Levi subgroups, which we will denote by $\GL_{2,l}$ and $\GL_{2,s}$.
\end{itemize}
The nomenclature is chosen so that the groups labelled as $\GL_{2,l}$ above are naturally identified with each other via {\it admissible isomoprhisms},  and likewise for the groups labelled as $\GL_{2,s}$. 
These admissible isomorphisms  arise from the fact that $\xi_H(\GL_{2,l}^{\vee})$ gives the same  subgroup up to conjugacy in $\Spin_8(\C)$ for the various $H$. Likewise the various $\xi_H(\GL_{2,s}^{\vee})$ are the same subgroups in $\Spin_8(\C)$ up to conjugacy. 
\vskip 5pt

With the above notations, we can now tabulate the elliptic twisted endoscopic groups of the twisted Levi subspaces, as well as the Levi subgroups of $H$ they are identified with.
In the table, we have written
\[   {}^{1 - \theta}T := \text{ image of $[(1- \theta):T \rightarrow T]$.} \]
\vskip 5pt

\begin{center}
\begin{tabular}{|c|c|c|c|}
\hline 
Twisted Levi subspace &  $(T, \theta)$ & $(L_0,\theta)$ & $(L, \theta)$  \\
\hline  
Elliptic twisted endoscopic group &   $T / ^{1-\theta}T$  &  $\GL_2$ & $\GL_2$ \\
\hline 
Levi subgroup of $H$ identified with & $T_H$ & $\GL_{2,s}$    &  $\GL_{2,l}$   \\   
\hline
\end{tabular}
\end{center}

\vskip 10pt

\section{\bf Local Twisted Endoscopic Transfer}  \label{S:LET}
Let $F$ be a local field of characteristic $0$. 
Having determined the twisted endoscopic groups $H$ for $(\PGSO_8, \theta)$, the theory of twisted endoscopy furnishes a local transfer map from $\theta$-twisted orbital integrals  on $\PGSO_8(F)$ to orbital integrals on $H(F)$. In this section, we recall the basic objects and record some results on this local transfer that we will need.
\vskip 5pt

\subsection{\bf Local K-forms} \label{SS:local-K}
 A slight complication with the theory of stable twisted trace formula is the need to take into account of K-forms \cite[\S I.1.11]{MW1}, which are certain inner forms of the twisted space in question.
 \vskip 5pt
 
   In the context of $\PGSO_8$, we have seen in 
\S \ref{S:triality} that an octonion algebra $\mathbb{O}$ gives rise to a group $\PGSO_8^{\mathbb{O}}$, which is an inner form of the corresponding split group equipped with a triality automorphism. Over a number field or a local field,  these turn out to be the K-forms of the twisted space $(\PGSO_8, \theta)$.  Thus, over  a non-Archimedean local field  or $\C$, $(\PGSO_8, \theta)$ has no additional $K$-forms,  since there is a unique octonion algebra up to isomorphism. On the other hand,  over $\R$, there are two octonions algebras: the split one and a division algebra. As such, there are two K-forms of $(\PGSO_8, \theta)$ over $\R$.  
\vskip 5pt

\vskip 5pt

\subsection{\bf Twisted orbital integrals}
We now recall some properties of the space $\mathcal{I}^{\theta}(\PGSO^{\mathbb{O}}_8)$ of $\theta$-twisted orbital integrals on  a local K-form $\PGSO^{\mathbb{O}}_8$ over $F$. For simplicity, we will suppress $\mathbb{O}$ from the notation and write $G = \PGSO_8$.
 
\vskip 5pt

For a test function $f \in C^{\infty}_c(G(F))$,  its normalized $\theta$-twisted orbital integral $ \mathcal{O}^{\theta}(-, f) $ is a function on the subset of $\theta$-regular semisimple elements $\gamma$ of $G(F)$ defined by
\begin{equation} \label{E:orb-int1}    
\mathcal{O}^{\theta}(\gamma,f) = \Delta_{\PGSO_8} (\gamma \cdot \theta) \cdot \int_{ G_{\gamma \cdot \theta}(F) \backslash G(F) }  f( g^{-1} \gamma \theta (g)) \, \frac{dg}{dt},
\end{equation}
where $G_{\gamma\cdot \theta}$ denotes the stabilizer of $\gamma \cdot \theta$ in $G$ and 
\[  \Delta_{\PGSO_8} (\gamma \cdot \theta )  = | \det ( 1 - {\rm Ad}(t\theta) | \mathfrak{g} /  \mathfrak{g}_{\gamma \cdot \theta})|^{1/2} \]
is the positive  square root of the $\theta$-twisted Weyl discriminant. The space $\mathcal{I}^{\theta}(G)$ is the topological vector space of such normalized twisted orbital integrals.
\vskip 5pt

Assuming now that $G$ is split, we have the following structures on the space $\mathcal{I}^{\theta}(G)$:
\vskip 5pt

\begin{itemize}
\item[(i)] \cite[I.3.1]{MW1} \,  For each $\theta$-stable Levi subgroup $M$ of $G = \PGSO_8$ (or equivalently any twisted Levi subspace $M \cdot \theta$), there is a restriction or ``constant term" map:
\[ {\rm rest}^{G \cdot \theta}_{M \cdot \theta} : \mathcal{I}^{\theta}(G) \longrightarrow \mathcal{I}^{\theta}(M)^{W(G,M)^{\theta}} \]
where $W(G,M) = N_{G(F)}(M)/M(F)$ and $W(G,M)^{\theta}$ is the  $\theta$-fixed subgroup. Moreover, one has the transitivity of restrictions:
\[  {\rm rest}^{M \cdot \theta}_{T \cdot \theta}   \circ  {\rm rest}^{G\cdot \theta}_{M \cdot \theta}   = {\rm rest}^{G \cdot \theta}_{T \cdot \theta} \]
where $M = L_0$ or $L$.
\vskip 5pt

\item[(ii)] \cite[I.4.2]{MW1}\,  The restriction maps  give rise to a natural filtration 
\[  0=\mathcal{F}^{-1}  \subset \mathcal{F}^0 \subset \mathcal{F}^1 \subset \mathcal{F}^2  = \mathcal{I}^{\theta}(\PGSO_8). \]
These subspaces are given by:
\[  \mathcal{F}^1 = {\rm Ker}({\rm rest}^{G \cdot \theta}_{T \cdot\theta})  \quad \text{and} \quad 
 \mathcal{F}^0 =  {\rm Ker}({\rm rest}^{G \cdot \theta}_{L_0 \cdot \theta}) \cap  {\rm Ker}({\rm rest}^{G \cdot \theta}_{L \cdot \theta}). \]
The subspace $\mathcal{F}^0$ is also called the cuspidal subspace and is denoted by $\mathcal{I}^{\theta}_{cusp}(G)$.
\vskip 5pt

\item[(iii)] \cite[I.4.2]{MW1}\, The successive quotients of the above filtration are given by
\[   \mathcal{F}^2 / \mathcal{F}^1 \simeq  \mathcal{I}^{\theta}_{cusp}(T)^{W(G,T)^{\theta}} \]
and
\[ \mathcal{F}^1/ \mathcal{F}_0 \simeq \mathcal{I}^{\theta}_{cusp}(L_0)^{W(G,L_0)^{\theta}} \oplus \mathcal{I}^{\theta}_{cusp}(L)^{W(G,L)^{\theta}}.   \]
The isomorphisms are naturally induced  by the restriction maps.
\end{itemize}

\vskip 5pt

\noindent{\bf Remark:} In (i) above, we have used the group $W(G,M)^{\theta}$, whereas in \cite[I.3.1 and I.4.2]{MW1}, one sees instead
\[  W(M \cdot \theta) = N_{G(F)}(M \cdot \theta) / M(F). \]
Let us show that these are the same. Suppose that $g \in W(M \cdot \theta)$, so that $g^{-1} M \theta(g) = M$.  This implies that $gMg^{-1} = M \cdot \theta(g)g^{-1}$. 
For this to hold, one must have $\theta(g) g^{-1} \in M$, so that $g$ normalizes $M$ and its image in $W(G,M) = N_{G(F)}(M)/M(F)$ is fixed by $\theta$. 
Conversely, if $g \in N_{G(F)}(M)$ is such that $\theta(g) g^{-1} \in M(F)$, then
\[  g^{-1} M \cdot \theta \cdot g = g^{-1} M \cdot \theta(g) \cdot \theta = g^{-1}M g \cdot\theta = M \cdot \theta, \]
so that $g \in N_{G(F)}(M \cdot\theta)$. 
\vskip 5pt

\subsection{\bf Stable orbital integrals}
For  an connected reductive group $H$ over $F$, let $\mathcal{SI}(H)$ denote the space of stable orbital integrals on $H(F)$.
Recall that  for a test function $f_H \in C^{\infty}_c(H(F))$, its normalized orbital integral is a function on strongly regular semisimple elements $\gamma$ of $H(F)$
defined by
\begin{equation} \label{E:orb-int2}
 \mathcal{O}(\gamma, f_H) = \Delta_H(\gamma) \cdot  \int_{H_{\gamma}(F) \backslash H(F)} f_H( x^{-1} \gamma  x) \, \frac{dx}{dt}, \end{equation}
where $H_{\gamma}$ is the stabilizer of $\gamma$ in $H(F)$ (and hence is a maximal $F$-torus) and 
\[ \Delta_H(\gamma) =  |\det (1- {\rm Ad}(\gamma) |  \mathfrak{h} / \mathfrak{h}_{\gamma}))|^{1/2} \]
is the positive square root of the Weyl discriminant.  We may further replace the normalized orbital integral above by its stabilized version:
\begin{equation} \label{E:orb-int3}
 \mathcal{SO}(\gamma, f_H) = \sum_{\gamma'}   \mathcal{O}(\gamma', f_H) \end{equation}
with the sum running over the conjugacy classes in $H(F)$ which are stably conjugate to $\gamma$.  
The space spanned by  such normalized stable orbital integrals is the  topological vector space $\mathcal{SI}(H)$.
\vskip 5pt

We have the analogous structures on $\mathcal{SI}(H)$:
\vskip 5pt

\begin{itemize}
\item[(i)] For each Levi subgroup $M \subset H$, there is a  restriction or ``constant term" map:
\[  {\rm rest}_M^H:  \mathcal{SI}(H) \longrightarrow \mathcal{SI}(M)^{W(H,M)} \]
where $W(H,M) = N_{H(F)}(M)/ M(F)$. 

\vskip 5pt

\item[(ii)] The restriction maps give rise to a natural filtration on $\mathcal{SI}(H)$:
\[ 0  =: \mathcal{F}^{-1} \subset  \mathcal{F}^0  \subset \mathcal{F}^1 \subset ...\subset \mathcal{F}^r = \mathcal{SI}(H). \] 
The subspace $\mathcal{F}^k$ is defined by 
\[  \mathcal{F}^k =  \bigcap_{M \in \mathcal{L}_{>k}(H)}  {\rm Ker}({\rm rest}_M^H) \] 
where the intersection runs over the set $\mathcal{L}_{>k}(H)$ of representatives of $H(F)$-conjugacy classes of Levi subgroups $M$ of $H$   with ${\rm corank}(M) :=  \dim Z_M/ Z_H > k$ (here $Z_H$ and $Z_M$ denote the centers of $H$ and $M$).  
In particular, $\mathcal{F}^0$ denotes the subspace of elements whose restriction to any proper Levi subgroup vanish. 
This subspace is also called the {\it cuspidal subspace} and is denoted by $\mathcal{SI}_{cusp}(H)$. 
\vskip 5pt

\item[(iii)] The successive quotients of the above filtration is given by:
\[  \mathcal{F}^k / \mathcal{F}^{k-1} \simeq  \bigoplus_{M \in \mathcal{L}_k(H)}  \mathcal{SI}_{cusp}(M)^{W(H,M)} \]
where the sum runs over a set $\mathcal{L}_k(H)$ of representatives of $H(F)$-conjugacy classes of Levi subgroups $M$ with corank($M$) $=k$.  
The isomorphism is given by the restriction maps to Levi subgroups of corank $k$. 
\end{itemize}

\vskip 5pt

We will of course be applying the above results to the case when $H$ is an elliptic twisted endoscopic group of $(\PGSO_8, \theta)$.
As an illustration, let us explicate the above general results when  $H = \SL_3$. In this case, using the notations from \S \ref{SS:twisted Levi}, we have
\[   \mathcal{L}_1(H) =  \{ \GL_{2,l} \} \quad \text{and} \quad \mathcal{L}_2(H) = \{T_{\SL_3} \}. \]
Then the above results say that
\begin{equation} \label{E:Fil1}
 \mathcal{F}^1/\mathcal{F}^0 \simeq  \mathcal{SI}_{cusp}(\GL_{2,l})  = \mathcal{I}_{cusp}(\GL_{2,l}) \end{equation}
and
\begin{equation} \label{E:Fil2}
 \mathcal{F}^2/\mathcal{F}^1 \simeq \mathcal{SI}(T_{\SL_3})^{W(\SL_3, T_{\SL_3})} = \mathcal{I}(T_{\SL_3})^{S_3}. \end{equation}
 Here, for the first identity, we note that the Weyl group $W(\SL_3, \GL_{2,l})$ is trivial and any invariant distribution on $\GL_{2,l}$ is already stably-invariant, whereas for the second identity,  we have $W(\SL_3, T_{\SL_3}) \simeq S_3$. 
 \vskip 5pt
 
\subsection{\bf Local transfer} \label{SS:local-T}
For each local $K$-form $(\PGSO_8^{\mathbb{O}}, \theta)$ and each twisted endoscopic group $H$, the theory of twisted endoscopy furnishes a local transfer of orbital integrals from $(\PGSO_8^{\mathbb{O}}, \theta)$ to $H$. More precisely, we have:
 \vskip 5pt

\begin{prop} \label{P:lt}
Let $F$ be a  local field of characteristic 0 and consider the setting of $(\PGSO_8, \theta)$ over $F$.
 Then we have:
\vskip 5pt

\noindent (i) The twisted endoscopic transfer gives an isomorphism of topological vector spaces
\[  {\rm trans}: \bigoplus_{\mathbb{O}} \mathcal{I}^{\theta}_{cusp}(\PGSO^{\mathbb{O}}_8) \longrightarrow \mathcal{SI}_{cusp}({\rm G}_2) \oplus \mathcal{SI}_{cusp}(\SO_4) \oplus \mathcal{SI}_{cusp}(\SL_3), \]
where the sum runs over the isomorphism classes of octonion $F$-algebras.

\vskip 5pt

\noindent (ii) More generally,  the twisted endoscopic transfer gives an injective map
\[  {\rm trans} : \bigoplus_{\mathbb{O}} \mathcal{I}^{\theta}(\PGSO^{\mathbb{O}}_8) \longrightarrow \mathcal{SI}({\rm G}_2) \oplus \mathcal{SI}(\SO_4) \oplus \mathcal{SI}(\SL_3).\]
Its image is the subspace of  ``compatible families"  of stably invariant distributions of its target  (a notion we will recall after the proposition; see \cite[I.4.11]{MW1}).

\vskip 5pt

 \noindent (iii) Let $F$ be a non-Archimedean local field whose residue characteristic is sufficiently large. For each of the group $G$ under consideration here, let $\mathcal{H}_G$ denote the spherical Hecke algebra of $G(F)$ with respect to a hyperspecial subgroup.  For each elliptic twisted endoscopic group $H$ of $(\PGSO_8, \theta)$, the  embedding $\xi_H: H^{\vee} \hookrightarrow \Spin_8(\C)$ induces (via the Satake isomorphism) a natural algebra homomorphism 
 \[   \xi_H^*: \mathcal{H}_{\PGSO_8} \longrightarrow  \mathcal{H}_H. \]
 Then one has a commutative diagram (the fundamental lemma for spherical Hecke algebras):
 \[  \begin{CD}
  \mathcal{H}_{\PGSO_8} @>\xi_H^*>>  \mathcal{H}_H  \\
  @VVV   @VVV   \\
    \mathcal{I}^{\theta}(\PGSO_8)   @>{\rm trans}_H>>  \mathcal{SI}(H),
   \end{CD} \]
   where the vertical arrows are the formation of  $\theta$-twisted orbital integrals and stable orbital integrals respectively.
\end{prop}
\vskip 5pt

\begin{proof}
For (i) and (ii), see \cite[Prop. I.4.11, Pg. 96]{MW1}.  For the fundamental lemma in (iii), see \cite[\S I.6.4, Pg. 165]{MW1} and the references therein.
  \end{proof}
\vskip 5pt

We have still to define the notion of a ``compatible family" of stably invariant distributions in Proposition \ref{P:lt}(ii) above.
Suppose that 
\[   (f_{G_2}, f_{\SO_4}, f_{\SL_3})   \in   \mathcal{SI}({\rm G}_2) \oplus \mathcal{SI}(\SO_4) \oplus \mathcal{SI}(\SL_3).\]
 We say that this is a compatible family if  
\[ {\rm rest}^{G_2}_{\GL_{2,l}}  ( f_{G_2}) =    {\rm rest}^{\SO_4}_{\GL_{2,l}}  ( f_{\SO_4}) =    {\rm rest}^{\SL_3}_{\GL_{2,l}}  ( f_{\SL_3}), \]
and
\[  {\rm rest}^{G_2}_{\GL_{2,s}}  ( f_{G_2}) =    {\rm rest}^{\SO_4}_{\GL_{2,s}}  ( f_{\SO_4}).  \]
Here, we recall that the Levi subgroups of the various twisted endoscopic groups $H$ labelled as $\GL_{2,l}$ (respectively $\GL_{2,s}$) are naturally identified with each other via admissible isomorphisms, so that the above identities make sense.
The transitivity of restriction maps implies that when the above two identities hold, we also have 
\[   {\rm rest}^{G_2}_{T_{G_2}}  ( f_{G_2}) =    {\rm rest}^{\SO_4}_{T_{\SO_4}}  ( f_{\SO_4}) =    {\rm rest}^{\SL_3}_{T_{\SL_3}}  ( f_{\SL_3}).\]         
\vskip 5pt

 For any elliptic twisted endoscopic group $H = G_2$, $\SL_3$ or $\SO_4$ of $(\PGSO_8, \theta)$, the projection of  the map ${\rm trans}$ in (ii)  to the $H$-factor gives 
\[  {\rm trans}_H:  \bigoplus_{\mathbb{O}} \mathcal{I}^{\theta}(\PGSO^{\mathbb{O}}_8)  \longrightarrow  \mathcal{SI}(H). \]
By Proposition \ref{P:lt}(ii), the image of ${\rm trans}_H$ is precisely the subspace of those elements which can be completed to a compatible family. Moreover, the natural filtration $(\mathcal{F}^{\bullet})$ on $\mathcal{SI}(H)$ induces the filtration  
\[ \mathcal{T}^{\bullet}  :=  \mathcal{F}^{\bullet} \cap {\rm trans}_H(I^{\theta}(G)) \quad \text{ on ${\rm trans}_H(I^{\theta}(G))$,} \]
so that
\[ \mathcal{T}^k / \mathcal{T}^{k-1} \subset \mathcal{F}^k/ \mathcal{F}^{k-1}.\]
 The following proposition characterizes 
${\rm trans}_H(I^{\theta}(G))$ in terms of this filtration and the associated graded pieces:
\vskip 5pt

\begin{prop} \label{P:image}
\noindent (i) Let $H = G_2$. Then the successive quotients of the filtration $(\mathcal{T}^{\bullet})$ are given by:
\[  \begin{cases}
\mathcal{T}^2/\mathcal{T}^1 =  \mathcal{I}(T_{G_2})^{W(G_2, T_{G_2})},\\
\mathcal{T}^1/\mathcal{T}^0  = \mathcal{I}_{cusp}(\GL_{2,l})^{{\rm Out}(\GL_2)} \oplus   \mathcal{I}_{cusp}(\GL_{2,s})^{{\rm Out}(\GL_2)},\\
\mathcal{T}^0 = \mathcal{F}^0 = \mathcal{I}_{cusp}(G_2).
\end{cases} \]
Here ${\rm Out}(\GL_2)$ refers to the outer automorphism group of $\GL_2$.
\vskip 5pt

(ii) Let $H = \SL_3$. Then the successive quotients of the filtration $(\mathcal{T}^{\bullet})$ are given by:
\[  \begin{cases}
\mathcal{T}^2/\mathcal{T}^1 = \mathcal{I}(T)^{W(G_2, T_{G_2})},  \\
\mathcal{T}^1/\mathcal{T}^0  = \mathcal{I}_{cusp}(\GL_{2,l})^{{\rm Out}(\GL_2)},  \\
\mathcal{T}^0 = \mathcal{F}^0 = \mathcal{SI}_{cusp}(\SL_3).
\end{cases} \]
Here, we have  made use of the natural identification $T_{\SL_3} \simeq T_{G_2}$ to transport the action of $W(G_2, T_{G_2})$ to $T_{\SL_3}$, and under this transport of structure, $W(G_2, T_{G_2}) = W(\SL_3, T_{\SL_3}) \times S_2$ with the nontrivial element of $S_2$ acting by inverting on $T_{\SL_3}$.

 \vskip 10pt

\vskip 10pt

\noindent (iii) Let $H = \SO_4$. Then the successive quotients of the filtration $(\mathcal{T}^{\bullet})$ are given by:
\[  \begin{cases}
\mathcal{T}^2/\mathcal{T}^1 =  \mathcal{I}(T)^{W(G_2,T_{G_2})},\\
\mathcal{T}^1/\mathcal{T}^0  = \mathcal{I}_{cusp}(\GL_{2,l})^{{\rm Out}(\GL_2)} \oplus   \mathcal{I}_{cusp}(\GL_{2,s})^{{\rm Out}(\GL_2)},\\
\mathcal{T}^0 = \mathcal{F}^0 = \mathcal{I}_{cusp}(\SO_4).
\end{cases} \]
Here, we have again made use of the natural identification $T_{\SO_4} \simeq T_{G_2}$ to transport the action of $W(G_2, T_{G_2})$ to $T_{\SO_4}$.
\end{prop}

\begin{proof}
This can be deduced from \cite[\S I.4.3, Lemme, Pg. 80]{MW1} and the definition of ``compatible families"; we omit the details. 
\end{proof}

\begin{cor} \label{C:image}
The map  ${\rm trans}_H$ is surjective if $H = G_2$, but is not surjective if $H = \SL_3$ or $\SO_4$. When $H = \SL_3$, the image of ${\rm trans}_{\SL_3}$ contains the subspace $\mathcal{SI}(\SL_3)^{{\rm Out}(\SL_3)}$ of $\mathcal{SI}(\SL_3)$ fixed by the outer automorphism group of $\SL_3$. 
\end{cor}
 \vskip 5pt
\begin{proof}
For the case $H = \SL_3$, one may compare Proposition \ref{P:image}(i) and the equations (\ref{E:Fil1}) and (\ref{E:Fil2}). The other cases are similar and so we omit the details. 
\end{proof}

 \vskip 5pt

\subsection{\bf Spectral transfer} Let us now specialize to the case $H = \SL_3$. The dual of the transfer map ${\rm trans}_{\SL_3}$ is a map
\[  {\rm trans}_{\SL_3}^* : \mathcal{SI}(\SL_3)^* \longrightarrow \mathcal{I}(\PGSO_8)^* \]
from the space of stably invariant distributions of $SL_3(F)$ to the space of invariant distributions on $\PGSO_8(F)$. 
The fact that the transfer map ${\rm trans}_{\SL_3}$ is not surjecitive implies that this dual transfer map is not injective. However, we note:
\vskip 5pt

\begin{lemma} \label{L:nonzero}
Suppose that $\pi$ is a semisimple admissible representation of $\SL_3(F)$ whose Harish-Chandra character $\Theta_{\pi}$ is a stable distribution.  
Then $  {\rm trans}_{\SL_3}^* (\Theta_{\pi}) \ne 0$.
\end{lemma} 

 \begin{proof}
 By the description of the image of ${\rm trans}_{\SL_3}$ given in Corollary \ref{C:image}, it suffices to show that $\Theta_{\pi}$ is not identically zero on $\mathcal{SI}(\SL_3)^{{\rm Out}(\SL_3)}$.  
 Let $\sigma$ be a nontrivial outer automorphism of $\SL_3$. Then the space $\mathcal{SI}(\SL_3)^{{\rm Out}(\SL_3)}$ is spanned by the stable orbital integrals of test functions of the form $f + f^{\sigma}$. 
 Now since $\pi^{\sigma} \simeq \pi^{\vee}$, we have 
 \[ \Theta_{\pi}(f^{\sigma}) =  \Theta_{\pi^{\sigma}}(f) = \Theta_{\pi^{\vee}}(f),  \]
  so that
 \[  \Theta_{\pi}(f + f^{\sigma}) = \Theta_{\pi \oplus \pi^{\vee}}(f). \]
 Since $\pi \oplus \pi^{\vee}$ is a nonzero admissible semisimple representation, it is thus clear that $\Theta_{\pi \oplus \pi^{\vee}}$ is nonzero on some test function  $f$. 
  \end{proof}

\vskip 10pt

\section{\bf Unramified Spectral Transfer} \label{SS:local-unram-T}
We now specialize to the setting of a non-Archimedean local field $F$ with ring of integers $\mathcal{O}_F$ (so that there is a unique K-form, namely the split $\PGSO_8$). Recall that in \S \ref{SS:CS}, we have started  with  the Chevalley group  $\PGSO_8$ defined over $\Z$ with a given Chevalley-Steinberg system of \'epinglage and equipped with a triality automorphism $\theta$. Hence we have a hyperspecial subgroup $K = G(\mathcal{O}_F)$ which is stablized by $\theta$. 
Since $\PGSO_8$ is adjoint, there is a unique conjugacy class of hyperspecial subgroups.
\vskip 5pt

\subsection{\bf Unramified L-packets}
 Now let $H$ be one of the elliptic endoscopic groups of $(\PGSO_8, \theta)$. In fact, we shall only be interested in the case $H = \SL_3$ and the reader can assume this is the case in the following.  Suppose that we have an unramified L-packet of $H$ associated to a semisimple conjugacy class $s \in H^{\vee}$.
Recall that $s$ determines an unramified character $\chi_s$ of a maximal split torus $T_H$ of $H$ and the L-packet in question consists simply of the constituents of the corresponding principal series representation $I_H(\chi_s)$ (parabolically induced from $\chi_s$) which have nonzero invariant vectors under some hyperspecial subgroups of $H(F)$.
Let us assume for simplicity that  the unramified constituents of  $I_H(\chi_s)$ are generic and unitary. In that case,  $I_H(\chi_s)$ is semisimple and the L-packet consists of all its irreducible constituents.
\vskip 5pt

The stable trace $\Theta_s$ of such an unramified  L-packet is the sum of the Harish-Chandra character of its members and hence (under our hypothesis above) is  simply the Harish-Chandra character of $I_H(\chi_s)$.  It is a stable distribution on $H(F)$ and hence defines  a linear  functional on $\mathcal{SI}(H)$.    We may consider the pullback of $\Theta_s$ under the transfer map ${\rm trans}_H$ of Proposition \ref{P:lt}(iii). For the purpose of applying the trace formula, we need to determine the unramified spectral transfer $\Theta_s \circ {\rm trans}_H$ precisely.

\vskip 5pt

  To formulate the answer, consider the unramified L-packet  of $\PGSO_8(F)$ corresponding to the semisimple conjugacy class of $\xi_H(s)$, where $\xi_H: H^{\vee} \hookrightarrow \Spin_8(\C)$ is the natural embedding. This unramified L-packet consists of the unramified constituents of an unramified principal series representation which can be described as follows. 
 The restriction of the embedding $\xi_H$ defines an isomorphism
  \[  \xi_H:  T_H^{\vee}  \longrightarrow (T^{\vee})^{\theta^{\vee}}  \subset T^{\vee}, \]
  where $T$ is a maximal split torus of $\PGSO_8$.  This injective  map of dual torus arises from a surjective morphism 
  \begin{equation} 
  \label{E:eta1}
      \eta: T \longrightarrow T_H\end{equation}
  satisfying
  \begin{equation} \label{E:eta2}
   {\rm Ker}(\eta) = {}^{1 - \theta}T := \text{ image of $[(1- \theta):T \rightarrow T]$.} \end{equation}
  By pulling back via $\eta$,  one obtains an unramified character  $\chi_s \circ \eta$  of $T(F)$.  Then the unramified principal series representation of $\PGSO_8(\F)$ corresponding to $\xi_H(s)$ is the representation
     $I(\chi_s \circ \eta)$ parabolically induced from $\chi_s \circ \eta$. 
  \vskip 5pt
  
\subsection{\bf Extensions to twisted space}   Since $\xi_H(s)$ is fixed by $\theta^{\vee}$, the character $\chi_s \circ \eta$ is $\theta$-invariant, which implies that 
     the isomorphism class of $I(\chi_s \circ \eta)$  is fixed by $\theta$ and hence  can be extended to the semi-direct product $\PGSO_8(F) \rtimes \langle \theta \rangle$ (or rather to the twisted space $\PGSO_8(F) \cdot \theta$).
     The extension is however not unique.     The following lemma gives several equivalent ways of fixing such an extension.
     
     \vskip 5pt
     
     \begin{lemma}. \label{L:extension}
     The following three extensions of $I(\chi_s \circ \eta)$ to $\PGSO_8(F) \rtimes \langle \theta \rangle$ are the same:
    \vskip 5pt
    \begin{itemize}
    \item[(i)] (Extension by induction)  Extend $\chi_s \circ \eta$ to $T(F) \rtimes \langle \theta \rangle$ by letting $\theta$ act trivially;  by the functoriality of induction, this gives rise to an action of $\theta$ on $I(\chi_s \circ \eta)$, described explicitly  by
   \[  (\theta \cdot \phi) (g) = \phi( \theta^{-1}(g)) \]
    for  any $\phi \in I(\chi_s \circ \eta)$ and $g \in \PGSO_8(F)$. 
    \vskip 5pt

    \item[(ii)]  (Spherical-normalized extension) Extend $I(\chi_s \circ \eta)$ by requiring that $\theta$ acts as identity on  the 1-dimensional space $I(\chi_s \circ \eta)^K$.
    \vskip 5pt
    
    \item[(iii)]  (Whittaker-normalized extension)  For a nontrivial additive character $\psi$ of $F$ of consuctor $\mathcal{O}_F$, consider the Whittaker datum determined by $\psi$ and the fixed Chavalley-Steinberg system of \'epinglage and a corresponding nontrivial Whittaker functional  $\mathcal{W}: I(\chi_s \circ \eta) \longrightarrow \C$. Then the Whittaker datum is fixed by $\theta$ and we  extend $I(\chi_s \circ \eta)$ by requiring that
  \[  \mathcal{W} \circ \theta = \mathcal{W}. \]
      \end{itemize}
      \end{lemma}
   \vskip 5pt
   
   \begin{proof}
   Let $\phi_0 \in I(\chi_s \circ \eta)$ be a  nonzero $K$-unramfiied vector which is unique up to scaling. Then $\phi_0$ is characterized by the requirement that $\phi_0|_K$ is a constant function. Since $\theta$ stabilizes $K$, one deduces easily that the extensions in (i) and (ii) are the same. On the other hand, by the Casselman-Shalika formula, one knows that $\mathcal{W}(\phi_0) \ne 0$. Hence the Whittaker-normalzied extension is also spherical-normalized, showing that the extensions in (ii) and (iii) are the same. 
      \end{proof}
   \vskip 5pt

We shall adopt the extension of $I(\chi_s \circ \eta)$ to $\PGSO_8(F) \rtimes \langle \theta \rangle$  supplied by the above lemma, denoting the resulting action of $\theta$  by $I(\theta)$. Using this extension, we define the $\theta$-twisted trace of $I(\chi_s \circ \eta)$ as the Harish-Chandra character of the extended representation, restricted to the coset $\PGSO_8(F) \cdot \theta$.  
More precisely, it is given by  the linear functional
\[  \Theta^{\theta}_{\xi_H(s)}:  f \mapsto   {\rm Tr}( I(\theta) \circ  I(\chi_s \circ \eta)(f) ), \quad \text{ for $f \in C^{\infty}_c(\PGSO_8(F))$.} \]
which is $\theta$-conjugacy-invariant  and hence factors to a linear functional on $\mathcal{I}^{\theta}(\PGSO_8)$.
Observe that if $f \in \mathcal{H}_{\PGSO_8}$, then 
\[   {\rm Tr}( I(\theta) \circ  I(\chi_s \circ \eta)(f) ) =  {\rm Tr}(  I(\chi_s \circ \eta)(f) ), \]
since  $I(\chi_s \circ \eta)(f)$ has image  contained in $I(\chi_s \circ \eta)^K$ and $I(\theta)$ acts as identity on $I(\chi_s \circ \eta)^K$. 
Hence, 
\[ \Theta^{\theta}_{\xi_H(s)} = \Theta_{\xi_H(s)} \quad \text{on $\mathcal{H}_{\PGSO_8}$.} \]
\vskip 5pt

\subsection{\bf Unramified spectral transfer}
Now one has:
\vskip 5pt

\begin{prop} \label{P:unram-spec-T}
In the above context,  one has the twisted endoscopic character identity:
\[  \Theta_s \circ {\rm trans}_H = \Theta^{\theta}_{\xi_H(s)}   \]
as linear functionals on $\mathcal{I}^{\theta}(\PGSO_8)$. In particular, the twisted endoscopic character identity defines a lifting of unramified L-packets from $H$ to $\PGSO_8$ which is functorial with respect to the embedding
\[  \xi_H: H^{\vee} \longrightarrow \PGSO_8^{\vee} = \Spin_8(\C). \] 
\end{prop}
\vskip 5pt

\begin{proof}
For a test function $f \in C^{\infty}_c(\PGSO_8(F))$, we need to show that
\[ \Theta_s ({\rm trans}_H(f)) =  \Theta^{\theta}_{\xi_H(s)} (f). \]
If $f \in \mathcal{H}_{\PGSO_8}$, then this identity follows readily from Proposition \ref{P:lt}(iii). In the general case,
the proof of this identity  is similar to  that of \cite[Prop. 6.2]{CG}.  
\vskip 5pt

We first  compute each side of the purported identity.
For a test function $f_H \in C^{\infty}_c(H(F))$, we first compute $\Theta_s(f_H)$. 
Recall that the Harish-Chandra character $\Theta_s$ of $I_H(\chi_s)$ is given by a locally constant and locally integrable conjugacy-invariant function on the open dense set $H(F)_{rs}$ of regular semisiple elements of $H(F)$. We denote this function by $\Theta_s$ as well, so that
\[   \Theta_s(f_H) = \int_{H(F)_{rs}} \Theta_s( x) \cdot f_H(x) \,dx. \]
Moreover, since $\Theta_s$ is the character of a principal series representation, the function $\Theta_s$ is supported on the conjugacy classes of regular elements of the maximal split torus $T_H$.  Hence, by the Weyl integration formula, one has (cf. \cite[\S 5.8]{CG})
\[
   \Theta_s(f_H) =      \int_{T_H(F)_{rs}}    \chi_s(t_H) \cdot  \mathcal{O}( t_H, f_H) \, dt_H \]
where   $\mathcal{O}(- , f_H)$ is the normalized orbital integral of $f_H$ defined in (\ref{E:orb-int2}).
  Since $t_H$ is a regular element in the maximal split torus, there is a unique conjugacy class in the stable conjugacy class of $t_H$, and hence
  one has $\mathcal{SO}(t_H, f_H)  =  \mathcal{O}(t_H, f_H)$, where  $\mathcal{SO}(t_H, f_H)$ is the stable orbital integral defined in (\ref{E:orb-int3}).
 Hence, we have:
\begin{equation} \label{E:trace1}
   \Theta_s(f_H) =   \int_{T_H(F)_{rs}}    \chi_s(t_H) \cdot  \mathcal{SO}( t_H, f_H) \, dt_H. \end{equation}

\vskip 5pt

For a test function $f \in C^{\infty}_c(\PGSO_8(F))$, we may likewise compute the twisted trace $\Theta^{\theta}_{\xi_H(s)}(f)$ analogously, using a twisted Weyl integration formula (see \cite[\S 5.7 and \S 5.8]{CG}). This gives
\begin{equation} \label{E:trace2}
  \Theta^{\theta}_{\xi_H(s)}(f) =  \int_{ ^{1-\theta}T(F) \backslash T(F)_{\theta-rs}}  \chi_{\xi_H(s)}(t)  \cdot \mathcal{O}^{\theta}(t, f) \, dt \end{equation}
where $T(F)_{\theta-rs}$ denotes the open dense subset of $\theta$-regular elements and $\mathcal{O}_{\theta}(t, f) $ is the $\theta$-twisted orbital integral  of $f$ defined 
in (\ref{E:orb-int1}).
 By (\ref{E:eta1}) and (\ref{E:eta2}), one has  
\[  \eta: T / {}^{1-\theta}T   \simeq T_H \quad \text{and} \quad    \chi_{\xi_H(s)} = \chi_s \circ \eta. \]
In addition, under the orbit correspondence in the theory of twisted endoscopy, 
the pair $(\eta(t), t)  \in T_H(F)_{rs} \times T(F)_{\theta_rs}$ correspond. Moreover, the value of Whittaker-normalized twisted transfer factor on $(\eta(t), t)$ is given by
\[   \Delta(\eta(t), t) = 1, \]
so that by the definition of transfer,
\[  \mathcal{SO}( \eta(t), {\rm trans}_H(f)) =  \mathcal{O}^{\theta}(t, f). \]   
Hence, by a change of variable, setting $t_H = \eta(t)$,   
\begin{equation} \label{E:trace3}
 \Theta^{\theta}_{\xi_H(s)}(f) =  \int_{ T_H(F)}  \chi_s( t_H)  \cdot \mathcal{SO}( t_H, {\rm trans}_H(f)) \,    dt_H. \end{equation}
\vskip 5pt

Comparing (\ref{E:trace1}) and (\ref{E:trace3}) and taking $f_H = {\rm trans}_H(f)$ in (\ref{E:trace1}), we deduce the desired unramified spectral transfer identity.
\end{proof}
\vskip 5pt

It seems to us that this unramified spectral transfer  could and should be shown in the general context of twisted endoscopy; this will eliminate the need to deal with it on a case-by-case basis.
\vskip 5pt

\subsection{\bf The case $H = \SL_3$} 
Let us specialize further to the case of the twisted endoscopic group $H = \SL_3$ and explicate some of the constructions discussed above.
 In this case,  as we have noted before, the natural map of dual groups:
\[ \xi_H : H^{\vee} = \PGL_3(\C) \longrightarrow \PGSO_8^{\vee} = \Spin_8(\C) \]
  is the unique lift of  the adjoint representation 
\[   {\rm Ad}: \PGL_3(\C) \longrightarrow \SO_8(\C) \]
of $\PGL_3(\C)$ on its Lie algebra, from $\SO_8(\C)$ to $\Spin_8(\C)$. 
\vskip 5pt

Let us explicate  explicate the map 
\[  \xi: T_H^{\vee}  \longrightarrow T^{\vee} \]
and its dual
\[ \eta:  T \longrightarrow T_H \subset \SL_3. \]
We shall regard $T_H^{\vee}$ as the diagonal torus in $\PGL_3(\C)$ and recall from
 the discussion in \S \ref{SS:Lie alg} that  the torus $T^{\vee}$ is given by
\[ T^{\vee} = \{ \alpha_0(u) \cdot \alpha_1(t_1) \cdot  \alpha_2(t_2) \cdot \alpha_3 (t_3): u, t_i \in \mathbb{G}_m \}.\]
  Then the map $\xi$ is given on the diagonal maximal torus by:
\[  \xi:   \left( \begin{array}{ccc}
a & & \\
& b &  \\
& & c \end{array} \right) \mapsto  u\left(\frac{ab}{c^2}, \frac{a}{c}\right) = \alpha_0\left(\frac{ab}{c^2}\right) \cdot    \alpha_1\left(\frac{a}{c}\right) \cdot  \alpha_2\left(\frac{a}{c}\right) \cdot \alpha_3 \left( \frac{a}{c}\right), \]
where the element $s(u,t) \in (T^{\vee})^{\theta}$ was defined in (\ref{E:torus}). On the other hand , we may realize $T_H$ as the diagonal torus in $H= \SL_3$, and identify the torus $T$ with $\mathbb{G}_m^4$ via:
\[  T = \{ \omega_0^{\vee}(u) \cdot\omega_1^{\vee}(t_1)  \cdot\omega_2^{\vee}(t_2)  \cdot\omega_3^{\vee}(t_3): u, t_i \in \mathbb{G}_m \}, \]  
where $\{\omega_0^{\vee} ,\omega_1^{\vee}, \omega_2^{\vee}, \omega_3^{\vee} \}$  is the basis of $X_*(T)$ formed by fundamental coweights, so that $\langle \alpha_i, \omega_j^{\vee} \rangle = \delta_{ij}$. 
Then  the dual map $\eta: T \longrightarrow T_H$ is  given by
\[  \eta: \omega_0^{\vee}(u) \cdot\omega_1^{\vee}(t_1)  \cdot\omega_2^{\vee}(t_2)  \cdot\omega_3^{\vee}(t_3)
\mapsto  \left( \begin{array}{ccc}
u t_1t_2t_3 & & \\
& u &  \\
& & u^{-2}t_1^{-1} t_2^{-1} t_3^{-1}  \end{array} \right). \] 
The kernel of $\eta$ is thus
\[  {\rm Ker}(\eta) = \{  \omega_1^{\vee}(t_1)  \cdot\omega_2^{\vee}(t_2)  \cdot\omega_3^{\vee}(t_3): t_1t_2t_3 =1 \}  = {}^{1-\theta}T \simeq  \mathbb{G}_m^2. \]
We are not sure if these explicit formulas are especially attractive, but perhaps they provide a certain level of psychological comfort and may certainly be useful for explicit computations.

\vskip 5pt
\subsection{\bf Fibers of unramified transfer}
The unramified spectral transfer defined by the twisted endoscopic character identity established in Proposition \ref{P:unram-spec-T}  realizes the functorial lifting  of unramified L-packets associated to $\xi$. Since unramified L-packets are parametrized by semsimple conjugacy classes in the dual group, this unramified spectral transfer is thus specififed by the induced map
\[  \xi :  T_H^{\vee}(\C) / W(T_H^{\vee}) \longrightarrow   T^{\vee}(\C) / W(T^{\vee}) \]
on Weyl-orbits on maximal (dual) tori.   
\vskip 5pt

The following proposition determines the nonempty fibers of the map $\xi$.
\vskip 5pt

\begin{prop} \label{P:fibers}
Suppose that $[s]$ and $[s']$ are two semisimple conjugacy classes of $\PGL_3(\C)$ such that $\xi(s) = \xi(s')$. Then one has
$[s'] = [s]$ or $[s'] = [s]^{-1}$. In particular, 
\[  \Theta_s \circ {\rm trans}_{\SL_3} = \Theta_{s'} \circ {\rm trans}_{\SL_3} \Longleftrightarrow  \text{ $[s'] = [s]$ or $[s'] = [s]^{-1}$.} \]
\end{prop}
\vskip 5pt

\begin{proof}
 If $\xi(s) = \xi(s')$, then under the adjoint representation of $\PGL_3(\C)$, the elements ${\rm Ad}(s)$ and ${\rm Ad}(s')$ of $\SO_8(\C)$ have the same set of eigenvalues.
The eigenvalues of ${\rm Ad}(s)$  are precisely the value $\gamma(s)$ of the 6 roots $\gamma$ of $\PGL_3(\C)$  on $s$, together with the eigenvalue 1 (with multiplicity 2). Now if we group the 6 roots in 3 pairs (by grouping $\pm \gamma$ together), then the Weyl group $S_3$ of $\PGL_3(\C)$ acts on the 3 pairs  of roots by permutation. By applying an appropriate  Weyl element, we may thus assume that for any pair of roots $\pm \gamma$, 
\[  \{ \gamma(s), \gamma(s)^{-1} \} = \{ \gamma(s') , \gamma(s')^{-1} \}. \]
We need to show that this forces $[s'] = [s]$ or $[s^{-1}]$.
\vskip 5pt

Suppose that $\alpha$ and $\beta$ are the two simple roots, with $\alpha + \beta$ the highest root. We consider two cases:
\vskip 5pt
\begin{itemize}
\item[(a)]  If
\[ \alpha(s') = \alpha(s)^{\epsilon} \quad \text{and} \quad \beta(s') = \beta(s)^{\epsilon} \]
for a fixed $\epsilon = \pm 1$, then we conclude that $s' = s^{\epsilon}$, as desired. 
\vskip 5pt

\item[(b)]  If
\[ \alpha(s') = \alpha(s) \quad \text{and} \quad \beta(s') = \beta(s)^{-1}, \]
then we have
\[  (\alpha + \beta) (s') = \alpha(s) \cdot \beta(s)^{-1}. \]
But we also know that
\[ (\alpha+\beta)(s') = (\alpha + \beta)(s) \quad \text{or} \quad  (\alpha + \beta)(s')^{-1}. \]
This implies, for the two respective cases,  that 
\[  \beta(s) = \epsilon \quad \text{or} \quad  \alpha(s) = \epsilon \]
for some $\epsilon = \pm 1$.  But then, in the respective case,  we have
\[  \beta(s') = \beta(s)^{-1} = \beta(s) \quad \text{or} \quad \alpha(s') = \alpha(s) = \alpha(s)^{-1}. \]
In any case, we are back to the scenario (a). 
\end{itemize}
This proves the proposition.
\end{proof}
\vskip 5pt
We can provide another perspective on the above proposition. As we noted in Proposition \ref{P:lt}(iv), the map $\xi: \PGL_3(\C) \rightarrow \Spin_8(\C)$ gives rise to a homomorphism of Hecke algebras
\[ \xi^*:  \mathcal{H}_{\PGSO_8} \longrightarrow  \mathcal{H}_{\SL_3}. \]
Indeed, via the Satake isomorphism, we have
\[  \mathcal{H}_{\PGSO_8}  \simeq  {\rm R}[\Spin_8(\C)] \quad \text{and} \quad  \mathcal{H}_{\SL_3} \simeq {\rm R}[\PGL_3(\C)] \]
where ${\rm R}[G]$ denotes the representation ring of finite-dimensional algebraic representations of $G$. 
The map $\xi^*$ is then given by the restriction of representations via $\xi$. 
Now ${\rm R}[\Spin_8(\C)]$ is generated by the fundamental representations, namely its  three 8-dimensional representations $\{ \rho_1, \rho_2, \rho_3 \}$ and the adjoint representation $\mathfrak{so}_8$, whereas ${\rm R}[\PGL_3(\C)]$ is generated by the adjoint representation $\mathfrak{sl}_3$ as well as ${\rm Sym}^3({\rm std})$ and  
$ {\rm Sym}^3({\rm std}^{\vee})$, where ${\rm std}$ denotes the standard 3-dimensional representation of $\SL_3(\C)$.
Under restriction via $\xi$, one has
\[ 
  \rho_i \mapsto  \mathfrak{sl}_3 \quad \text{and} \quad \mathfrak{so}_8 \mapsto  \mathfrak{sl}_3 \oplus {\rm Sym}^3({\rm std}) \oplus {\rm Sym}^3({\rm std}^{\vee}). \]
 From this , we see that the restriction map $\xi^*$ is not surjective. Indeed, its image is contained in the subring  ${\rm R}[\PGL_3(\C)]^+$ of self-dual representations of $\PGL_3(\C)$.  This is the subring fixed by the nontrivial outer automorphism of $\PGL_3(\C)$. This explains why both $\chi \circ \xi^* = \chi^{-1} \circ\xi^*$ for any character $\chi: \mathcal{H}_{\SL_3} \rightarrow \C$. 
\vskip 10pt

\section{\bf Cuspidal Spectrum of $\SL_n$} \label{SS:SLn}
Since we are interested in showing the twisted endoscopic transfer from $\SL_3$ to $\PGSO_8$, it will be pertinent to recall some facts about the cuspidal automorphic spectrum of $\SL_3$. We shall consider the case of $\SL_n$ below, though we shall only be interested in the case $n=3$ in this paper. Our discussion in this section is informed by the papers of Labesse-Langlands \cite{LL}, Blasius \cite{B}, Lapid \cite{L1} and Hiraga-Saito \cite{HiS1, HiS2}.
\vskip 5pt

\subsection{\bf Local L-packets} \label{SS:Lpackets}
Let $F$ be a local field. The L-packets of $\SL_n(F)$ are defined as the set of irreducible constituents of the restriction of an irreducible representation $\pi$ of $\GL_n(F)$. 
This restriction is known to be semisimple and multiplicity-free \cite{GK} and we denote the corresponding L-packet of $\SL_n(F)$ by $\mathcal{L}[ \pi] \subset {\rm Irr}(\SL_n(F))$. 
 Moreover,   for irreducible representations $\pi$ and $\pi'$ of $\GL_n(F)$, one has 
 \[  \mathcal{L}[\pi'] \cap  \mathcal{L}[\pi]  \ne \emptyset \]
 if and only if 
\[    \text{$\pi' \simeq  \pi \otimes (\chi \circ \det)$ for some character $\chi$ of $F^{\times}$,} \]
in which case one has $\mathcal{L}[\pi'] = \mathcal{L}[\pi] \subset {\rm Irr}(\SL_n(F))$. We shall also write $\mathcal{L}[\pi]$ for the multiplicity-free semisimple representation given by the direct sum of the members of $\mathcal{L}[\pi]$.
 \vskip 5pt

\vskip 5pt

In particular,  any unramified L-packet of $\SL_n(F)$ is of the form $\mathcal{L}[\pi]$ for $\pi$ an unramified representation of $\GL_n(F)$.  If the Satake parameter of $\pi$ is the semisimple conjugacy class $c(\pi) \in \GL_n(\C)$, then that for $\mathcal{L}[\pi]$ is the image of $c(\pi)$ under the natural  projection map
\[  p: \GL_n(\C) \longrightarrow \PGL_n(\C) = \SL_n^{\vee}. \]
More generally,  if $\phi_{\pi}$ is the L-parameter of $\pi \in {\rm Irr}(\GL_n(F))$, then the L-parameter of $\mathcal{L}[\pi]$ under  the local Langlands correspondence for $\SL_n(F)$ is given by $p \circ \phi_{\pi}$. 
\vskip 5pt

\subsection{\bf Endoscopy}
The internal structure of a local  L-packet $\mathcal{L}[\pi]$ can be probed via the theory of endoscopy. For this purpose, let us explicate the endoscopic data for $\SL_3$, since this is the main case of interest for us. 
\vskip 5pt

\begin{lemma}
Up to isomorphism, the local endoscopic data for $\SL_3$  are given by pairs $(E/F, s)$ where
\begin{itemize}
\item $E/F$ is a cyclic (Galois) cubic field extension;
\item $s \in \PGL_3(\C) = \SL_3^{\vee}$ is one of the following two elements
\begin{equation} \label{E:s}
   \left( \begin{array}{ccc}
1 & & \\
& \zeta & \\
& & \zeta^2 \end{array} \right) \quad \text{or} \quad 
 \left( \begin{array}{ccc}
1 & & \\
& \zeta^2 & \\
& & \zeta \end{array} \right).  \end{equation}
\end{itemize}
For a given $(E/F, s)$, the associated endoscopic group is the norm-one torus
\[ T_E := {\rm Ker}( N_{E/k}: {\rm Res}_{E/k}(\mathbb{G}_m)\longrightarrow \mathbb{G}_m) \]
and the associated L-embedding
\[  \xi_{E/F,s}: {^L}T_E\simeq  T_E^{\vee} \rtimes {\rm Gal}(E/F)  \longrightarrow \PGL_3(\C) \]
 sends $T_E^{\vee}$ isomorphically to the diagonal torus of $\PGL_3(\C)$ and ${\rm Gal}(E/F)$ isomorphically to the cyclic group generated by the order 3 permutation matrices (in one of the two possible ways).  

 \end{lemma}
In particular, for each cyclic cubic extension $E/F$, the norm-one torus $T_E$ occurs as an endoscopic group two times; in other words, there are two transfer maps 
\[  {\rm trans}_{T_E, s}: \mathcal{I}(\SL_3) \longrightarrow \mathcal{SI}(T_E) = \mathcal{I}(T_E) \]
corresponding to the two choices for $s$.  
\vskip 5pt

\vskip 5pt

\subsection{\bf Endoscopic local L-packets} \label{SS:endoL}
For general $n$,  it has been shown by Henniart-Herb \cite{HH}, Hiraga-Saito \cite{HiS1, HiS2} and Hiraga-Ichino \cite{HI}  that 
the L-packets of $\SL_n(F)$, defined in the above manner, satisfies the expected endoscopic character identities, as dictated by the theory of endoscopy. 
Hence, the span of the Harish-Chandra characters of the members of $\mathcal{L}[\pi]$ 
contains a 1-dimensional subspace of stable distribution, given by  the  sum of  all the characters of all members of  $\mathcal{L}[\pi]$ and hence is simply given by the Harish-Chandra character $\Theta_{\pi}$ of $\pi$. We shall refer to $\Theta_{\pi}$ as 
the stable trace of the L-packet $\mathcal{L}[\pi]$ and denote it by $S\Theta_{[\pi]}$.
\vskip 5pt

Let us illustrate with the case of endoscopic local L-packets of $\SL_3$.
These endoscopic local L-packets are those whose L-parameter factors through some $\xi_{E/F,s}$.
They are precisely the $\mathcal{L}[\pi]$ such that $\pi \otimes \omega_{E/F} \simeq \pi$ for some cyclic cubic extension $E/F$ with associated cubic character $\omega_{E/F}^{\pm 1}$. In particular, the L-parameter of $\pi$ is of the form ${\rm Ind}_{W_E}^{W_F} \chi$ for some character $\chi$ of $E^{\times}$, in which case we will write $\pi =\pi_{\chi}$. 
\vskip 5pt

Let us assume that $\chi$ is not fixed by  ${\rm Gal}(E/F)$, so that $\pi_{\chi}$ is supercuspidal.  For a generic choice of $\chi$,    $\mathcal{L}[\pi_{\chi}]$ is the sum of 3 irreducible representations, say $\pi_1$, $\pi_2$ and $\pi_3$.
 Then the local endoscopic character identities expresses ${\rm trans}_{T_E,s}(\chi)$ (for the two choices of $s$) as a linear combination of the Harish-Chandra characters of the $\pi_i$'s. More precisely, one has
 \[ \begin{cases}
  {\rm trans}_{T_E,s}(\chi) = \Theta_{\pi_1} + \zeta \cdot \Theta_{\pi_2} + \zeta^2 \cdot  \Theta_{\pi_3};    \\
  {\rm trans}_{T_E,s^{-1}}(\chi) = \Theta_{\pi_1} + \zeta^2 \cdot \Theta_{\pi_2} + \zeta \cdot  \Theta_{\pi_3}.     
  \end{cases} \]
 These two invariant distributions, together with the stable trace $S\Theta_{[\pi_{\chi}]}$, give a basis for the 3-dimensional space of invariant distributions spanned by the characters of the $\pi_i$'s.
\vskip 5pt

 \subsection{\bf Cuspidal spectrum} \label{SS:cuspidal}
Now let $k$ be a number field. If $\pi \simeq \bigotimes'_v\pi_v\subset \mathcal{A}_{cusp}(\GL_n)$ is a cuspidal automorphic representation of $\GL_n$,  
the abstract restriction of $\pi$ to $\SL_n(\A)$ gives rise to a multiplicity-free semisimple representation of $\SL_n(\A)$, all of whose irreducible summands are nearly equivalent with each other. This is a global L-packet of $\SL_n$ and we shall denote this by $\mathcal{L}[\pi]$, noting that
\[  \mathcal{L}[\pi]  \simeq \bigotimes'_v \mathcal{L}[\pi_v]. \]
 On the other hand, one has the restriction of automorphic forms
\[ {\rm rest}: \mathcal{A}_{cusp}(\GL_n) \longrightarrow \mathcal{A}_{cusp}(\SL_n). \]
Then the image $L[\pi] := {\rm rest}(\pi)$ is a  submodule of both $\mathcal{A}_{cusp}(\SL_n)$ and $\mathcal{L}[\pi]$ as an $\SL_n(\A)$-module. 
\vskip 5pt

As the restriction map ${\rm rest}$ is surjective, one can hope to understand the cuspidal spectrum of $\SL_n$ in terms of that of $\GL_n$. For this purpose, following Lapid \cite{L1}, we consider the following three equivalence relations on cuspidal representations of $\GL_n$: 
\begin{definition}
Let $\pi$ and $\pi'$ be cuspidal representations of $\GL_n$.
\vskip 5pt

\noindent (a)  Write  $\pi {\thicksim}_s \pi' $ if and only if
\[
  \text{ $\pi' \simeq  \pi \otimes (\chi \circ \det)$ for some character $\chi$ of $F^{\times} \backslash \A^{\times}$,} \]
in which case one has
\[  \pi' =   \pi \otimes (\chi \circ \det) \subset \mathcal{A}_{cusp}(\GL_n). \]
We shall denote the equivalence class of $\pi$ under this equivalence relation by $[\pi]_s$.
\vskip 5pt

\noindent (b)  Write  $\pi {\thicksim}_{ew} \pi'$ if and only if
 \[  \text{$\pi' \simeq \pi \otimes (\chi \circ \det)$  for some character $\chi$ of $\A^{\times}$.} \]
 We shall denote the equivalence class of $\pi$ under this equivalence relation by $[\pi]_{ew}$.
\vskip 5pt

\noindent (c)  Write   $\pi {\thicksim}_{w} \pi'$ if and only if 
  \[  \text{$\pi'$ is nearly equivalent to $\pi \otimes (\chi \circ \det)$  for some character $\chi$ of $\A^{\times}$.} \]
  In other words, 
\[ \text{$\pi'_v \simeq \pi_v \otimes (\chi_v \circ \det)$  for some character $\chi_v$ outside a finite set $S$ of places of $k$.} \]
We shall denote the equivalence class of $\pi$ under this equivalence relation by $[\pi]_w$.
 \end{definition}
Clearly, the equivalence relation gets weaker as one goes down the list. When $n=2$, it is a result of Ramakrishnan \cite{Ra1} that the three equivalence relations above are in fact the same. 
For $n \geq 3$, Blasius \cite{B} has shown that they are distinct from each other.  The relevance of the three equivalence relations to the cuspidal spectrum of $\SL_n$ is addressed  by the following proposition \cite{L1, HiS1, HiS2}:
\vskip 5pt

\begin{prop}
Suppose that $\pi$ and $\pi'$ are cuspidal representations of $\GL_n$. Then we have:
\vskip 5pt
\noindent (i) $\pi {\thicksim}_s \pi' $ if and only if 
\[  0 \ne L[\pi'] \cap  L[ \pi]   \subset \mathcal{A}_{cusp}(\SL_n), \]
 in which case one has $L[\pi'] = L[\pi] $.    
 
 \vskip 5pt
 
 \noindent (ii)  $\pi {\thicksim}_{ew} \pi'$ if and only if
 \[  \text{$L[ \pi'] \simeq L[\pi] $ as abstract $\SL_n(\A)$-modules.} \]

 \vskip 5pt
 
 \noindent (iii) $\pi {\thicksim}_{w} \pi'$ if and only if 
 \[ \text{$L[\pi]$ and $L[\pi']$ have  nearly equivalent irreducible summands.} \]
   \end{prop}
 Observe that when $\pi$ and $\pi'$ satisfies (iii) but not (ii),  then $\mathcal{L}[\pi]$ and $\mathcal{L}[\pi']$ have no isomorphic irreducible summands, because their local components at some places  belong to different local L-packets, and hence are inequivalent. 
 In other words, the so-called strong multiplicity one property (or more appropriately, rigidity) does not hold for $\SL_n$, at the level of L-packets.
\vskip 5pt
As a consequence of the above proposition, one obtains the following decomposition of the cuspidal spectrum of $\SL_n$:
 \[
 \mathcal{A}_{cusp}(\SL_n) = \bigoplus_{ [\pi]_s}  L[\pi]    =   \bigoplus_{[\pi]_{ew}} \left(  \bigoplus_{[\pi']_s {\thicksim}_{ew} [\pi]_s}   L[\pi'] \right)  \simeq   \bigoplus_{[\pi]_{ew}} m_{\mathcal{L}[\pi]} L[\pi]. 
  \]
 Here, the first sum runs over the ${\thicksim}_s$-equivalence classes $[\pi]_s$ of cuspidal representations of $\GL_n$, whereas the subsequent sums run over the ${\thicksim}_{ew}$-equivalence classes $[\pi]_{ew}$. Moreover,  $m_{\mathcal{L}[\pi]}$ is called the {\it multiplicity} of the L-packet $\mathcal{L}[\pi]$ and is equal to  the number of ${\thicksim}_s$-equivalence classes in the ${\thicksim}_{ew}$-equivalence class of $\pi$. In particular, $m_{\mathcal{L}[\pi]}$ gives the multiplicity of the irreducible summands of $L[\pi]$ in the cuspidal spectrum of $\SL_n$. 
 \vskip 5pt
 
 We may  also write
 \[  \mathcal{A}_{cusp}(\SL_n)
   = \bigoplus_{[\pi]_w}\left(  \bigoplus_{ [\pi']_s {\thicksim}_w [\pi]_s}  L[\pi'] \right). \]
where the outer sum runs over the ${\thicksim}_w$-equivalence classes of cuspidal representations of $\GL_n$ and the inner sum runs over the ${\thicksim}_s$-equivalence classes in the ${\thicksim}_w$-equivalence class of $\pi$. This is a decomposition of the cuspidal spectrum of $\SL_n$ into near equivalence classes of representations of $\SL_n(\A)$. Following Lapid, we define the {\it global multiplicity} of $\mathcal{L}[\pi]$ by:
\begin{equation} \label{E:globalmult}
 \mathcal{M}_{\mathcal{L}[\pi]} = \sum_{ [\pi']_s {\thicksim}_w [\pi]_s}  m_{\mathcal{L}[\pi']}. \end{equation}

\vskip 10pt

\subsection{\bf Stable trace formula}
Let us now specialize to the case $n=3$.
One can enhance the above understanding of $\mathcal{A}_{cusp}(\SL_3)$ by considering the  stable trace formula of $\SL_3$. We give a very brief impressionistic sketch of this. 
\vskip 5pt

Analogous to the local case, the elliptic endoscopic data of $\SL_3$ are indexed by pairs $(E/k, s)$ where $E/k$ is a cyclic cubic extension and $s$ is one of the two elements of $\PGL_3(\C)$ given in (\ref{E:s}).
For a given pair $(E/k,s)$, the corresponding elliptic endoscopic group of $\SL_3$  is the anisotropic norm-one  torus
\[ T_E := {\rm Ker}( N_{E/k}: {\rm Res}_{E/k}(\mathbb{G}_m)\longrightarrow \mathbb{G}_m).\]
 Then the stable trace formula of $\SL_3$ takes the form
\[  
I^{\SL_3}_{disc} (f) =   S_{disc}^{\SL_3}(f)   + \frac{1}{9} \cdot \sum_{(E/k,s)}  I_{disc}^{T_E} ( f_{T_E,s}) \]
for test functions $f$ on $\SL_3(\A)$ with transfer $f_{T_E,s}$ on $T_E(\A)$ relative to $(E/k,s)$.  Here, $I_{disc}^{\SL_3}$ and $I_{disc}^{T_E}$ refers to the discrete part of the invariant trace formula for $\SL_3$ and $T_E$ respectively. On the other hand, $S_{disc}^{\SL_3}$ is a stable distribution on $\SL_3(\A)$. 
\vskip 5pt

Suppose that $c = \{ c_v \}_{v \notin S} \subset \PGL_3(\C)$ is a Hecke-Satake family, i.e. a collection of semisimple conjugacy classes in  $\PGL_3(\C)$ for all $v$ outside a finite set $S$ of places of $k$. Then $c$ determines a character 
\[   \chi_c:  \mathcal{H}_{\SL_3}^S:= \bigotimes_{v \notin S} \mathcal{H}_{\SL_3(k_v)} \longrightarrow \C \]
of the spherical Hecke algebras outside  $S$ and hence a near equivalence class of representations of $\SL_3(\A)$. 
Considering the $c$-isotypic part of the stable trace formula, i.e. where the Hecke algebra $\mathcal{H}_{\SL_3}^S$ acts by the character $\chi_c$, we obtain the more refined identity:
\[  I_{disc,c }^{\SL_3} (f_S)   = S_{disc,c}^{\SL_3}(f_S)  +  \frac{1}{9} \cdot \sum_{(E/k,s)}  I_{disc,c}^{T_E} ( {\rm trans}_{E/k,s}(f_S)),   \quad \text{for $f_S \in \bigotimes_{v \in S} C^{\infty}_c(\SL_3(k_v))$,} 
\]
where $I_{disc,c }^{\SL_3}$ and $S_{disc,c}^{\SL_3}$ denote the $c$-isotypic part of the relevant distributions, i.e. where the Hecke algebra $\mathcal{H}_{\SL_3}^S$ acts by the character $\chi_c$.

\vskip 5pt

\subsection{\bf Endoscopic L-packets}
A cuspidal global L-packet $\mathcal{L}[\pi]$ of $\SL_3$ is endoscopic if its Hecke-Satake family $c(\pi)$ satisfies
\[  I_{disc,c}^{T_E} \circ  {\rm trans}_{E/k,s} \ne 0 \quad \text{ for some $(E/k, s)$.} \]
Such an  L-packet is thus the endoscopic transfer of a nontrivial automorphic character of $T_E$. 
\vskip 5pt

To describe these endoscopic L-packets more explicitly, let us start with a nontrivial automorphic  character 
$\chi$ of $\A_E^{\times}$ and consider   the cuspidal representation $\pi_{\chi}$ of $\GL_3$ obtained by automorphic induction from $\chi$ (see \S \ref{SS:AI} below for a more detailed discussion of automorphic induction).  The cuspidal representations of $\GL_3$ which are obtained by automorphic induction from a cyclic cubic extension are characterized as those which are invariant under twisting by an automorphic (nontrivial) cubic character (corresponding to the cyclic cubic extension in question). 
Then the endoscopic L-packet of $\SL_3$ determined by $\chi|_{T_E}$ is simply $\mathcal{L}[\pi_{\chi}]$, so that its contribution to the cuspidal spectrum of $\SL_3$ is the submodule $L[\pi_{\chi}]$. 
\vskip 5pt

 Observe that the property of a global L-packet $\mathcal{L}[\pi]$ being endoscopic depends only on the ${\thicksim}_{w}$-equivalence class of the cuspidal representation $\pi$ of $\GL_3$. Indeed, if $\pi' {\thicksim}_w \pi$, 
 then there is an abstract character $\mu$ of $\A^{\times}$ such that $\pi'$ and $\pi \otimes \mu$ are nearly equivalent. Hence, if $\pi$ is invariant under twisting by a nontrivial cubic automorphic character $\nu$, then $\pi' \otimes \nu$ and $\pi'$ are nearly equivalent. 
 By the strong multiplicity or rigidity property, we must have $\pi' \otimes \nu \simeq \pi'$, so that $\pi'$ is obtained as an automorphic induction from the same cubic field extension as $\pi$.   
 \vskip 5pt
 
\subsection{\bf Non-endoscopic L-packets}  
On the other hand, suppose that  $\mathcal{L}[\pi]$ is {\it not} endoscopic, i.e. 
\[   I_{disc,c(\pi)}^{T_E} \circ  {\rm trans}_{E/k,s} \equiv 0 \quad \text{for any $(E/k,s)$.} \]
In this case, one has
\[   S_{disc,c(\pi)}^{\SL_3} = I_{disc,c(\pi) }^{\SL_3}  =  \bigoplus_{[\pi']_s {\thicksim}_w [\pi]_s} L[\pi']. \]
Since this distribution is stable, we deduce that 
\[  L[\pi'] \simeq \mathcal{L}[\pi'] \quad \text{for each $[\pi']_s {\thicksim}_w [\pi]_s$,} \]
 so that  each member of the global L-packet $\mathcal{L}[\pi']$ is cuspidal automoprhic.
 Hence we deduce that
\begin{equation} \label{E:SMF}
 S_{disc,c(\pi) }^{\SL_3}   = \bigoplus_{[\pi']_s  {\thicksim}_w [\pi]_s}  S\Theta_{[\pi']}, \end{equation}
where
 \[
 S\Theta_{[\pi']} =  \bigotimes_v  S\Theta_{[\pi'_v]}  =  \Theta_{\pi'}|_{\SL_3(\A)} \quad \text{ if $\pi' \simeq \otimes'_v \pi'_v$.} \]
 
\vskip 5pt

\subsection{\bf Stable multiplicity formula} By  the stable trace formula recounted above, the local endoscopic character identities mentioned in \S \ref{SS:endoL} and the description of the cuspidal spectrum of $\SL_3$ given in \S \ref{SS:cuspidal}, one can obtain 
the same formula (\ref{E:SMF}) for the endoscopic L-packets. As a consequence, one has
an explicit  spectral decomposition of the cuspidal part of $S_{disc}^{\SL_3}$, namely
\begin{equation} \label{E:SMF2}
  S_{cusp} ^{\SL_3}  = \bigoplus_{ [ \pi]_s}   S\Theta_{[\pi]},  \end{equation}
where the sum runs over the ${\thicksim}_s$-equivalence classes $[\pi]_s$ of cuspidal automorphic representations $\pi$ of $\GL_3$.
This is sometimes called the {\it stable multiplicity formula} for $\SL_3$.
\vskip 5pt

\vskip 10pt

\section{\bf Adjoint lifting of $\GL_3$} \label{S:adjoint-T}
In this section, we exploit the stable twisted trace formula for the twisted space  $(\PGSO_8,\theta)$ to show the adjoint lifting from $\GL_3$ to $\GL_8$.
Recall from the introduction that the adjoint representation ${\rm Ad}: \GL_3(\C)  \longrightarrow \GL_8(\C)$ factors as:
\[ \begin{CD}
 \GL_3(\C) @>>> \PGL_3(\C) @>\tilde{{\rm Ad}}>> \Spin_8(\C) @>>> \SO_8(\C)  @>{\rm std}>> \GL_8(\C), \end{CD} \]  
The functoriality relative to all other arrows being known, our purpose in this section is to establish the functorial lifting with respect to $\tilde{\rm Ad}$. 

\vskip 5pt

 \subsection{\bf Global K-forms of $\PGSO_8$} \label{SS:global-K}
As we noted in \S \ref{SS:local-K},   a slight complication with the theory of stable (twisted) trace formula is the need to take into account of K-forms \cite[\S I.1.11]{MW1}, which are certain inner forms of the twisted space in question. 
    In the context of $(\PGSO_8,\theta)$, we have seen there that the K-forms correspond to isomorphism classes of octonion algebras over local or number fields.
 Thus, over  a non-Archimedean local field  or $\C$, $(\PGSO_8, \theta)$ has no additional $K$-forms,  since there is a unique octonion algebra up to isomorphism. On the other hand,  over $\R$, there are two octonions algebras: the split one and a division algebra. As such, there are two K-forms of $(\PGSO_8, \theta)$ over $\R$. For a number field $k$,
with $r$ real places,  the number of such octonion algebras is $2^r$, so that there are $2^r$ global K-forms over $k$.   

\vskip 5pt

\subsection{\bf Stable twisted trace formula} \label{SS:STTF}
The stable twisted trace formula for $(\PGSO_8, \theta)$ takes the form  \cite[Thm. X.8.1, Pg 1241]{MW2}
\begin{equation} \label{E:stf}
\sum_{\mathbb{O}}  I^{\mathbb{O}, \theta}_{disc}(f_{\mathbb{O}})  =  S^{{\rm G}_2}_{disc}( f_{{\rm G}_2}) + \frac{1}{2} S_{disc}^{\SO_4}(f_{\SO_4})  +  S_{disc}^{\SL_3}(f_{\SL_3}), \end{equation}
where $\{ f_{\mathbb{O}} \in C^{\infty}_c(\PGSO^{\mathbb{O}}_8(\A))\}$ is a family of  test functions on the K-forms  and $f_{{\rm G}_2}$, $f_{\SO_4}$ and $f_{\SL_3}$ are its twisted endoscopic transfer to the twisted endoscopic groups ${\rm G}_2$, $\SO_4$ and $\SL_3$ respectively.   For each twisted endoscopic group $H$, the stable distribution $S^H_{disc}$ is  the discrete part of the stable trace formula for $H$. 

\vskip 5pt

Let us explicate  the distributions  $ I^{\mathbb{O}, \theta}_{disc}$, following \cite[\S 14.3]{LW} and \cite[\S X.5.1]{MW2}.
Writing $G = \PGSO^{\mathbb{O}}_8$, one has the expression  (see \cite[Thm. 14.3.1 and Prop. 14.3.2]{LW} and \cite[Pg. 125-128]{A}).
\[ 
  I^{\mathbb{O}, \theta}_{disc} (f) = \sum_M   
\frac{1}{|W(G,M)|}  I^{\mathbb{O}, \theta}_{disc,M} (f) \quad \text{ for  $f \in C^{\infty}_c(G(\A))$, } \]
where the sum runs over the conjugacy classes of Levi subgroups $M$ of $G$ and \cite[\S X.5.1]{MW2}
\[    I^{\mathbb{O}, \theta}_{disc,M}(f)  = 
 \sum_{s  \in W(G \cdot \theta, M)_{reg} }  |\det (1- s | \mathfrak{a}_M)|^{-1} \cdot  {\rm Tr} ( M_{P|s(P)}(0)\cdot  \rho_{P, s, disc}(f ) ). \] 
 Here:
 \vskip 5pt
 
 \begin{itemize}
 \item[-]  $P = M \cdot U$ is a standard parabolic subgroup with  Levi subgroup $M$ and unipotent radical $U$; 
 \vskip 5pt
 \item[-] $W(G \cdot \theta, M)$ is the Weyl set of elements of $G(F) \cdot \theta$ which normalizes $M$ (taken modulo $M(F)$), and $W(G \cdot \theta, M)_{reg}$ is the subset of regular elements, i.e. those $s$ such that $\det (1- s  | \mathfrak{a}_M) \ne 0$. 
 \vskip 5pt
 \item we regard elements $s \in G(F) \cdot \theta$ as outer automorphisms of $G$;
 \vskip 5pt
 
 \item[-]  $\rho_{P, s , disc}$ refers to the  representation of the twisted space $G(\A) \cdot \theta$ on 
 \[ \mathcal{A}_{disc} (   U(\A) M(k) Z(M)(k_{\infty})^0 \backslash G(\A)) \]
 with the action of  $\gamma \cdot \theta \in G(\A) \cdot \theta$  given  as the composite:
 \[ \begin{CD}
 \mathcal{A}_{disc} (   U(\A) M(k) Z(M)(k_{\infty})^0 \backslash G(\A))  \\
 @VVr_s(\gamma \cdot \theta)V \\
  \mathcal{A}_{disc} (   s (U(\A)) M(k) Z(M)(k_{\infty})^0 \backslash G(\A))  \\
  @VVM_{P, s(P)}(0)V \\
   \mathcal{A}_{disc} (   U(\A) M(k) Z(M)(k_{\infty})^0 \backslash G(\A)) \end{CD}  \]
 where
 \[  (r_s(\gamma \cdot \theta)\phi)(g) =  \phi ( s^{-1}(g \gamma)).\]
 and $M_{P,s \cdot P}$ is a standard intertwining operator.
 \vskip 5pt
 
 \item for $f \in C^{\infty}_c(G(\A))$, one has
 \[ \rho_{P,s, disc} (f) = \int_{G(\A)} f(\gamma) \cdot \rho_{P,s, disc}(\gamma \cdot \theta) \, d\gamma. \]
 \end{itemize}
 
  In particular, when  $M = G$, one has $s = \theta$ and
  \[ \rho_{G, \theta, disc}(\gamma \cdot \theta) \phi(g) = \phi (  \theta^{-1} (g \gamma)), \quad \text{ for $\phi \in \mathcal{A}_{disc}(G)$.} \]
  Thus,
 \[  I^{\mathbb{O}, \theta}_{disc,G} (f) = {\rm Tr}( f \cdot \theta | \mathcal{A}_{disc}(G)), \]
 where the action of $f \in C^{\infty}_c(G(\A))$  on  $\mathcal{A}_{disc}(G))$ is the usual one and that of
  $\theta$ is given by
 \[  \theta (\phi)(g) = \phi(\theta^{-1}(g)). \]
 On the other hand, for $M \ne G$, the representations of $G(\A)$  which intervene in $I^{\mathbb{O}, \theta}_{disc,M}$ are the $\theta$-stable representations  of the form
 \[ {\rm Ind}_{P(\A)}^{G(\A)}  \sigma, \]
 where $\sigma \subset \mathcal{A}_{disc, \chi}(M)$ is a $s$-stable submodule with unitary central character  $\chi$, for some $s \in W(G \cdot \theta, M)_{reg}$.
 \vskip 5pt

  \vskip 10pt

\vskip 5pt

\subsection{\bf Adjoint lifting.}
Using the stable twisted trace formula recalled above, we can now show the following theorem:
\vskip 5pt

\begin{thm} \label{T:main lifting}
Let  $\pi$ be a cuspidal representation of $\GL_3$ such that $\pi_{v_0}$ is a discrete series representation for some place $v_0$. Then the (weak) adjoint lifting of $\pi$ exists as an automorphic representation ${\rm Ad}(\pi)$ of  $\GL_8$. Moreover, if $\pi_v$ is unramified with Satake parameter $c(\pi_v) \in \GL_3(\C)$, then ${\rm Ad}(\pi)_v$ is unramified with Satake parameter ${\rm Ad}(c(\pi_v)) \in \GL_8(\C)$.
\end{thm}
\vskip 5pt

\begin{proof}

   To use the stable twisted trace formula  (\ref{E:stf}), we need to specify the test function $f_{\mathbb{O}} = \prod_v f_{\mathbb{O},v}$ used there, for each octonion $k$-algebra $\mathbb{O}$. We do this as follows:
   
   \vskip 5pt
\begin{itemize}
\item  At a real place $v$, we let the test functions $f_{\mathbb{O},v}$ be arbitrary (so that they may depend on $\mathbb{O}$). 
\vskip 5pt

\item If $v$ is a non-Archimedean distinct from $v_0$ or a complex  place of $k$, we will choose $f_{\mathbb{O},v} = f_v$ to be independent of $\mathbb{O}$ but otherwise arbitrary.
\vskip 5pt

\item  At the place $v_0$,  we take a a pseudo-coefficient $f_{\pi_{v_0}}$ of $\pi_{v_0}$ and consider its stable orbital  integral  which is a nonzero element of 
$\mathcal{SI}_{cusp}(\SL_3)$. By Proposition \ref{P:lt}(i), one then pick $f_{v_0}  \in  \mathcal{I}^{\theta}_{cusp}(\PGSO_8)$ such that 
\begin{equation}  \label{E:transfer}
 {\rm trans}([f_{v_0}])  = (0 ,0, [f_{\pi_{v_0}}]). \end{equation}
We then take $f_{\mathbb{O}, v_0} = f_{v_0}$ for any $\mathbb{O}$.
 \end{itemize}
\vskip 5pt

With such a family of test functions  \{$f_{\mathbb{O}} \}$, (\ref{E:stf}) simplifies to:
\begin{equation} \label{E:STF1}\
 \sum_{\mathbb{O}} I^{\mathbb{O}, \theta}_{disc}(f_{\mathbb{O}})  =   S_{disc}^{\SL_3}(f_{\SL_3}), \end{equation}
where
\[  f_{\SL_3} = \sum_{\mathbb{O}} {\rm trans}^{\mathbb{O}}_{\SL_3} ( f_{\mathbb{O}}) = \left(   \prod_{\text{real $v$}} \left( \sum_{\mathbb{O}_v}  {\rm trans}^{\mathbb{O}_v}_{\SL_3}(f_{\mathbb{O},v}) \right) \right)  \cdot \prod_{\text{other $v$}}  {\rm trans}_{\SL_3} ( f_v),
 \]
with the sum over $\mathbb{O}_v$ running over the two octonion algebras at a real place $v$.
\vskip 5pt

Now we may use the action of the spherical Hecke algebras of $G = \PGSO_8$  at almost all places of $k$ to isolate a given Hecke-Satake family on the LHS of (\ref{E:STF1}). 
More precisely, with our given  cuspidal representation $\pi$ of $\GL_3$ with associated submodule $L[\pi] \subset \mathcal{A}_{cusp}(\SL_3)$, let $S$ be the finite set of places (containing all Archimedean places and the place $v_0$) such that $\pi_v$ is unramified if and only if  $v \notin S$. Considering  the Hecke-Satake family 
\[ c ( [\pi]) = \{ c([\pi_v]) : v \notin S \} \subset \PGL_3(\C),\]
 let us set 
\[    c := \{   \tilde{\rm Ad} ( c([\pi_v])): v \notin S \}  \subset \Spin_8(\C) \]
where $\tilde{\rm Ad}:  \PGL_3(\C) \hookrightarrow \Spin_8(\C)$ is the adjoint map. 
At each Archimedean place $v$, we may also consider the infinitesimal character $\lambda([\pi_v])$ of $\mathcal{L}[\pi_v]$, which may be regarded as a semisimple conjugacy class in the complexified dual Lie algebra $\mathfrak{sl}_3$,  giving rise to  
\[  \lambda([\pi]) = ( \lambda([\pi_v]))_{v \in S_{\infty}} \in \mathfrak{sl}_3(k \otimes_{\Q} \C), \]
where $S_{\infty}$ denotes the set of Archimedean places of $k$. 
 We set
\[  \lambda := ( \tilde{\rm Ad}_*(\lambda([\pi_v])) )_{v \in S_{\infty}} \in  \mathfrak{so}_8( k \otimes_{\Q} \C), \]
By \cite[\S X.5.9]{MW2}, one has;
\begin{equation}  \label{E:MW5.9}
    \sum_{\mathbb{O}}  I_{disc, S, c, \lambda}^{\mathbb{O}, \theta} (f_{\mathbb{O} ,S})    = \sum_{ [\pi']_s}  S^{\SL_3}_{disc, [\pi']_s} (f_{\SL_3,S})  \end{equation}
where 
\[ f_{\mathbb{O},S}  \in  \left( \bigotimes_{v \in S} C^{\infty}_c(\PGSO_8^{\mathbb{O}_v}) \right) \otimes \left( \bigotimes_{v \notin S} 1_{K_v} \right) \]
with $1_{K_v}$ the characteristic function of $K_v= \PGSO_8(\mathcal{O}_v)$,  and  
the sum on the RHS ranges over all $[\pi']_s$ whose restriction  to $\SL_3$ is unramified outside $S$, and such that  
\[ \tilde{\rm Ad} ( c([\pi']))  =  c \quad \text{and} \quad  \tilde{\rm Ad}_*(\lambda[\pi']) = \lambda.
\]
We shall denote the set of such $[\pi']_s$ by $\Xi_{S, c, \lambda}$.
\vskip 5pt

Now for each  $[\pi'] \in \Xi_{S, c, \lambda}$, it follows by Proposition \ref{P:fibers} that, for each $v \notin S$,  the local L-packets
 $\mathcal{L}[\pi'_v]$ and $\mathcal{L}[\pi_v]$ are either equal or contragredient of each other; in particular, $\mathcal{L}[\pi'_v]$ is a generic L-packet for each $v \notin S$.  Moreover, since 
 $f_{v_0}$ lies in $\mathcal{I}^{\theta}_{cusp}(\PGSO_8)$, any $[\pi']_s \in \Xi_{S, c, \lambda}$  is necessarily contained in the cuspidal spectrum $\mathcal{A}_{cusp}(\SL_3)$. 
 \vskip 5pt
 
 By the stable multiplicity formula  (\ref{E:SMF2})) for $\SL_3$,  (\ref{E:MW5.9}) becomes:
\[
 \sum_{\mathbb{O}}  I_{disc,S, c, \lambda}^{\mathbb{O}, \theta} (f_{\mathbb{O} ,S})  = 
 \sum_{[\pi']_s \in \Xi_{S, c, \lambda}}  {\rm Tr} \left( \pi'_S(f_{\SL_3,S}) \right). \]
 Now the RHS is the pullback by ${\rm trans}_{\SL_3,S}$ of the  character distribution  of a stable admissible representation of  $\prod_{v \in S}\SL_3(k_v)$.  
 By Lemma \ref{L:nonzero} (or rather the semi-local version of it), one deduces that the  
  RHS of the above identity defines a nonzero linear functional on $\otimes_{v \in S} C^{\infty}_c(\PGSO_8^{\mathbb{O}_v}(k_v))$. 
This implies that the LHS is a nonzero linear functional and in particular that 
\[ \sum_{\mathbb{O}}  I_{disc, S, c, \lambda}^{\mathbb{O}, \theta}  \ne 0. \]
This shows that the pair  $(c, \lambda)$   can be detected in the spectrum of the $I_{disc}^{\mathbb{O},\theta}$ for some $\mathbb{O}$ and hence can be realized by an automorphic subrepresentation $\Pi_{\mathbb{O}}$ of $\PGSO_8^{\mathbb{O}}$ for some $\mathbb{O}$. Moreover, such a $\Pi_{\mathbb{O}}$  is unramified outside of $S$.
\vskip 5pt

Pulling $\Pi_{\mathbb{O}}$ back to $\SO_8$ via $f_1: \SO_8^{\mathbb{O}} \rightarrow \PGSO_8^{\mathbb{O}}$ and
  transfering of $f_1^*(\Pi_{\mathbb{O}})$ to $\GL_8$ (appealing to, for example, \cite{CZ} or \cite{CFK} for the case when $\mathbb{O}$ is not split), one obtains
the desired weak adjoint lifting ${\rm Ad}(\pi)$ of $\pi$. More precisely,  by its construction, ${\rm Ad}(\pi)$ is unramified outside of $S$ with Satake oparameter ${\rm Ad}(c(\pi_v))$ for all $v \notin S$, and with infinitesimal character ${\rm Ad}_*(\lambda[\pi_{\infty}]) \in \mathfrak{gl}_8(k \otimes_{\Q} \C)$.
\end{proof}
\vskip 5pt

\subsection{\bf General case} \label{SS:general}
In fact, it is not  too hard to remove the hypothesis that some local component of $\pi$ is discrete series in Theorem \ref{T:main lifting}, using arguments in \cite[\S 3.5, esp. Prop. 3.5.1]{A}. 
 Let us briefly indicate how the argument goes. 
 \vskip 5pt
 
 As the reader will notice, the main reason for imposing the local condition in Theorem \ref{T:main lifting} is to ensure that only one twisted elliptic endoscopic group (namely $\SL_3$) survives on one side of the stable twisted trace formula of $(\PGSO_8, \theta)$. This ensures that when localizing at a Hecke-Satake family $c$, there is a nonzero contribution since there is no possibility of cancellation of the $c$-contribution from $\SL_3$ with the $c$-contribution from the other twisted endoscopic groups. 
In the general case,  there is a priori a  possibility of cancellation. It is partly due to the nature of the terms in the trace formula (which could introduce negative coefficients) and partly due to the fact that the local transfer of test functions is not always surjective (as we noted in Cor. \ref{C:image}). 
\vskip 5pt

 On the other hand,  \cite[Prop. 3.5.1]{A}  is designed to ensure that such cancellation can be avoided, as long as one could ensure that the $c$-contributions from each of the three twisted endoscopic groups are  (locally finite) non-negative linear combinations of character distributions  of automorphic representations.  To put oneself in a position to apply \cite[Prop. 3.5.1]{A}, we first note:
 \vskip 5pt
 
 \begin{prop} \label{P:shin-takanashi}
 Let $\sigma$ be an  irreducible automorphic representation of any elliptic twisted endoscopic group $H$ of $(\PGSO_8, \theta)$ which satisfies one of the following:
 \vskip 5pt
 \begin{itemize}
 \item $\sigma$ is a subquotient of ${\rm Ind}_P^H \tau$, where $P$ is a parabolic subgroup of $H$ and $\tau$ is an automorphic representation of the Levi factor of $P$;
 \vskip 5pt
 
 \item $\sigma$ is obtained as an endoscopic lift (from an elliptic endoscopic group of $H$).
 \end{itemize}
 Then the (weak) twisted endoscopic lifting of $\sigma$ to $\PGSO_8$ exists (as an automorphic representation). 
  \end{prop}
The proof of this proposition is not difficult but will be deferred to a joint work with Shin and Takanashi, since it involves some  case-by-case considerations and a discussion of endoscopic lifting for $G_2$, for example, which will take us too far afield.
\vskip 5pt

Let us assume the above proposition. Suppose now that $\pi$ is a cuspidal representation of $\GL_3$. 
For the purpose of showing the desired adjoint lifting,  we may assume (by the proposition) that 
$L[\pi]$ is not endoscopic, so that $\pi$ is not an automorphic induction from a cyclic cubic extension. Indeed, independent of the proposition, one knows that if $\pi$ is an automorphic induction from a cyclic cubic extension, then ${\rm Ad}(\pi)$ exists (as we explain in \S \ref{S:AI}). As in the proof of Theorem \ref{T:main lifting}, let us set
\[    c := \{   \tilde{\rm Ad} ( c([\pi_v])): v \notin S \}  \subset \Spin_8(\C) \]
and consider the $c$-isotypic part of the stable twisted trace formula for $(\PGSO_8,\theta)$:
\[   \sum_{\mathbb{O}} I^{\mathbb{O}, \theta}_{disc,c} (f_{\mathbb{O}}) = \sum_H  S^H_{dics,c}(f_H). \]
To establish the desired adjoint lifting, it suffices to show that the distribution $\sum_{\mathbb{O}} I^{\mathbb{O}, \theta}_{disc,c}$ is nonzero, and for this, we shall examine the RHS.
\vskip 5pt

Let us examine the stable distribution $S^H_{disc}$  for each of the elliptic twisted endoscopic groups $H$. One can write:
\[  S_{disc}^H   = {\rm Tr}( - | \mathcal{A}_{disc}[H]) + \mathcal{E}_H \]
where $ \mathcal{E}_H$ can be explicated (as we did for the twisted trace formula of $(\PGSO_8,\theta)$ in \S \ref{SS:STTF}). 
 Considering  the $c$-isotypic part, one has: 
 \[  S_{disc,c}^H   = {\rm Tr}( - | \mathcal{A}_{disc,c}(H)) + \mathcal{E}_{H,c}. \]
The main fact we need to know about $\mathcal{E}_H$ is that its spectral support consists entirely of automorphic representations of the type highlighted in Proposition \ref{P:shin-takanashi}. 
 \vskip 10pt
 
 We now examine two cases:
 \vskip 5pt
 
 \begin{itemize}
 \item Suppose that $\mathcal{E}_{H,c}$ is nonzero for some $H$. In this case, there is an automorphic representation $\sigma$ in the spectral support of $\mathcal{E}_H$ such that
 \[  c = \{   \tilde{\rm Ad} ( c([\sigma_v])): v \notin S \}  \subset \Spin_8(\C). \]
 Such a $\sigma$ is of type considered in Proposition \ref{P:shin-takanashi} and hence has an automorphic lift $\Pi =\tilde{\rm Ad}_*(\sigma)$ on   $\PGSO_8$ (by the proposition). The transfer of $\Pi$ to  $\GL_8$ is then the desired adjoint lifting of $\pi$.
 \vskip 5pt
 
 \item Suppose that $\mathcal{E}_{H,c}$ is zero for all $H$. Then one has
 \[   \sum_{\mathbb{O}} I^{\mathbb{O}, \theta}_{disc,c} (f_{\mathbb{O}}) = \sum_H    {\rm Tr}( f_H | \mathcal{A}_{disc,c}(H)). \]
Now the RHS  is a (locally finite) nonnegative linear combination of  character distributions of (square-integrable) automorphic representations.
 On the other hand,  the contribution of the term $H = \SL_3$ is nonzero (as we argued in the proof of Theorem \ref{T:main lifting}). 
 Hence, \cite[Prop. 3.5.1]{A} implies  that the RHS is nonzero (i.e. no cancellation occurs), and hence $ \sum_{\mathbb{O}} I^{\mathbb{O}, \theta}_{disc,c}$ is nonzero, as desired. 
 \end{itemize}
\vskip 5pt
Thus, taking for granted Proposition \ref{P:shin-takanashi}, we have shown the adjoint lifting for $\GL_3$  without the local hypothesis in Theorem \ref{T:main lifting}.

\section{\bf Automorphic Induction Case}. \label{S:AI}
The next few sections of the paper are devoted to the analysis of ${\rm Ad}(\pi)$ and in particular to the proof of Theorem \ref{T:intro3}.
In this section, we examine the case when the cuspidal representation $\pi$ is obtained by automorphic induction from a Hecke character of a cubic field extension $E/k$. In particular, we shall see that the adjoint lifting ${\rm Ad}(\pi)$ always exists (without the local condition in Theorem \ref{T:main lifting}. Moreover, we shall be able to determine the possible isobaric decomposition of ${\rm Ad}(\pi)$.
\vskip 5pt

\subsection{\bf Cubic fields} Let $E/k$ be a cubic field extension, so that $E/k$ gives rise to a conjugacy class of homomorphsms $\alpha_E: {\rm Gal}(\overline{k}/k) \longrightarrow S_3$. 
Composing $\alpha_E$ with the sign character of $S_3$ gives a quadratic character, which is trivial if and only if $E/k$ is Galois (or normal). The \'etale quadratic 
algebra determined by $({\rm sign}) \circ \alpha_E$ is the discriminant quadratic algebra $K_E$ of $E/k$.  Composing $\alpha_E$ with the two-dimensional irreducible representation of $S_3$, one obtains a two-dimensional representation $\rho_E$  of ${\rm Gal}(\overline{k}/k)$ which is irreducible if and only if $E/k$ is non-Galois.  We note that
\[  {\rm Ind}_{W_E}^{W_k} 1 = 1 \oplus \rho_E. \]
\vskip 5pt

\subsection{\bf Automorphic induction} \label{SS:AI}
 Given a Hecke character $\chi: E{^\times} \backslash \A_E^{\times} \longrightarrow \C^{\times}$, one obtains a $1$-dimensional representation (still denoted by $\chi$) 
\[ \chi: W_E \longrightarrow W_E^{ab} \simeq E^{\times} \backslash \A_E^{\times}  \longrightarrow \C^{\times} \]
of the Weil group $W_E$ via  class field theory. By induction, this gives rise to a 3-dimensional representation
\[  \rho_{\chi} := {\rm Ind}_{W_E}^{W_k} \chi  \quad \text{of $W_k$.} \]
We say that $\pi$ is  obtained from $\chi$ by automorphic induction if its  Hecke-Stakae family $c(\pi)$ is equal to the family of Frobenius classes of $\rho_{\chi}$, and express this as $\pi = {\rm AI}_{E/k}(\chi)$. 

\vskip 5pt

If $E/k$ is an arbitrary  Galois extension with cyclic Galois group of prime order, then the existence of ${\rm AI}_{E/k}(\chi)$ has been shown by Arthur-Clozel \cite{AC} using the twisted trace formula. When $E/k$ is non-Galois cubic, the existence of ${\rm AI}_{E/k}(\chi)$ has been shown by Jacquet-Piatetski-Shapiro-Shalika \cite[\S 14]{JPSS} using the converse theorem.
There is also an alternative approach by Kazhdan via exceptional theta correspondence \cite{K}. 
\vskip 5pt

When $\pi ={\rm AI}_{E/k}(\chi)$, we shall see that the adjoint lifting of $\pi$ to $\GL_8$ more directly, without recourse to the twisted trace formula. 
Indeed, from the viewpoint of Galois representations, one has:
\[  \rho_{\chi} \otimes \rho_{\chi}^{\vee} = {\rm Ind}_{W_E}^{W_k} \chi \otimes  {\rm Ind}_{W_E}^{W_k} \chi^{-1} 
= {\rm Ind}_{W_E}^{W_k}\left(  \left(  {\rm Ind}_{W_E}^{W_k} \chi \right)|_{W_E} \otimes \chi^{-1} \right). \]
We shall now consider the Galois and non-Galois case separately.
\vskip 5pt

\subsection{\bf Galois case}
 When $E/k$ is Galois,  let $\sigma$ be a generator of  ${\rm Gal}(E/k) \simeq \Z/3\Z$ and let $\{ \omega_{E/k}, \omega_{E/k}^{-1} \}$ be the pair of cubic Hecke characters of $\A^{\times}$ associated to $E/k$ by class field theory. Then
\[  \left(  {\rm Ind}_{W_E}^{W_k} \chi \right)|_{W_E}  \simeq \chi + \chi^{\sigma} + \chi^{\sigma^2} \]
so that
 \[   \rho_{\chi} \otimes \rho_{\chi}^{\vee} \simeq  {\rm Ind}_{W_E}^{W_k} 1  \oplus   {\rm Ind}_{W_E}^{W_k} (\chi^{\sigma}/\chi ) \oplus   {\rm Ind}_{W_E}^{W_k} (\chi^{\sigma^2}/\chi). \] 
Hence,
\[  {\rm Ad}(\rho_{\chi}) \simeq \omega_{E/k} \oplus \omega_{E/k}^{-1} \oplus  {\rm Ind}_{W_E}^{W_k} (\chi^{\sigma}/\chi ) \oplus  {\rm Ind}_{W_E}^{W_k} (\chi^{\sigma^2}/\chi).\]
Now observe that each of the summands above has an automorphic analog and hence we may construct the (weak) adjoint lifting of $\pi$ as an isobaric sum:
\begin{equation} \label{E:Add}
 {\rm Ad}(\pi) =  \omega_{E/k} \boxplus \omega_{E/k}^{-1} \boxplus  {\rm AI}_{E/k}( \chi^{\sigma}/ \chi) \boxplus {\rm AI}_{E/k}(\chi^{\sigma^2}/ \chi). \end{equation}
Thus, we see that in this case, the adjoint lifting of $\pi$ is non-cuspidal, as it contains (at least) two $\GL_1$-summands. Indeed, we have the following result:
\vskip 5pt

\begin{thm} \label{T:AI1}
Let $\pi$ be a cuspidal representation of $\GL_3$. Then the following are equivalent:
\vskip 5pt
\begin{itemize}

\item[(i)] $\pi$ is an automorphic induction from  some Galois cubic extension;
\vskip 5pt

\item[(ii)] ${\rm Ad}(\pi)$ exists as an automorphic representation of $\GL_8$ and contains a $\GL_1$-summand in its isobaric sum decomposition. 
\vskip 5pt

\item[(iii)] $\pi$ is invariant under twisting by some nontrivial Hecke character (necessarily cubic). 
\end{itemize}
\vskip 10pt

\noindent More precisely, when the above conditions hold, so that $\pi ={\rm AI}_{E/k}(\chi)$ for some Galois cubic extension $E/k$ with associated pair of cubic characters 
$\{ \omega_{E/k}, \omega_{E/k}^{-1}\}$, then $\pi$ is invariant under twisting by $\omega_{E/k}^{\pm 1}$ and ${\rm Ad}(\pi)$ contains $\omega_{E/k}^{\pm 1}$ as $\GL_1$-summands.
Further,  we have the following two possibilities for the isobaric sum decomposition of ${\rm Ad}(\pi)$.
\vskip 5pt

\begin{itemize}
\item[(a)]  If $\chi^{\sigma} \cdot \chi^{\sigma^2} \ne  \chi^2$,  then the characters $\chi^{\sigma} / \chi$ and $\chi^{\sigma^2}/ \chi$ are non-invariant under ${\rm Gal}(E/k)$.
In this case, 
\[ {\rm Ad}(\pi) =  \omega_{E/k} \boxplus \omega_{E/k}^{-1} \boxplus  {\rm AI}_{E/k}( \chi^{\sigma}/ \chi) \boxplus {\rm AI}_{E/k}(\chi^{\sigma^2}/ \chi) \]
with the last two summands  cuspidal representations of $\GL_3$ (which are dual of each other). 
The two $\GL_3$-summands are isomorphic to each other if and only if $(\chi^2)^{\sigma} = \chi^2$. 
In particular, $\pi$ is an automorphic induction from precisely one Galois cubic field (namely $E/k$).
\vskip 5pt

\item[(b)]   If $\chi^{\sigma} \cdot \chi^{\sigma^2} =  \chi^2$, then the characters $\chi^{\sigma} / \chi$ and $\chi^{\sigma^2}/ \chi$ are invariant under ${\rm Gal}(E/k)$. If $\chi^{\sigma}/\chi = \mu \circ N_{E/k}$ for a Hecke character $\mu$ over $k$, then 
\[   {\rm Ad}(\pi) =  \omega_{E/k} \boxplus \omega_{E/k}^{-1} \boxplus  \mu \boxplus  \mu \omega_{E/k} \boxplus \mu \omega_{E/k}^{-1}  
\boxplus  \mu^{-1} \boxplus  \mu^{-1} \omega_{E/k} \boxplus \mu^{-1} \omega_{E/k}^{-1}. \]
In particular, $\mu$ is a cubic character and ${\rm Ad}(\pi)$ is an isobaric sum of $8$ distinct cubic Hecke characters. Hence, $\pi$ can be expressed as an automorphic induction from four distinct Galois cubic field extensions of $k$.
\end{itemize}
In particular, the self-dual automorphic representation ${\rm Ad}(\pi)$   is the transfer of an automorphic but non-cuspidal representation of $\SO_8$.
 \end{thm}
\vskip 5pt

\noindent{\bf Remark}:
\vskip 5pt

\noindent (i)  If $\pi$ can be obtained from two Galois cubic field extensions $E_1$ and $E_2$, then $\pi$ is invariant under twisting by $\omega_{E_1/k}$ and $\omega_{E_2/k}$. Hence $\pi$ is also invariant under twisting by $\omega_{E_1/k} \cdot \omega_{E_2/k}$ and $\omega_{E_1/k} \cdot \omega_{E_2/k}^{-1}$.  The latter two cubic characters correspond to two other Galois cubic extensions $E_3$ and $E_4$. The four cubic extensions $E_i$ are precisely the four cubic subfields contained in the bi-cyclic-cubic ($=\Z/3\Z \times \Z/3\Z$) extension which is the composition of $E_1$ and $E_2$.  Hence, if $\pi$ is an automorphic induction from more than one Galois cubic extension, it is automatically an automorphic induction from four Galois cubic extensions. Part  (b) of the theorem implies  that $\pi$ cannot be an automorphic induction from more than four Galois cubic extensions.
\vskip 5pt

\noindent (ii) The condition on $\chi$  in (b) of the Theorem  is the same condition in \cite[\S 6.1, Prop. 2 (1)] {L1}, which implies that the global multiplicity of the L-packet $\mathcal{L}[\pi]$ of $\SL_3$ is $2$. Under the condition in (a), the global multiplicity of $\mathcal{L}[\pi]$ is $1$.

\vskip 5pt

\begin{proof}
Let us first show the equivalence of (i)-(iii):
\vskip 5pt

\begin{itemize}
\item (i) $\Longrightarrow$ (ii): this has been shown by our discussion before the statement of the theorem.
\vskip 5pt

\item (ii) $\Longrightarrow$ (iii): Suppose ${\rm Ad}(\pi)$ contains $\mu$ as a $\GL_1$-summand. Then the partial L-function $L^S(s, {\rm Ad}(\pi) \times \mu^{-1})$ has a pole at $s = 1$. But one has an identity of partial L-functions:
\[  L^S(s, \mu^{-1}) \cdot L^S(s, {\rm Ad}(\pi) \times \mu^{-1}) = L^S(s, \pi \times (\pi \otimes \mu)^{\vee}). \]
Hence the partial L-function on the RHS also has a pole at $s=1$. By Jacquet-Shalika, this implies that $\pi \simeq \pi \otimes \mu$, as desired. 
Further, $\mu$ is necessarily cubic, since $\omega_{\pi} = \omega_{\pi \otimes \mu} = \omega_{\pi} \cdot \mu^3$, where $\omega_{\pi}$ denotes the central character of $\pi$.
\vskip 5pt

\item  (iii) $\Longrightarrow$ (i): this is due to Arthur-Clozel \cite{AC}
\end{itemize}
This proves the equivalence of (i), (ii) and (iii). Moreover, when these conditions hold, we see by (\ref{E:Add}) that if $\pi ={\rm AI}_{E/k}(\chi)$, then $\omega_{E/k}^{\pm 1}$ are $\GL_1$-summands of ${\rm Ad}(\pi)$. Indeed, (\ref{E:Add}) gives the more precise description in (a) and (b) immediately. 

\end{proof}

\vskip 5pt

\subsection{\bf Non-Galois case}
We now consider the non-Galois case. 
If $E/k$ is a non-Galois cubic extension, with discriminant quadratic field extension $K = K_E$, then $L = E \cdot K_E$ is the Galois closure of $k$ in $\overline{k}$. Hence,  $L/k$ is an $S_3$-extension, and the cubic extension $L/K$ is Galois. We let $\sigma$ be a generator of ${\rm Gal}(L/K) \simeq \Z/3\Z$ and note that $L$ contains two other cubic subfields besides $E$, namely 
\[ E_{\sigma} =\{ \sigma(e) : e \in E \} \quad \text{and} \quad  E_{\sigma^2} = \{ \sigma^2(e): e \in E\}. \]
\vskip 5pt

In this case, if $\chi$ is a Hecke character of $E$, then 
\[ \left( {\rm Ind}_{W_E}^{W_k} \chi \right)|_{W_E} \simeq \chi \oplus {\rm Ind}_{W_L}^{W_E} (\chi|_{W_L}^{\sigma}). \]
Here, note that $\chi|_{W_L}^{\sigma}$  is  defined by $\chi^{\sigma}(w) = \chi(\sigma^{-1}w \sigma)$  for  $w \in W_L$. To simplify notation, we shall omit $|_{W_L}$ in what follows and regard the restriction to $W_L$ as implicit.
Hence
\[  \rho_{\chi} \otimes \rho_{\chi}^{\vee} \simeq  {\rm Ind}_{W_E}^{W_k} 1  \oplus  {\rm Ind}_{W_L}^{W_k}  (\chi^{\sigma}/ \chi),\]
so that
\[  {\rm Ad}(\rho_{\chi}) \simeq  \rho_E \oplus {\rm Ind}_{W_L}^{W_k}  (\chi^{\sigma}/ \chi). \]
Noting that $\rho_E \simeq {\rm Ind}_{W_K}^{W_k} \omega_{L/K}$ and 
\[  {\rm Ind}_{W_L}^{W_k}  = {\rm Ind}_{W_K}^{W_k}  \circ {\rm Ind}_{W_L}^{W_K} \]
is a composite of inductions through cyclic extensions of prime degrees, we see that the automorphic analog of ${\rm Ad}(\rho_{\chi})$ can be constructed.

\vskip 5pt

More precisely, if $\pi = {\rm AI}_{E/k}(\chi)$, then
\begin{equation} \label{E:Add2}
  {\rm Ad}(\pi)  =  {\rm AI}_{K/k}(\omega_{L/K})  \boxplus {\rm AI}_{L/k}\left(  \chi_L^{\sigma} / \chi_L \right) \end{equation}
as an automorphic representation of $\GL_8$, where we have written $\chi_L := \chi \circ N_{L/E}$ for the base change of $\chi$ from $E$ to $L$.
Henceforth, we will write $\rho_E$ to denote its automorphic analog $ {\rm AI}_{K/k}(\omega_{L/K})$ as well.

\vskip 5pt

Now one has the following analog of Theorem \ref{T:AI1}:
\vskip 5pt

\begin{thm} \label{T:AI2}
Let $\pi$ be a unitary cuspidal representation of $\GL_3$.  The following are equivalent:
\vskip 5pt

\begin{itemize}
\item[(i)] $\pi$ can be realized as an automorphic induction from a non-Galois cubic extension.
\vskip 5pt

\item[(ii)]  ${\rm Ad}(\pi)$ exists as an automorphic representation of $\GL_8$ and contains a unitary cuspidal $\GL_2$-summand in its isobaric  decomposition. 
\vskip 5pt

\item[(iii)]  for some unitary cuspidal representation $\tau$ of $\GL_2$, $\pi \boxtimes \tau  \supset \pi$ in its isobaric decomposition (where $\pi \boxtimes \tau$ is the Rankin-Selberg lifting from $\GL_2 \times \GL_3$ to $\GL_6$ established by Kim-Shahidi \cite{KS1}). 
\end{itemize}
\vskip 10pt

More precisely, if $\pi ={\rm AI}_{E/k}(\chi)$ for a non-Galois cubic extension $E/k$ with Galois closure $L/k$ and discriminant quadratic field $K/k$, then ${\rm Ad}(\pi)$ contains the cuspidal representation $\rho_E$ of $\GL_2$   as a summand in its isobaric decomposition.  Further, we have the following possibilities for the isobaric decomposition of ${\rm Ad}(\pi)$:
\vskip 5pt

\begin{itemize}
\item[(a)]  If 
\[ \chi_L^{\sigma} \cdot \chi_L^{\sigma^2}  \ne  \chi_L^2 \quad \text{and} \quad   (\chi_L^{\sigma})^2 \ne \chi_L^2, \]
 then ${\rm AI}_{L/k}\left( \chi_L^{\sigma}/ \chi_L \right)$ is a self-dual cuspidal representation of $\GL_6$ over $k$ with central character $\omega_{K/k}$, so that
 \[ {\rm Ad}(\pi) = \rho_E \oplus   {\rm AI}_{L/k}\left( \chi_L^{\sigma}/ \chi_L \right). \]
 
\vskip 5pt

\item[(b)]  If 
\[  (\chi_L^{\sigma})^2 = \chi_L^2, \]
 then  ${\rm AI}_{L/k}\left( \chi_L^{\sigma}/ \chi_L \right)$ 
 is a self-dual cuspidal representation of $\GL_3$ over $K$ which is invariant under ${\rm Gal}(K/k)$, so that 
\[  {\rm AI}_{L/k}\left( \chi_L^{\sigma}/ \chi_L \right) = A \oplus A  \cdot \omega_{K/k} \]
for some self-dual cuspidal representation $A$ of $\GL_3$ with trivial central character and
\[ {\rm Ad}(\pi) = \rho_E \oplus  A \oplus A \cdot \omega_{K/k}. \]

\item[(c)]  If 
\[ \chi_L^{\sigma} \cdot \chi_L^{\sigma^2}  =  \chi_L^2, \]
 then 
\[  {\rm AI}_{L/K}\left( \chi_L^{\sigma}/ \chi_L \right)   = \omega_1 \boxplus \omega_2 \boxplus \omega_3 \]
is the isobaric sum of three distinct cubic characters, defining three cyclic cubic extensions $E_1$, $E_2$ and $E_3$ of $K$, such that 
$E$ and the $E_i$'s are the four cubic extensions in  the bi-cyclic-cubic extension $E \cdot E_i$ of $K$, and each  $E_i/k$ is an $S_3$-extension.
In this case, 
\[ {\rm Ad}(\pi) = \rho_E \oplus \rho_{E_1} \oplus \rho_{E_2} \oplus \rho_{E_3} \]
and $\pi$ can be obtained by automorphic induction from four  distinct non-Galois cubic extensions.  
\end{itemize}

In particular,  ${\rm Ad}(\pi)$ is the transfer of a globally generic discrete (indeed cuspidal) automorphic  representation of $\SO_8$.
\end{thm}
\vskip 5pt

\begin{proof}
We first show the equivalence of (i)-(iii). 
\vskip 5pt
\begin{itemize}
\item (i) $\Longrightarrow$ (ii): this  follows by discussion before the theorem, in particular (\ref{E:Add2}). 
\vskip 5pt

\item (ii) $\Longrightarrow$ (iii):  if ${\rm Ad}(\pi)$ contains a $\GL_2$ cuspidal summand $\tau$ in its isobaric decomposition, then the partial L-function $L^S(s, {\rm Ad}(\pi) \times \tau^{\vee})$ has a pole at $s=1$.  Now we have the L-function identity:
\[  L^S(s, \pi \times \pi^{\vee} \times \tau^{\vee}) = L^S(s, \tau^{\vee}) \cdot L^S(s, {\rm Ad}(\pi) \times \tau^{\vee}), \]
where all L-functions above are of Langlands-Shahidi type; indeed, the L-function on the LHS occurs when one considers the maximal parabolic subgroup in $E_6$ corresponding to the branch vertex in the Dynkin diagram.
From this identity, we deduce that the LHS has a pole at $s=1$. But
\[  L^S(s, \pi \times (\pi \boxtimes \tau)^{\vee})) =  L^S(s, \pi \times \pi^{\vee} \times \tau^{\vee}). \]
Hence $L^S(s, \pi \times (\pi \boxtimes \tau^{\vee}))$ has a pole at $s=1$, which implies that $\pi \otimes \tau$ contains $\pi$ in its isobaric decomposition.
\vskip 5pt

\item (iii) $\Longrightarrow$ (i): 
We shall make use of a cuspidality criterion for the Rankin-Selberg lifting from $\GL_2 \times \GL_3$ to $\GL_6$ shown by Ramakrishnan-Wang \cite{RW}. 
Suppose that $\pi \boxtimes \tau$ contains $\pi$ in its isobaric decomposition.  By \cite[Thm. 3.1]{RW}, this can happen only if 
$\pi$ is an automorphic induction from a non-Galois cubic extension of $k$.
\end{itemize}
\vskip 5pt

It follows from (\ref{E:Add2}) that when $\pi ={\rm AI}_{E/k}(\chi)$ for some non-Galois cubic extension $E/k$, then  ${\rm Ad}(\pi)$ contains the cuspidal $\GL_2$-summand $\rho_E$. 
To determine the possible isobaric decomposition of ${\rm Ad}(\pi)$, it remains to investigate the term 
\[  {\rm AI}_{L/k}\left( \chi_L^{\sigma}/ \chi_L \right) = {\rm AI}_{K/k} \left( {\rm AI}_{L/K} \left( \chi_L^{\sigma}/ \chi_L \right) \right) \quad \text{in (\ref{E:Add2}).}  \]
We examine the mutually exclusive cases (a)-(c) in the theorem in turn:
\vskip 5pt

\begin{itemize}
\item[(a)] Under the two conditions in (a),  ${\rm AI}_{L/K} \left( \chi_L^{\sigma}/ \chi_L \right)$ is a cuspidal representation of $\GL_3$ over $K$, and is not invariant under ${\rm Gal}(K/k)$. Hence
$ {\rm AI}_{L/k}\left( \chi_L^{\sigma}/ \chi_L \right) $ is a self-dual cuspidal representation of $\GL_6$ over $k$, which is self-dual with central character $\omega_{K/k}$ (since ${\rm Ad}(\pi)$ is self-dual with trivial central character).
\vskip 5pt

\item[(b)] Under the conditions in (b), ${\rm AI}_{L/K} \left( \chi_L^{\sigma}/ \chi_L \right)$ is a self-dual cuspidal representation of $\GL_3$ over $K$ with trivial central character, which  is invariant under ${\rm Gal}(K/k)$.  
This implies that $ {\rm AI}_{L/k}\left( \chi_L^{\sigma}/ \chi_L \right)$ has the shape indicated in (b).
\vskip 5pt

\item[(c)] Under the conditions in (c), ${\rm AI}_{L/K} \left( \chi_L^{\sigma}/ \chi_L \right)$  is the isobaric sum of three automorphic characters $\omega_i$, satisfying
\[  (\omega_i)_L = \chi_L^{\sigma} / \chi_L., \]
so that $\omega_2 = \omega_1 \cdot \omega_{L/K}$ and $\omega_3 = \omega_1 \cdot \omega_{L/K}^{-1}$ (without loss of generality).
Each of the $\mu_i$'s is a cubic character; this is because $\mu_i$ occurs as a summand in the isobaric decomposition of ${\rm Ad}(\pi_K)$, where $\pi_K$ denotes the base change of $\pi$ to $K$, which is still cuspidal.
With $E_i/K$ denoting the Galois cubic extensions defined by $\omega_{E_i/K}$, it is clear from the above relations between $\omega_i$ and $\omega_{L/K}$ that $E$ and the $E_i$'s are the four cubic subfields in the bi-cyclic-cubic field $E \cdot E_i$ (alternatively, one can appeal to Theorem \ref{T:AI1}, applied to $\pi_K$. To see that $E_i$ is non0Galis over $k$, we observe that if $c$ is the non-trivial element in ${\rm Gal}(L/E)$ (so that $c|_K$ is the non-trivial element in ${\rm Gal}(K/k)$),   then
\[  (\chi_L^{\sigma} / \chi_L )^c = \chi_L^{\sigma_2} / \chi_L =  (\chi_L^{\sigma} / \chi_L )^{-1}, \]
under the hypothesis of (c).  This implies that the conclusion of (c). 
  \end{itemize}
\end{proof}
\vskip 5pt

\begin{cor}
A cuspidal representation of $\GL_3$ cannot be simultaneously obtained as 
 an  automoprhic induction from a Galois cubic extension and also an automorphic induction from a non-Galois cubic extension.

\end{cor}
\vskip 5pt

\vskip 10pt

\section{\bf Self-Dual Case}
In this section, we consider the case when the cuspidal representation $\pi$ of $\GL_3$ is  the twist of a self-dual cuspidal representation.
In this case, there is no loss of generality in assuming that $\pi$ is self dual  and hence is the adjoint lift ($=$ Gelbart-Jacquet lift) of a cuspidal representation $\tau$ of $\GL_2$ \cite{GJ}. 
We shall write $\pi ={\rm ad}_2(\tau)$.

\vskip 5pt

From a Galois theoretic point of view, if $\tau$ is associated with a 2-dimensional Galois representation $\rho$, then ${\rm ad}_2(\tau)$ is associated to the 3-dimensional Galois representation ${\rm ad}_2(\rho) := {\rm Sym}^2(\rho) \otimes \det(\rho)^{-1}$.  Then
\[ {\rm Ad}({\rm ad}_2(\rho) )  \simeq  {\rm ad}_2(\rho) \oplus {\rm ad}_4(\rho) \]
where
\[ {\rm ad}_4(\rho) = {\rm Sym}^4(\rho) \otimes \det(\rho)^{-2}. \]
Now the ${\rm Sym}^4$-lifting from $\GL_2$ to $\GL_5$ has been shown by H. Kim \cite{Ki}, and hence the automorphic analog of  ${\rm ad}_4(\rho)$ is  known to exist; we shall denote it by ${\rm ad}_4(\tau)$. In particular, this allows us to construct the adjoint lifting of $\pi ={\rm ad}_2(\tau)$ as:
\begin{equation} \label{E:Add3}
 {\rm Ad}(\pi)= \pi \boxplus {\rm ad}_4(\tau) = {\rm ad}_2(\tau) \boxplus {\rm ad}_4(\tau). \end{equation}
Now we have:
\vskip 5pt
\begin{thm} \label{T:selfdual}
Let $\pi$ be a unitary cuspidal representation of $\GL_3$ and assume that $\pi$ is not an automorphic induction from a cubic field extension of $k$ (Galois or not). 
Then the following are equivalent:
\vskip 5pt

\begin{itemize}
\item[(i)]  $\pi$ is the twist by an automorphic character of a self-dual cuspidal representation;
\vskip 5pt

\item[(ii)]  ${\rm Ad}(\pi)$ exists as an automorphic representation of $\GL_8$ and contains a cuspidal $\GL_3$-summand in its isobaric decomposition;
\end{itemize}
More precisely, when the above conditions hold with $\pi$ self-dual, the isobaric decomposition of ${\rm Ad}(\pi)$ is of type $3 +5$ and is given by (\ref{E:Add3}). 
In particular,  ${\rm Ad}(\pi)$ is the transfer of a globally generic discrete (indeed cuspidal) automorphic  representation of $\SO_8$.
\end{thm}

\begin{proof}
Assuming (i), then (ii) follows by (\ref{E:Add3}). 
Now assume (ii) and suppose that ${\rm Ad}(\pi)$ contains a cuspidal $\GL_3$-summand $A$ in its isobaric decomposition. By the hypotheses on $\pi$ and Theorems \ref{T:AI1} and \ref{T:AI2},  we see that the isobaric decomposition of ${\rm Ad}(\pi)$ is of type $3+5$, so that 
\[ {\rm Ad}(\pi) = A  \boxplus B \]
with $B$ a cuspidal $\GL_5$-summand.
Since ${\rm Ad}(\pi) = {\rm Ad}(\pi)^{\vee}$ is necessarily self-dual with trivial central character, we deduce that $A$ and $B$ are both self-dual with the same central character. 
As $3$ and $5$ are odd, the self-dual representations $A$ and $B$ must be of orthogonal type.
\vskip 5pt

By Corollary \ref{C:rep}, we see that
\begin{equation} \label{E:exteriorcube}
  L^S(s, {\rm Ad}(\pi)  \times {\rm Ad} (\pi)) = L^S(s, {\rm Ad}(\pi)) \cdot L^S(s, {\rm Ad}(\pi), \wedge^3). \end{equation}
  Here,  note that all the partial L-functions  in this identity are  Langlands-Shahidi L-functions. Indeed, the exterior cube L-function for $\GL_8$ can be obtained as a Langlands-Shahidi L-function by considering the maximal parabolic subgroup 
  of $E_8$ with Levi factor of type $A_7$.  Since ${\rm Ad}(\pi)$ has two  cuspidal summands in its isobaric decomposition, the LHS has a pole of order $2$ at $s=1$. Hence 
$L^S(s, {\rm Ad}(\pi), \wedge^3)$ has a pole of order $2$ at $s=1$, since $L^S(s, {\rm Ad}(\pi))$ is entire and nonvanishing at $s=1$.
\vskip 5pt

On the other hand, as ${\rm Ad}(\pi) = A \oplus B$, $L^S(s, {\rm Ad}(\pi), \wedge^3)$ is equal to:
\begin{equation} \label{E:wedge3} 
 L^S(s, A, \wedge^3) \cdot L^S(s, A \times B, \wedge^2 \times {\rm std}) \cdot L^S(s, A \times B, {\rm std} \times \wedge^2) \cdot L^S(s, B, \wedge^3). \end{equation}
Now we note that:
\begin{itemize}
\item The L-function
\[  L^S(s, A, \wedge^3) = L^S(s, \omega_{A}), \]
where $\omega_{A}$ is the central character of $A$, has a pole at $s=1$ if and only if $\omega_{A}$ is trivial;

\item the L-function
\[  L^S(s, A \times B, \wedge^2 \times {\rm std}) = L^S(s, (A \cdot \omega_{A}) \times B) \]
is holomorphic  at $s=1$;

\item the L-function
\[  L^S(s, A \times B, {\rm std} \times \wedge^2) \]
has at most a simple pole at $s=1$ by a result of Kim-Shahidi \cite[Thm. 1(b) and Remark 1]{KS3}.

\item the L-function
\[ L^S(s, B, \wedge^3) = L^S(s, B, \wedge^2 \times \omega_{B}) \]
is a twisted exterior square L-function and does not have a pole at $s=1$ (since $5$ is odd).
\end{itemize}
Taken together, we see that the product of L-functions in (\ref{E:wedge3}) has at most a double pole at $s=1$ and this double pole is achieved ionly if $\omega_{A} = 1$.
Since we know that $L^S(s, {\rm Ad}(\pi), \wedge^3)$ has a double pole  at $s=1$, we deduce that $\omega_{A} = \omega_{B}$ is trivial. This shows that 
\[ A = {\rm ad}_2(\tau) \]
 is the  Gelbert-Jacquet lift of a cuspidal representation $\tau$ of $\GL_2$. Moreover, $\tau$ is non-dihedral since $A$ is cuspidal.
\vskip 5pt

It remains to show that $\pi$ is a twist of $A = {\rm ad}_2(\tau)$ by an automorphic character. For this, we note that the partial L-function
$L^S(s, {\rm Ad}(\pi) \times A)$  has a  pole at $s=1$. Now consider the L-function
\[ L^S(s, (\pi \boxtimes \tau) \times (\pi \boxtimes \tau)^{\vee}) \]
where $\pi \boxtimes \tau$ is an automorphic representation on $\GL_6$ obtained via the  Rankin-Selberg lifting from $\GL_2 \times \GL_3$ shown by Kim-Shahidi \cite{KS1}.
This L-function can be expressed as:
\begin{align}
  &L^S(s, {\rm Ad}(\pi) \times {\rm ad}_2(\tau)) \cdot L^S(s, {\rm Ad}(\pi)) \cdot  L^S(s, {\rm ad}_2(\tau) \cdot \zeta^S(s)  \notag \\
  = &L^S(s, {\rm Ad}(\pi) \times A) \cdot L^S(s, {\rm Ad}(\pi)) \cdot L^S(s,A) \cdot \zeta^S(s). \notag
  \end{align}
  We thus deduce that $ L^S(s, (\pi \boxtimes \tau) \times (\pi \boxtimes \tau)^{\vee})$ has a double pole at $s=1$; in particular, $\pi \boxtimes \tau$ is not cuspidal on $\GL_6$.
  Since $\tau$ is not dihedral, it follows by the cuspidal criterion of Ramakrishnan-Wang \cite[Thm. 3.1(a)]{RW} that $\pi$ is a twist of $A = {\rm ad}_2(\tau)$ by an automorphic character, as desired.
\end{proof}
\vskip 5pt

\begin{cor}
Let $\pi$ be a self-dual cuspidal automorphic representation of $\GL_3$ which is not an automorphic induction from a cubic field extension. If $\pi'$ is a cuspidal automorphic representation such that $\pi' {\thicksim}_w \pi$, then $\pi' {\thicksim}_s \pi$.
Hence, the global multiplicity of the cuspidal  L-packet $\mathcal{L}[\pi]$ is $1$. 
\end{cor}
\vskip 5pt

\begin{proof}
By (\ref{E:Add3}) and Theorem \ref{T:selfdual} applied to $\pi$, ${\rm Ad}(\pi)$ is automorphic on $\GL_8$ and has isobaric decomposition $\pi \boxplus \tau$ for some cuspidal $\GL_5$-summand $\tau$. By hypothesis, $\pi'$ is nearly equivalent to $\pi \otimes \chi$ for some abstract character $\chi$\ of $\A^{\times}$.
Hence,  ${\rm Ad}(\pi')$ exists as an automorphic representation on $\GL_8$, since one could simply take 
\[ {\rm Ad}(\pi') := {\rm Ad}(\pi) = \pi \boxplus  \tau. \]
Hence, by Theorems \ref{T:AI1} and \ref{T:AI2}, we see that $\pi'$ is also not an automorphic induction from a cubic field extension.  Moreover,
 by Theorem \ref{T:selfdual} applied to $\pi'$, we deduce that 
 \[  \pi' \simeq \pi_0 \otimes \mu \]
 for some self-dual cuspidal representation $\pi_0$ and an automorphic character $\mu$.  But then we would have the isobaric decomposition
 \[  {\rm Ad}(\pi')  = \pi_0 \boxplus \tau_0, \]
 for some $\GL_5$-summand $\tau_0$.  Comparing the two isobaric decompositions of ${\rm Ad}(\pi')$ obtained above, we conclude that $\pi   =\pi_0$.  In particular, 
 \[  \pi' \simeq  \pi \otimes \mu \quad \text{ so that $\pi' {\thicksim}_s \pi$.} \]
\end{proof} 

\vskip 10pt

\noindent{\bf Remark}: Suppose $\pi = {\rm ad}_2(\tau)$ is a self-dual cuspidal automorphic representation of $\GL_3$, so that $\tau$ is not dihedral. Then as noted in \cite[Thm. 3.3.7]{KS2}, one has:
\vskip 5pt

\begin{itemize}
\item $\pi$ is not an automorphic induction from a cyclic cubic extension if and only if $\tau$ is not tetrahedral, in which case ${\rm Sym}^3(\tau)$ is cuspidal;
\vskip 5pt

\item $\pi$ is not an automorphic induction from a cubic extension (cyclic or noncyclic) if and only if $\tau$ is neither tetrahedral nor octahedral, in which case both ${\rm Sym}^3(\tau)$ and 
${\rm ad}_4(\tau)$ are cuspidal.
\end{itemize}

\vskip 10pt

\section{\bf Cuspidality Criterion}
In this final section, we consider the case when the cuspidal representation $\pi$ of $\GL_3$ is not of the type considered in the previous two sections, i.e. $\pi$ is neither an automorphic induction from a cubic field extension, nor a twist of a self-dual representation. 
We shall make the following hypothesis:
\vskip 5pt

\noindent{\bf Hypothesis}: 
Assume that the adjoint lifting ${\rm Ad}(\pi)$ has been constructed as an automorphic representation of $\GL_8$ via the sequence of functorial lifting in (\ref{E:seq}). 
\vskip 5pt

\noindent By Theorem \ref{T:main lifting}, the hypothesis is satisfied when $\pi$ has a discrete series local component. 
If one is willing to take Proposition \ref{P:shin-takanashi} for granted, then the hypothesis is always satisfied.
Now one has:
 
\vskip 5pt

\begin{thm} \label{T:cuspidality}
Let $\pi$ be a unitary cuspidal automorphic representation of $\GL_3$ which is not an automorphic induction from a cubic field extension or a twist of a self-dual representation. 
Assume that the above hypothesis holds for $\pi$. Then ${\rm Ad}(\pi)$ is a self-dual automorphic representation of orthogonal type with trivial central character. Moreover,
 one of the following holds:
\vskip 5pt
\begin{itemize}
\item[(i)]  ${\rm Ad}(\pi)$ is  cuspidal  and $L^S(s, {\rm Ad}(\pi), \wedge^3)$ has a simple pole at $s=1$.
\vskip 5pt

\item[(ii)]  ${\rm Ad}(\pi) = A \boxplus B$ with $A$ and $B$ self-dual cuspidal representations of orthogonal type with the same  central character which is a nontrivial quadratic character. If $K/k$ is the quadratic field extension determined by the central character of $A$,  then
$\pi_K$ is dihedral with respect to four distinct non-Galois cubic extensions of $K$.  Moreover, $L^S(s, {\rm Ad}(\pi), \wedge^3)$ has a double pole at $s=1$.
\end{itemize}
\end{thm}
\vskip 5pt

\begin{proof}
It is clear that if ${\rm Ad}(\pi)$ exists as an automorphic representation on $\GL_8$, then it is self dual with trivial central character, regardless of whether the adjoint lifting is constructed via the sequence of liftings depicted in (\ref{E:seq}). If it is constructed via (\ref{E:seq}), then it follows that $\pi$ is of orthogonal type, since it is lifted from the split $\SO_8$.  This shows the first assertion of the Theorem.
\vskip 5pt

In view of the hypotheses on $\pi$ in the Theorem, as well as  Theorems \ref{T:AI1}, \ref{T:AI2} and \ref{T:selfdual}, we already know that ${\rm Ad}(\pi)$ cannot contain any cuspidal $\GL_1$-, $\GL_2$- or $\GL_3$-summands. Hence, the only possibilities for the isobaric decomposition of ${\rm Ad}(\pi)$ are of type $4+4$ or $8$. 
Assume thus that ${\rm Ad}(\pi) = A \boxplus B$ is an isobaric sum of two  cuspidal representations of $\GL_4$. We need to show that $A$ and $B$ are distinct self-dual of orthogonal type with nontrivial central character. 
 \vskip 5pt
 
 We will  again make use of the identity of L-functions:
 \begin{equation}  \label{E:exteriorcube2} 
  L^S(s, {\rm Ad}(\pi)  \times {\rm Ad} (\pi)) = L^S(s, {\rm Ad}(\pi)) \cdot L^S(s, {\rm Ad}(\pi), \wedge^3) \end{equation}
  given in (\ref{E:exteriorcube}). Since ${\rm Ad}(\pi)$ is self-dual, we see that the order of pole at $s=1$ of the LHS is:
  \begin{equation} \label{E:poles}
  \begin{cases} 
  2, \text{ if $A$ and $B$ are distinct;} \\
  4, \text{ if $A$ and $B$ are isomorphic.} \end{cases} \end{equation}
  On the other hand, if ${\rm Ad}(\pi) = A \boxplus B$, then $L^S(s, {\rm Ad}(\pi), \wedge^3)$ is equal to the following product:
  \[  L^S(s, A^{\vee} \times \omega_A) \cdot L^S(, B^{\vee} \times \omega_B) \cdot L^S(s, A \times B, \wedge^2 \times {\rm std})  \cdot L^S(s, B \times A, \wedge^2 \times {\rm std}), \] 
   where $\omega_A$ and $\omega_B$ denote the central characters of $A$ and $B$ respectively.  The exterior square functorial lifting from $\GL_4$ to $\GL_6$ has been shown by Kim \cite{Ki}.  Writing $\wedge^2A$ to denote the automorphic representation of $\GL_6$ which is the exterior square functorial lift of $A$, the above product can be written as:
    \[  L^S(s, A^{\vee} \times \omega_A) \cdot L^S(, B^{\vee} \times \omega_B) \cdot L^S(s, \wedge^2 A \times B)  \cdot L^S(s, \wedge^2B \times A). \]
Now the Rankin-Selberg L-functions $L^S(s, \wedge^2 A \times B)$ and $L^S(s, \wedge^2B \times A)$ have at most simple poles at $s=1$, since the isobaric decomposition of $\wedge^2A$ and $\wedge^2B$ can contain at most one $\GL_4$-cuspidal summand. 
Hence, the RHS of (\ref{E:exteriorcube2}) has at most a double pole at $s=1$. Comparing with (\ref{E:poles}), we conclude that $A$ and $B$ are distinct, and 
\[   \wedge^2 A \supset B^{\vee} \quad \text{and} \quad \wedge^2B \supset A^{\vee} \]
in their isobaric decompositions. In particular, $\wedge^2 A$ and $\wedge^2 B$ are not cuspidal and contain a cuspidal $\GL_4$-summand.
\vskip 5pt

We shall show:
\vskip 5pt

\begin{lemma}
Assume that $A$ and $B$ are distinct cuspidal representations of $\GL_4$ such that $\wedge^2 A \supset B^{\vee}$ and $\wedge^2 B \supset A^{\vee}$. Then at least one of $A$ or $B$ is invariant under twisting by a nontrivial quadratic character.
\end{lemma}

\begin{proof}
We will make use of the cuspidality criterion for the exterior square lifting of $\GL_4$ established by Asgari  and Raghuram \cite{AR}.  By \cite[Prop. 3.1, Prop. 3.2, Prop. 3.3 and Prop. 3.5]{AR}, we see that for $\wedge^2 A$ to contain a cuspidal $\GL_4$-summand in its isobaric decompostion, 
one of the following must hold:
 \vskip 5pt
 
 \begin{itemize}
 \item[(a)]  $A \simeq {\rm AI}_{K/k}(\tau)$ is an automorphic induction of a cuspidal representation $\tau$ of $\GL_2$ over a quadratic field extension $K$ of $k$;
 \item[(b)] $A \simeq  {\rm As}_{K/k}^+(\tau)$ is an Asai lift of a cuspidal representation $\tau$ of $\GL_2$ over a quadratic field extension $K$ of $k$;
 \item[(c)] $A$ is obtained as a functorial lift for a globally generic cuspidal representation $\sigma$ of $\GSp_4$. 
  \end{itemize}
 
If (a) holds, then $A$ is invariant under twisting by the nontrivial quadratic character $\omega_{K/k}$ and we are done.
We examine the remaining two cases in turn:
\vskip 5pt

\begin{itemize}
\item[(b)] Under (b), it follows by \cite[Proof of Prop. 3.2]{AR} that
\[  \wedge^2 A = {\rm AI}_{K/k}( {\rm ad}_2(\tau)) \cdot \omega_{\tau}|_{\A_k^{\times}}. \]
For this to contain a cuspidal $\GL_4$-summand, it is necessary that ${\rm ad}_2(\tau)$ has isobaric decomposition of type $1+2$, say ${\rm ad}_2(\tau) = \nu \boxplus \sigma$, in which case the cuspidal $\GL_4$-summand is 
\[  B^{\vee} = {\rm AI}_{K/k}(\sigma) \cdot \omega_{\tau}|_{\A_k^{\times}}. \]
This shows that $B$ is invariant under twisting by the nontrivial quadratic character $\omega_{K/k}$.
\vskip 5pt

\item[(c)] Under (c),  $A$ is the functorial  lift of a globally generic cuspidal representation $\sigma$ of $\GSp_4$.  
In fact, this functorial lift can be constructed as a similitude theta lift from $\GSp_4$ to $\GSO_6$ (which is a quotient of  $\GL_4 \times \GL_1$ by a central $\mathbb{G}_m$). 
  By \cite[Prop. 3.3]{AR}, 
\[ \wedge^2 A =  {\rm std}_*(\sigma|_{\Sp_4})\cdot \omega_{\sigma}   \boxplus \omega_{\sigma}, \]
where ${\rm std}_*(\sigma)$ is the automorphic representation of $\GL_5$ obtained as a functorial lift of $\sigma$ via the morphism of dual groups
\[ {\rm std}:  \GSp_4^{\vee} = \GSp_4(\C) \longrightarrow \PGSp_4(\C) \simeq \SO_5(\C) \hookrightarrow \GL_5(\C). \]
The functorial lifting  for the first morphism is given simply by the restriction of automorphic forms from $\GSp_4$ to $\Sp_4$, whereas the second was shown by \cite{CKPSS, A}. 
\vskip 5pt

For $\wedge^2A$ to contain a cuspidal $\GL_4$-summand,  it is necessary that ${\rm std}_*(\sigma)$ has isobaric decomposition of type $1 +4$. Since $B^{\vee} \subset \wedge^2 A$, we must have 
${\rm std}_*(\sigma) = \nu \boxplus B^{\vee} \cdot\omega_{\sigma}^{-1}$ for some automorphic quadratic character $\nu$.  
Now this implies that $L^S(s, \sigma|_{\Sp_4}, {\rm std} \times \nu)$ has a pole at $s=1$, which implies that $\sigma$ has a nonzero theta lift to a quasi-split group $\GSO_4$ associated to $\nu$. 
 Since $\sigma$ has cuspidal theta lift to the split $\GSO_6$, it does not lift to the split $\GSO_4$ under the theta correspondence. Hence, $\nu$ must be nontrivial and $\sigma$ has nonzero theta lift to a quasi-split but non-split $\GSO_4$ associated to  $\nu$. This implies that $A$ is obtained as an automorphic induction from the quadratic field determined by $\nu$ and hence is invariant under twisting by $\nu$. 



\end{itemize}
\vskip 5pt

We have thus completed the proof of the lemma.

\end{proof}
\vskip 5pt

With the lemma in hand, we can now continue with the proof of  Theorem \ref{T:cuspidality}. By the lemma, we may assume without loss of generality that $A \otimes \omega_{K/k} \simeq A$ for some quadratic field extension $K/k$,
so that $A \simeq {\rm AI}_{K/k}(\tau)$ for some cuspidal representation $\tau$ of $\GL_2$ over $K$.  
Consider now the base change $\pi_K$ of $\pi$ to $K$, noting that $\pi_K$ is still cuspidal.  Then
\[  {\rm Ad}(\pi_K) = {\rm Ad}(\pi)_K  = A_K  \boxplus B_K = \tau \boxplus  \tau^c \boxplus B_K \]
where $c$ is the nontrivial element in ${\rm Gal}(K/k)$ and the base change $B_K$ has isobaric decomposition of type $2+2$ or $4$. 
\vskip 5pt

Now we may appeal to Theorem \ref{T:AI2}.  This says that:
\vskip 5pt
\begin{itemize}
\item $\tau = \rho_E$ for some non-Galois cubic extension $E/K$, whose Galois closure $\tilde{E}$ is an $S_3$-extension of $K$ with discriminant quadratic field $L/K$, so that 
\[  A_K = \rho_E \oplus \rho_{E^c}. \] 

\item the bi-cyclic-cubic field extension $\tilde{E} \cdot \tilde{E}^c$ over $L$ contains two other cubic extensions $\tilde{E}_1$ and $\tilde{E}_2 = \tilde{E}_1^c$, which are $S_3$-extensions over $K$ with discriminant field extension $L/K$, so that
\[  B_K = \rho_{E_1} \boxplus \rho_{E_2}, \] 
where $E_1$ and $E_2$ are cubic subfields of $\tilde{E}_1/K$ and $\tilde{E}_2 /K$.  
\end{itemize}
In partiular, we have:
\[ A \simeq {\rm AI}_{K/k}(\rho_E)  \quad \text{and} \quad B \simeq {\rm AI}_{K/k}(\rho_{E_1}). \]
Since $\rho_E$ and $\rho_{E_1}$ are self-dual representations with central character $\omega_{L/K}$, we deduce that $A$ and $B$ are self-dual and 
\[ \omega_A = \omega_B = \omega_{L/K}|_{\A^{\times}}. \]
On the other hand,  with $A$ and $B$ as above, one has \cite[proofs of Prop. 3.2 and Prop. 3.5]{AR}
\[ \wedge^2 A = {\rm As}_{K/k}^+(\rho_E) \cdot \omega_{K/k} \boxplus {\rm AI}_{K/k}(\omega_{L/K}) \quad \text{and} \quad 
\wedge^2 B = {\rm As}^+_{K/k}(\rho_{E_1}) \cdot \omega_{K/k} \boxplus {\rm AI}_{K/k}(\omega_{L/K}), \]
 so that
 \[  A \simeq {\rm As}^+_{K/k}(\rho_{E_1})  \cdot \omega_{K/k} \quad \text{and} \quad 
 B \simeq {\rm As}^+_{K/k}(\rho_E)  \cdot \omega_{K/k}. \]
 From these alternate expressions of $A$ and $B$, one deduces by \cite{Ra2}  that 
 \[ \omega_A = \omega_{K/k}. \]
 In particular, since $\omega_A$ and $\omega_B$ are nontrivial, $A$ and $B$ are necessarily of orthogonal type. 
 Moreover, we also conclude that $\omega_{L/K}|_{\A^{\times}} = \omega_{K/k}$, which implies  that $L/k$ is a cyclic degree 4 extension. 
 \end{proof}
\vskip 5pt

Here is an interesting corollary, which is an application of  our results on the possible isobaric decompositions for ${\rm Ad}(\pi)$ to the Rankin-Selberg lifting from $\GL_2 \times \GL_3$ to $\GL_6$.
\vskip 5pt

\begin{cor}
Suppose that
\begin{itemize}
\item  $\pi$ is  a unitary cuspidal representation of $\GL_3$ which is neither an automorphic induction from a cubic field extension nor a twist of a self-dual cuspidal representation, and  ${\rm Ad}(\pi)$ exists on $\GL_8$ (for example when $\pi$ has a discrete series local component);
\item $\tau_1$ and $ \tau_2$ are cuspidal representations of $\GL_2$ with the same central character. 
\end{itemize}
Then
\[  \pi \boxtimes \tau_1 \simeq \pi \boxtimes \tau_2 \Longrightarrow  \tau_1 \simeq \tau_2. \]
\end{cor}
\vskip 5pt

\begin{proof}
If  $\pi \boxtimes \tau_1 \simeq \pi \boxtimes \tau_2$, then the partial L-function
\[  L^S(s,  (\pi \boxtimes \tau_1) \times (\pi \boxtimes \tau_2)^{\vee}) = L^S(s, \pi \times \pi^{\vee} \times \tau_1 \times \tau_2^{\vee}) \]
has a pole at $s=1$.  But this L-function is equal to
\[  L^S({\rm Ad}(\pi) \times (\tau_1 \boxtimes \tau_2^{\vee})) \cdot  L^S(s, \tau_1 \times \tau_2^{\vee}). \]
Under the hypothesis on $\pi$, it follows by Theorem \ref{T:cuspidality} that ${\rm Ad}(\pi)$ is either cuspidal or an isobaric sum of two self-dual cuspidal representations of $\GL_4$ with nontrivial central character.
This implies that the L-function $L^S({\rm Ad}(\pi) \times (\tau_1 \boxtimes \tau_2^{\vee}))$ is holomorphic at $s=1$, since the self-dual representation $\tau_1 \boxtimes \tau_2$ has trivial central character. 
Hence, the L-function $L^S(s, \tau_1 \times \tau_2^{\vee})$ must have a pole at $s=1$, which implies that $\tau_1 \simeq \tau_2$.
\end{proof}
\vskip 5pt

\subsection{\bf Summary}
The following proposition summarizes an interesting property  which should play an important role in the characterization of the image of the adjoint lifting.
\vskip 5pt

\begin{cor}
Assume that $\pi$ is a cuspidal automorphic representation of $\GL_3$ which is not invariant under twisting by a nontrivial cubic automorphic character.
If $\pi$ is in the context of Theorem \ref{T:cuspidality}, assume that ${\rm Ad}(\pi)$ is constructed via the sequence of lifting (\ref{E:seq}). Then
${\rm Ad}(\pi)$ is a discrete orthogonal L-parameter.  Moreover, 
\[ {\rm ord}_{s=1} L^S( s, {\rm Ad}(\pi), \wedge^3) = {\rm ord}_{s=1} L^S(s, {\rm Ad}(\pi), {\rm Sym}^2) < 0,  \]
so that the order of poles at $s=1$ of $L^S(s, {\rm Ad}(\pi), \wedge^3)$  (which is an L-function of Langlands-Shahidi type) is equal to the number of summands in the isobaric decomposition of ${\rm Ad}(\pi)$.
\end{cor}

One may wonder if the above property of the exterior cube L-function characterizes those discrete orthogonal parameter which arises as the adjoint lifting from $\GL_3$. As the exterior cube L-function occurs as a Langlands-Shahidi L-function in the group $E_8$, it is conceivable that one may establish this expectation by an automorphic decent construction from $\GL_8$ to $\GL_3$.
\vskip 10pt

\section{\bf Tetrahedral Representations} \label{S:artin}
In Theorems \ref{T:AI1}, \ref{T:AI2}, \ref{T:selfdual} and \ref{T:cuspidality},  we have described the possible shapes of the isobaric decomposition of ${\rm Ad}(\pi)$. One may ask if all these possibilities are actually attained. While it is not so surprising that the more nondegenerate cases actually happen, it may be less clear for the most degenerate cases in Theorem \ref{T:AI1}(b), Theorem \ref{T:AI2}(c) and Theorem \ref{T:cuspidality}(ii). In this final section, we will give a general construction of  cuspidal representations $\pi$
such that ${\rm Ad}(\pi)$ has the form in these degenerate cases. This construction was first studied by Blasius \cite{B} and  Lapid \cite[\S 5]{L2} in their work on multiplicities and rigidity of cuspidal $L$-packets of $\SL_n$.
\vskip 5pt

\subsection{\bf Jacobi group}
We shall consider the two dimensional symplectic vector space $W$ over the finite field $\Z/3\Z$. Choosing a Witt basis for $W$, we may identify $W$ with $\Z/3\Z \times \Z/3\Z$. Let $\Sp(W)$ 
denote the finite symplectic group attached to $W$, which is isomorphic to $\SL_2(\Z/3\Z)$. Note that 
\[ \SL_2(\Z/3\Z) \simeq  Q_8 \rtimes \Z/3\Z \]
where $Q_8$ denotes the finite quaternion group (with $8$ elements). 
In particular, $|\SL_2(\Z/3\Z)| = 24$ and $Q_8$ is its unique $2$-Sylow subgroup. Indeed, $Q_8$ is also the commutator subgroup of $\SL_2(\Z/3\Z)$.
\vskip 5pt

One can also form the Heisenberg group
\[ H(W) = W \oplus Z, \quad \text{with center $Z = \Z/3\Z$.} \]
Via its natural action on $W$, $\Sp(W)$ acts on $H(W)$ as group automorphisms fixing the center $Z$ elementwise. Hence one has the {\it Jacobi group}
\[   J(W) = H(W) \rtimes \Sp(W). \]
For our purpose, we will consider the following sequence of subgroups of $J(W)$:
\begin{equation} \label{E:chain}
  H(W) \subset H(W)  \cdot Z_{\Sp(W)}  \subset H(W) \cdot   C_4 \subset  H(W) \cdot Q_8 \subset J(W), \end{equation}
where $Z_{\Sp(W)} \simeq  \mu_2$ is the center of $\Sp(W)$ and $C_4$ is the cyclic group of order $4$ generated by the Weyl element
\[ w_0 = \left( \begin{array}{cc}
0 & 1 \\
-1 & 0 \end{array} \right). \]

\vskip 5pt

\subsection{\bf Tetrahedral representations} \label{SS:tetra}
By the Stone-von-Neumamn theorem, for a given nontrivial central character $\chi: \Z/3\Z \rightarrow \C^{\times}$, there is a unique (up to isomorphism) irreducible representation $\omega_{\chi}$ of $H(W)$. 
This irreducible representation is of dimension $3$ and can be constructed as follows. Let $L \subset W$ be any line (i.e. one-dimensional subspace) in $W$. Then $H(L) := L \oplus Z \subset H(W)$ is an abelian subgroup and we may extend $\chi$ from $Z$ to $H(L)$
trivially on $L$. Denoting this extended character by $\chi_L$, one has 
\[ \omega_{\chi} = {\rm Ind}_{H(L)}^{H(W)} \chi_L. \]
\vskip 5pt

By the theory of Weil representations, this representation can be extended to 
to $\Sp(W)$.  Since $\Sp(W) = \SL_2(\Z/3\Z)$ is not perfect, the extension is of course not unique, as we may twist by cubic characters of $\Sp(W)$. However, the restriction of these extensions to the commutator subgroup $Q_8$   are isomorphic. We will denote this representation of $H(W) \rtimes Q_8$ by $\omega_{\chi}$ as well.
Following Lapid \cite[\S 5]{L2}, we will call $\omega_{\chi}$ or its restriction to any subgroup of $H(W) \rtimes Q_8$ containing $H(W)$  a {\it generalized tetrahedral  representation}, though it would also be appropriate to call it a {\it Heisenberg-Weil representation}. 

\vskip 5pt

\subsection{\bf Setting of Theorem \ref{T:AI1}}
Let $M/k$ be a Galois extension with Galois group isomorphic to the Heisenberg group $H(W)$ and consider its irreducible 3-dimensional representation $\omega_{\chi}$ constructed above. 
The fixed field of the center $Z$ of $H(W)$ is thus a bi-cyclic-cubic extension $E_Z$ of $k$. 
For any line $L \subset W$, the fixed field of the abelian normal subgroup $H(L)$ is a Galois cubic extension $E_L$ of $k$. Since 
\[  \omega_{\chi} \simeq {\rm Ind}_{H(L)}^{H(W)} \chi_L, \]
we see that $\omega_{\chi}$ gives rise to the cuspidal representation 
\[ \pi = {\rm AI}_{E_L/k}(\chi_L). \]
In other words, $\pi$ is an automorphic induction from $E_L$. Since there are 4 possible lines $L \subset W$ (as $|\mathbb{P}^1(\mathbb{F}_3)| =4$), we see that $\pi$ is an automorphic induction from the 4 cubic subfields in the bi-cyclic-cubic extension $E_Z$ of $k$, so that we are in situation (b) of Theorem \ref{T:AI1}. In particular, ${\rm Ad}(\pi)$ is the sum of the 8 cubic characters $\omega_{E_L/k}^{\pm 1}$.

\vskip 5pt

\subsection{\bf Setting of Theorem \ref{T:AI2}}
Now suppose that $M/k$ is a Galois extension with Galois group $H(W) \rtimes \mu_2$. Let $K$ be the fixed field of $H(W)$, so that $K/k$ is a quadratic extension. If $E_Z$ is the fixed field of the center $Z$ of $H(W)$, then the action of ${\rm Gal}(K/k)$ on ${\rm Gal}(E_Z/K) =W$ is the action of $\mu_2 \subset \Sp(W)$ on $W$. In other words, it acts as $-1$ on ${\rm Gal}(E_Z/K) = W$. This shows that for each line $L \subset W$, its fixed field $E_L$ is a cubic extension of $K$, which is Galois over $k$ with ${\rm Gal}(E_L/k) \simeq S_3$. 
\vskip 5pt

 Now the irreducible 3-dimensional representation $\omega_{\chi}$ of $H(W)$ 
gives rise to a cuspidal representation $\pi_K$ of $\GL_3(\A_K)$ as described above. Hence, $\pi_K$ is an automorphic induction from any of the 4 cubic subfields $E_L$ of  $M/K$. 
Since $\omega_{\chi}$  extends to $H(W) \rtimes \mu_2$, the cuspidal representation $\pi_K$ is invariant under the action of ${\rm Gal}(K/k)$; in particular, it follows from \cite{AC} that $\pi_K$ is the base change  of a cuspidal representation  $\pi$ of $\GL_3$ over $k$.  Now the cuspidal representation $\pi$ is not unique, as we can twist it by the quadratic Hecke character $\omega_{K/k}$.
But there is a unique such $\pi$ whose central character $\omega_{\pi}$ corresponds to $\det \omega_{\chi}$ under global class field theory. 
By  \cite[\S 5, Prop. 2]{L2}, this $\pi$ then corresponds to $\omega_{\chi}$ under the purported global Langlands correspondence for $\GL_3$. 
Moreover, since ${\rm Ad}(\pi_K)$ is the isobaric sum of the 8 cubic characters $\omega_{E_L/K}^{\pm 1}$, we deduce that ${\rm Ad}(\pi)$ is the isobaric sum of 4 cuspidal representations of $\GL_2$. In particular, we are in the situation of Theorem \ref{T:AI2}(c).  
\vskip 5pt

\vskip 5pt

\subsection{\bf Setting of Theorem \ref{T:cuspidality}}
Finally, suppose that $M/k$ is a Galois extension with Galois group $H(W) \rtimes C_4$ and set
\[  \begin{cases}
  L  = \text{  fixed field of $H(W)$;} \\
  K = \text{ fixed field of $H(W) \rtimes \mu_2$;} \end{cases} \]
  Then $L/k$ is a cyclic degree $4$ extension with Galois group $C_4$. 
  By the previous subsection,  the irreducible representation $\omega_{\chi}$ of $H(W) \rtimes \mu_2$ gives rise to a cuspidal representation $\pi_K$ of $\GL_3(\A_K)$ such that ${\rm Ad}(\pi_K)$ is the isobaric sum of 4 cuspidal representations of $\GL_2(\A_K)$, corresponding to the 4  lines in $W$.
 \vskip 5pt

  Since $\omega_{\chi}$ extends to $H(W) \rtimes C_4$, we see that $\pi_K$ is invariant under ${\rm Gal}(K/F)$, and hence is the base change of a cuspidal representation $\pi$ of $\GL_3$ over $k$ by \cite{AC}. As in the previous subsection, there is a unique such $\pi$ if we require that its central character corresponds to $\det \omega_{\chi}$ under global class field theory.  
By  \cite[\S 5, Prop. 2]{L2} again, one knows that this unique $\pi$ corresponds to $\omega_{\chi}$ under the purported global Langlands correspondence.
  \vskip 5pt
  
   Now the action of ${\rm Gal}(L/k) = C_4$ on ${\rm Gal}(E_Z/L) = W$ (where $E_Z$ is the fixed field of the center $Z \subset H(W)$) is generated by the action of the Weyl element 
 \[ w_0 = \left( \begin{array}{cc}
0 & 1 \\
-1 & 0 \end{array} \right). \]
As the element $w_0^2$ acts as $-1$ on $W$, we see that the action of $C_4$ has two orbits on  the set of lines in $W$.  This shows that ${\rm Gal}(K/k)$ has two orbits on the 4 isobaric summands in ${\rm Ad}(\pi_K)$, and hence
 ${\rm Ad}(\pi)$ is the isobaric sum of two cuspidal representations of $\GL_4$ over $k$. In particular, we are in the situation of Theorem \ref{T:cuspidality}(ii).
\vskip 5pt

Not surprisingly, if we consider a Galois extension with Galois group $H(W) \rtimes Q_8$, a similar argument as above  will lead to a cuspidal $\pi$ such that ${\rm Ad}(\pi)$ is cuspidal on $\GL_8$ (i.e. the situation of Theorem \ref{T:cuspidality}(i))  but the base change of ${\rm Ad}(\pi)$ to an appropriate   quadratic extension decomposes as an isobaric sum of two cuspidal $\GL_4$-summands. This gives an instance  of Theorem \ref{T:cuspidality}(ii), though one might expect the general situation of Theorem \ref{T:cuspidality}(ii) will be such that ${\rm Ad}(\pi)$ is not invariant under twisting by any quadratic Hecke character.
 
  \vskip 5pt
  
   \subsection{\bf Strong Artin conjecture for tetrahedral representations} In our discussion above, we have alluded to \cite[\S 5, Prop. 2]{L2} on more than one occasion to deduce that the generalized tetrahedral representation $\omega_{\chi}$ of one of the groups $H(W) \rtimes S$ in (\ref{E:chain}) corresponds to a certain cuspidal representation $\pi$ under the purported global Langlands correspondence. 
   In fact, \cite[\S 5, Prop. 2]{L2} was shown under the hypothesis that the adjoint lifting for $\GL_3$ exists. Hence it is now unconditional given our Theorem \ref{T:main lifting} and its generalization as discussed in \S \ref{SS:general}. We record this instance of the strong Artin conjecture:
   \vskip 5pt
   
   \begin{thm} \label{T:artin2}
   Let $M/k$ be a Galois representation such that ${\rm Gal}(M/k) = H(W) \rtimes S$, with $S = Z_{\Sp(W)}$, $C_4$ or $Q_8$ (as in (\ref{E:chain})) Let 
   \[ \rho: {\rm Gal}(M/k) \hookrightarrow \GL_3(\C) \]
   be a generalized tetrahedral representation constructed in \S \ref{SS:tetra}.  Then  there is a cuspidal automorphic representation $\pi$ of $\GL_3$ over $k$ such that 
   one has the equality (outside a sufficiently large set $S$ of places of $k$)
   \[  c(\pi_v)  =  \rho({\rm Frob}_v) \]
   of the Hecke-Satake parameters  of $\pi$ and the Frobenius conjugacy  classes of $\rho$.
      \end{thm}
\vskip 5pt

Recall that the Heisenberg-Weil construction actually produces a 3-dimensional representation $\omega_{\chi}$ of $H(W) \rtimes \Sp(W)$ (up to twisting by a cubic character of $\Sp(W) = \SL_2(\mathbb{F}_3)$). A Galois representation of the form
 \[ \begin{CD}
  \rho: {\rm Gal}(M/k) \simeq H(W) \rtimes \Sp(W) @>\omega_{\chi}>> \GL_3(\C) \end{CD}  \]
is a {\it generalized octahedral representation}. In a forthcoming work of JW Wang, the strong Artin conjecture for  such an octahedral $\rho$ will be addressed.

\vskip 10pt

\section{\bf Application to Ramanujan Bounds} \label{S:Ramanujan}
In this final section,  we give an application of  the existence of the adjoint lifting to the question of bounding the Hecke eigenvalues of cuspidal representations of $\GL_3$. 
\vskip 5pt

\subsection{\bf Ramanujan-Petersson conjecture}
Let $\pi$ be a cuspidal representation of $\GL_n$ over $k$ with unitary central character. A a place $v$ of $k$ where $\pi_v$ is unramified,  the Satake parameter $c(\pi_v)$    is a semisimple conjugacy class in $\GL_n(\C)$. For any eigenvalue $\lambda$ of $c(\pi_v)$, we may write 
\[  |\lambda| = q_v^a \]
where $q_v$ is the cardinality of the residue field of $k_v$ and $a \in \R$.  The Ramanujan-Petersson conjecture asserts that $|\lambda| =1$ so that $a=0$. 

\vskip 5pt

\subsection{\bf Some prior results}
A general bound towards the Ramanujan-Petersson conjecture  was established by Luo-Rudnick-Sarnak \cite{LRS}. They showed that
\[  |a| \leq \frac{1}{2}  - \frac{1}{n^2+1}, \]
so that for the case of $\GL_3$, this reads:
\[  |a| \leq \frac{1}{2} - \frac{1}{10} = \frac{2}{5}. \]
There have certainly been improvements to this general bound when $n$ is small. In particular, for the case of $\GL_3$, the paper \cite{BB} of Blomer-Brumley showed that 
\[  |a| \leq  \frac{5}{14}, \]
thus improving upon the Luo-Rudnick-Sarnak bound. 
\vskip 5pt

\subsection{\bf Application of adjoint lifting}
As an application of  the existence of the adjoint lifting ${\rm Ad}(\pi)$ (as supplied by Theorem \ref{T:main lifting} and the discussion in \S \ref{SS:general}), we shall show:
\vskip 5pt

\begin{thm} \label{T:Rama2}
Let $\pi$ be a cuspidal representation of $\GL_3$ with unitary central character (and a discrete series local component) over a number field $k$. At each place $v$ where $\pi_v$ is unramified, let $\lambda$ be an eigenvalue of its Satake parameter $c(\pi_v) \in \GL_3(\C)$, and write $|\lambda| = q_v^a$ for some $a \in \R$. Then
\[  |a| \leq \frac{1}{4} - \frac{1}{130}. \] 
\end{thm}
\vskip 5pt

\begin{proof}
By Theorem \ref{T:main lifting} and the discussion in \S\ref{SS:general}, ${\rm Ad}(\pi)$ exists as an automorphic representation of $\GL_8$. 
Moreover, by Theorem \ref{T:intro3}, ${\rm Ad}(\pi)$ is an isobaric sum of cuspidal representations (of $\GL_r$, with $r \leq 8$)  with unitary central characters. 

 \vskip 5pt
 
 Let us fix the unramified place $v$ and let $\lambda_i$  (for $1 \leq i \leq 3)$ be the three eigenvalues of the Satake parameter $c(\pi_v)$. Writing $|\lambda_i| = p^{a_i}$, the condition of unitary central character implies that
 \[ a_1 + a_2 + a_3 =0. \]
 In addition, as pointed out to us by V. Blomer, since $\pi_v$ is unitary, one has $\pi_v^{\vee} \simeq \overline{\pi_v}$, so that
 \[   \{  \lambda_i^{-1} \} = \{ \overline{\lambda_i} \} \quad \text{as multi-sets.} \]
Hence,
\[  \{ -a_i \} = \{ a_i \} \quad \text{as multi-sets.} \]
A moment's reflection  shows that  one has:
\[  \{ a_1, a_2, a_3 \} = \{ a, 0, -a \} \]
for some $a \in \R$.  
 \vskip 5pt
 
 On the other hand, the Satake parameter of ${\rm Ad}(\pi)_v$  has eigenvalues $\lambda_i/ \lambda_j$, with $|\lambda_i/\lambda_j| = q_v^{a_i-a_j}$.
 Applying the Luo-Rudnick-Sarnak bound to each isobaric summand of ${\rm Ad}(\pi)$, one has
 \[   |a_i -a_j| \leq \frac{1}{2} - \frac{1}{65}. \]
 This gives in particular
 \[  2 |a| \leq \frac{1}{2} - \frac{1}{65} \quad \text{and hence} \quad |a| \leq \frac{1}{4} - \frac{1}{130}. \]
 The theorem is proved.
   \vskip 5pt
 
\end{proof}

\vskip 5pt





\end{document}